\documentclass{amsart}
\usepackage[normalem]{ulem}
\usepackage{amsmath,amsthm,amsfonts,amscd,amssymb, tikz-cd}
\usepackage{verbatim}       
\usepackage{mathtools}
\usepackage{epsfig}         
\usepackage{url} 
\usepackage{epic,eepic,latexsym}
\usepackage{graphicx}
\usepackage{color}
\usepackage{lcd}
\usepackage{mathrsfs}
\usepackage[centering,text={15.5cm,22cm},
        marginparwidth=20mm,dvips]{geometry}
\usepackage[hidelinks]{hyperref}
\usepackage{lscape}
\usepackage{tabularx}
\usepackage{marginnote}
\usepackage[all]{xy}
\usepackage{enumerate}
\usepackage{stmaryrd}
\usepackage{tipa}
\usepackage{upgreek}
\usepackage{tikz-cd}
\usepackage{adjustbox}
\usepackage[capitalize]{cleveref}
\usepackage{pdflscape}
\usepackage{tikz}
\usetikzlibrary{positioning}

\let\m\mathbb
\let\c\mathcal

\let\r\textnormal
\let\i\textit

\theoremstyle{plain}
 \newtheorem{thm}{Theorem}[section]

 \newtheorem{lem}[thm]{Lemma}
  
  \newtheorem{prop}[thm]{Proposition}
  
   \newtheorem{cor}[thm]{Corollary}

\theoremstyle{definition}

 \newtheorem{eg}[thm]{Example}
 
  \newtheorem{rmk}[thm]{Remark}
 
\theoremstyle{remark}

\title{On the Cohomological Hall Algebra of a Character Variety}
\author{Vivek Mistry}
\address{School of Mathematics, University of Edinburgh, Edinburgh, United Kingdom}
\email{s1829507@ed.ac.uk}

\begin{document}

\maketitle

\begin{abstract}
    A multiplication on the 2D cohomological Hall algebra (CoHA) of the variety of commuting matrices was described by Schiffman and Vasserot. This construction can be generalised to other varieties that exist as the zero-locus of a function on a smooth ambient variety. On the 2D CoHA of the character variety of the fundamental group of a genus $g$ Riemann surface, we compare the multiplication induced by the standard presentation and that of a brane tiling presentation.
\end{abstract}

\tableofcontents

\section{Introduction}

In \cite{sv} Schiffman and Vasserot defined a multiplication on the vector space $\bigoplus_{n \in \m N} \r{H}_{\r{BM}}^{G_n}(C_n, \m C)$ of $G_n= \r{GL}_n(\m C)$-equivariant Borel-Moore homology (or equivalently the dual of the the $G_n$-equivariant compactly supported cohomology) of the commuting variety $C_n = \big\{ (A,B) \in \r{Mat}_{n \times n}(\m C)^2 \,\, \big| \,\, [A,B]=0 \big\}$, giving rise to a so-called 2D cohomological Hall algebra (CoHA). We have an algebra homomorphism $\widehat{\mu}: \m C[x] \rightarrow \m C \langle a,b \rangle$ given by
$$x \longmapsto [a,b]= ab - ba$$
which induces the map of varieties $\mu_n: \r{Mat}_{n \times n}(\m C)^2 \rightarrow \r{Mat}_{n \times n}(\m C)$ that sends
$$(A,B) \longmapsto [A,B] =  AB - BA.$$
The commuting variety can be written as the zero-locus $C_n = \mu_n^{-1}(0)$ of this function on the smooth ambient variety $\r{Mat}_n(\m C)^2$ induced from a homomorphism of smooth finitely generated $\m C$-algebras. This construction can be generalised to the (equivariant) compactly supported cohomology of any such variety that can be written in a similar way, as we will explain in Section 3.1.

Specifically we are interested in the 2D CoHA of the character variety $\r{Rep}_n(\m C[\pi_1(\Sigma_g)])$ of the fundamental group of a Riemann surface $\Sigma_g$ of genus $g$. This has the natural or \emph{standard presentation} as the zero-locus of the $\r{GL}_n(\m C)$-equivariant function $\lambda_n : \r{GL}_n^{2g}(\m C) \rightarrow \r{Mat}_{n \times n}(\m C)$ given by
$$(A_i,B_i)_{i=1}^g \longmapsto \prod_{i=1}^g A_iB_iA_i^{-1}B_i^{-1}-\r{Id}_n$$
which is induced from the algebra homomorphism $\widehat{\lambda}: \m C[x] \rightarrow \m C \langle a_1^{\pm 1}, b_1^{\pm 1}, \ldots, a_g^{\pm 1}, b_g^{\pm 1} \rangle$ that sends
$$x \longmapsto \prod_{i=1}^g a_i b_i a_i^{-1} b_i^{-1} - 1.$$
This gives rise to a particular 2D CoHA structure on the vector space $\bigoplus_{n \in \m N} \r{H}_c\big(\r{Rep}_n(\m C[\pi_1(\Sigma_g)]), \m Q\big)^\vee$.

However, there is an alternative way to express the character variety as the zero-locus of a function induced from an algebra homomorphism, via brane tilings and Jacobi algebras. As explored in \cite{dav4} and \cite{dav1}, taking a brane tiling $\Delta$ of a Riemann surface $\Sigma_g$ we can obtain a quiver and potential $(Q_\Delta, W_\Delta)$ such that there exists a cut $E$ for $W_\Delta$ and an isomorphism of algebras
$$\r{Jac}(\widetilde{Q}_\Delta, W_\Delta, E) \cong \r{Mat}_r(\m C[\pi_1(\Sigma_g)])$$
for the localised 2D Jacobi algebra of $(Q_\Delta, W_\Delta, E)$, where $r= |Q_{\Delta, 0}|$ (see [\cite{dav1}, Proposition 5.4]). Hence we have isomorphisms
\begin{align*}
    \r{Rep}_{nr}(\r{Jac}(\widetilde{Q}_\Delta, W_\Delta, E)) &\cong \r{Rep}_{nr}(\r{Mat}_r(\m C[\pi_1(\Sigma_g)]))\\
    &\cong \r{Rep}_n(\m C[\pi_1(\Sigma_g)]).
\end{align*}
From standard results of Jacobi algebras we obtain an algebra homomorphism\\
${\m C \langle x_e \,:\, e \in E \rangle \rightarrow \m C (\widetilde{Q}_\Delta \setminus E)}$ given by
$$x_e \longmapsto \frac{\partial W_\Delta}{\partial e}$$
which induces a $\r{GL}_n^r(\m C)$-equivariant function
$$\r{Rep}_{nr}(\widetilde{Q}_\Delta \setminus E) \rightarrow \r{Mat}_{nr \times nr}(\m C)^{|E|}$$
whose zero-locus is $\r{Rep}_{nr}(\r{Jac}(\widetilde{Q}_\Delta, W_\Delta, E))$. This induces another 2D CoHA structure on the cohomology $\bigoplus_{n \in \m N} \r{H}_c\big(\r{Rep}_n(\m C[\pi_1(\Sigma_g)]), \m Q\big)^\vee$ and therefore the natural question is to ask if in fact these two 2D CoHA structures are the same?

In this paper we shall prove an affirmative answer to this question

\begin{thm}\label{thm1.1}
The 2D CoHA
$$\bigoplus_{n \in \m N} \r{H}_c\big(\r{Rep}_n(\m C[\pi_1(\Sigma_g)]), \m Q\big)^\vee$$
is isomorphic as an algebra to the 2D CoHA
$$\bigoplus_{n \in \m N} \r{H}_c\big(\r{Rep}_{nr}(\r{Jac}(\widetilde{Q}_\Delta, W_\Delta, E)), \m Q\big)^\vee$$
under the isomorphism induced by $\r{Jac}(\widetilde{Q}_\Delta, W_\Delta, E) \cong \r{Mat}_r(\m C[\pi_1(\Sigma_g)])$ from [\cite{dav1}, Proposition 5.4].
\end{thm}

The utility of this result is that it allows one to use powerful theorems already developed for CoHAs of quivers and Jacobi algebras, such as dimensional reduction from \cite{dav3} and the PBW isomorphism from \cite{dm}, in the character variety setting allowing for new avenues of study of these cohomological objects.

\subsection{Notation}
We fix some notation of common spaces and objects used throughout this paper. We shall always work over the field of complex numbers $\m C$. We denote by
\begin{align*}
    & \r{Mat}_{m \times n} := \r{Mat}_{m \times n}(\m C) && \r{the space of all $m \times n$-matrices}\\
    & \r{Mat}_n := \r{Mat}_{n \times n}(\m C) && \r{the space of all $n \times n$ square matrices}\\
    & \r{Mat}_{m,n} \subset \r{Mat}_{m+n} && \r{the space of all $(m+n)$ square matrices}\\
    & && \,\r{whose lower-left $n \times m$-block is 0}\\
    & \r{GL}_n := \r{GL}_n(\m C) && \r{the space of all $n \times n$ invertible matrices}\\
    & \r{GL}_{m,n} \subset \r{GL}_{m+n} && \r{the space of all $(m+n)$ invertible matrices}\\
    & && \, \r{whose lower-left $n \times m$-block is 0}.
\end{align*}
Note that $\r{GL}_{m,n}$ naturally acts on $\r{Mat}_{m,n}$ by conjugation via the inclusions $\r{GL}_{m,n} \subset \r{GL}_{m+n}$ and $\r{Mat}_{m,n} \subset \r{Mat}_{m+n}$.

\subsection*{Acknowledgements}
I would like to thank my supervisor Ben Davison for many helpful discussions and ideas and his support while writing this paper. This research was supported by the Royal Society studentship RGF\textbackslash R1\textbackslash 180093.

\section{Quivers, Jacobi algebras, and brane tilings}

\subsection{2D Jacobi algebras}
Let $Q=(Q_0, Q_1)$ be a quiver with potential $W \in \m CQ/[\m CQ, \m CQ]$. For $n=(n_i) \in \m N^{|Q_0|}$ let
$$M_n(Q) = \bigoplus_{a: i \rightarrow j \in Q_1} \r{Hom}_{\m C}(\m C^{n_i}, \m C^{n_j})$$
denote the space of $n$-dimensional representations of the path algebra $\m C Q$ and let
$$\r{Rep}_n(\m CQ) \cong M_n(Q)/G_n$$
denote the stack of $n$-dimensional representations of $\m CQ$, where $G_n = \prod_{i \in Q_0} \r{GL}_{n_i}$ acts on $M_n(Q)$ via simultaneous conjugation. For $a \in Q_1$ define $\partial W/\partial a \in \m CQ$ by taking all the terms in $W$ containing the arrow $a$, cyclically permuting $a$ to the front of each term and then deleting it. The \i{Jacobi algebra} of the $(Q,W)$ is the quotient algebra
$$\r{Jac}(Q,W) = \m CQ/I_W$$
where $I_W$ is the ideal $\big(\partial W/ \partial a \,\,\, | \,\, a \in Q_1 \big)$. Analogously we define $M_n(\r{Jac}(Q,W))$ as the space of $n$-dimensional representations of $\r{Jac}(Q,W)$ and $\r{Rep}_n(\r{Jac}(Q,W))$ as the stack of $n$-dimensional representations of $\r{Jac}(Q,W)$. Define $G_n$-equivariant maps of varieties 
$$\r{Tr}(W)_n: M_n(Q) \rightarrow \m C$$
by
$$(R_a)_{a \in Q_1} \longmapsto \r{Tr}(W(R_a))$$
and
$$(\partial W/\partial a)_{a \in Q_1}: M_n(Q) \rightarrow M_n(Q^{\r{op}})$$
by
$$(R_a)_{a \in Q_1} \longmapsto \left(\frac{\partial W}{\partial b}(R_a) \right)_{b \in Q_1}.$$
Then we have that
$$\r{Rep}_n(\r{Jac}(Q,W)) \cong \r{crit}((\r{Tr}(W)_n))/G_n$$
or alternatively
$$\r{Rep}_n(\r{Jac}(Q,W)) \cong (\partial W/\partial a)_{a \in Q_1}^{-1}(0)/G_n$$
where $0 \in M_n(Q^{\r{op}})$ is the representation that assigns the zero matrix to each arrow $a \in Q^{\r{op}}_1$.

A \i{cut} $E \subset Q_1$ for the potential $W$ is a choice of arrows such that $W$ is homogeneous of degree 1 with respect to the chosen arrows. Given the triple $(Q,W,E)$ let $\widehat{I}_{W,E}$ denote the ideal 
$$\big(\partial W/ \partial e \,,\, e \,\, \big| \,\, e \in E \big) \lhd \m CQ$$
and let $I_{W,E}$ denote the ideal
$$\big(\partial W/ \partial e \,\,\, \big| \,\, e \in E \big) \lhd \m C(Q \setminus E)$$
which makes sense since $\partial W/\partial e$ is a sum of paths in $Q \setminus E$ for any $e \in E$ due to the definition of a cut. We define the \i{2D Jacobi algebra} as the quotient algebra
\begin{align*}
    \r{Jac}(Q,W,E) &= \m CQ/\widehat{I}_{W,E}.
\end{align*}
We also write
$$\r{Jac}(Q \setminus E, W, E) = \m C(Q \setminus E)/I_{W,E}$$
then clearly
$$\r{Jac}(Q \setminus E, W, E) \cong \r{Jac}(Q,W,E)$$
as alternative notation we will use later on. Define the $G_n$-equivariant map
$$\partial W/\partial E: M_n(Q \setminus E) \rightarrow \prod_{e \in E} \r{Mat}_{n_{s(e)} \times n_{t(e)}}$$
by
$$(R_a)_{a \in Q_1 \setminus E} \longmapsto  \left(\frac{\partial W}{ \partial e}(R_a)\right)_{e \in E}$$
where $G_n$ acts on $\prod_{e \in E} \r{Mat}_{n_{s(e)} \times n_{t(e)}}$ via simultaneous conjugation. Then
$$\r{Rep}_n(\r{Jac}(Q,W,E)) \cong (\partial W/ \partial E)^{-1}(0)/G_n.$$

Let $\m C \widetilde{Q}$ denote the \i{localised path algebra of $Q$} in which we add formal inverses $a^{-1}$ to the algebra $\m C Q$ for each $a \in Q_1$ with the property that $a^{-1} a = e_{s(a)}$ and $a a^{-1} = e_{t(a)}$. We get corresponding notions for the localised Jacobi algebra denoted by $\r{Jac}(\widetilde{Q},W)$, and for the localised 2D Jacobi algebra denoted by $\r{Jac}(\widetilde{Q},W,E)$ for which we add inverses $a^{-1}$ for $a \in Q_1 \setminus E$.

\subsection{Brane tilings}
Let $\Sigma_g$ be a Riemann surface of genus $g$. A \i{brane tiling} $\Delta$ of $\Sigma_g$ is an embedding $\Gamma \hookrightarrow \Sigma_g$ of a bipartite graph $\Gamma$ such that each connected component of $\Sigma_g \setminus \Gamma$ is simply connected. We choose a partition of the vertex set of $\Gamma$ into two disjoint subsets of black and white vertices such that the edges in $\Gamma$ only go between a black vertex and a white vertex.

From a brane tiling we can obtain a quiver with potential $(Q_\Delta, W_\Delta)$. The underlying graph of $Q_\Delta$ is the dual graph to $\Gamma$ in $\Sigma_g$, and it is directed so that arrows in $Q_\Delta$ go clockwise around a white vertex and anticlockwise around a black vertex. For a vertex $v \in \Gamma$ let $c_v$ denote the minimal cycle in $Q_\Delta$ that encircles $v$, i.e. $c_v$ is the cycle consisting of all the arrows which are dual to the edges that come out of $v$. Then we take the potential to be
$$W_\Delta = \sum_{v \, \r{white}} c_v - \sum_{u \, \r{black}} c_u.$$
A \i{dimer} for a brane tiling $\Delta$ is selection of edges in the tiling such that every vertex is attached to exactly one edge in this selection. Note that dimers need not always exist, for example if the number of black and white vertices is not equal then we will not be able to find a dimer.

The idea behind using brane tilings in the context of fundamental groups of Riemann surfaces is as follows (from \cite{dav1}). If the arrow $a \in Q_{\Delta, 1}$ is dual to the edge between the white vertex $v$ and black vertex $u$ in $\Gamma$ then $a \cdot \partial W_\Delta/\partial a = c_v - c_u.$ This tells us that in $\r{Jac}(Q_\Delta, W_\Delta)$ we can identify paths that are homotopic via an edge in the brane tiling ie. the edge in $\Delta$ ``spans" the homotopy in $\Sigma_g$. This allows us to link the algebra over the fundamental group of the Riemann surface $\Sigma_g$ with a Jacobi algebra. However note that not all homotopic paths in $\m C Q_\Delta$ are identified in $\r{Jac}(Q_\Delta, W_\Delta)$, e.g. the minimal cycles $c_v$ are null-homotopic but are not equal to the constant paths in $\r{Jac}(Q_\Delta, W_\Delta)$.

\begin{eg}\label{eg:1.01}
For each genus $g$ we have the following brane tiling from [\cite{dav4} Section 7] containing four tiles. We have drawn this tiling for the genus 2 case in \Cref{pic:contil}. Informally one can think of it as being constructed by taking the surface and first cutting it in half down the middle through all the holes, leaving you with two cylinders with $g+1$ holes in each (including the two end holes). Then cut each of those cylinders in half, once again cutting through the holes, leaving you with two tiles per cylinder.
\begin{figure}[h]
    \centering
    \includegraphics[scale=0.5]{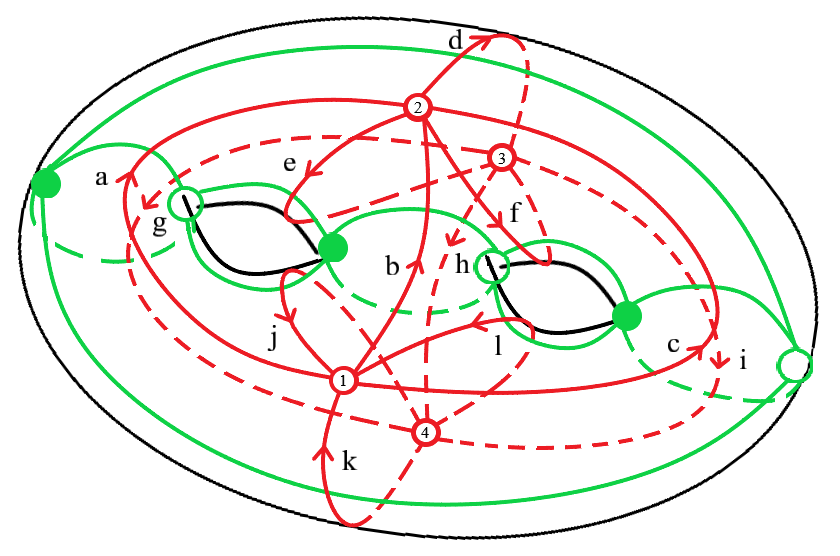}
    \caption{The tiling from \cite{dav4} for a genus 2 surface. The tiling $\Delta$ is in green and the dual quiver $Q_\Delta$ is in red.}
    \label{pic:contil}
\end{figure}
We get that $Q_\Delta$ is the quiver
\[
\begin{tikzcd}[row sep=7em, column sep=8em]
1 \arrow[r, shift left=0.5ex, bend left= 25,"a"] \arrow[r,"b"] \arrow[r, shift right=0.5ex, bend right=25,"c"]
& 2 \arrow[d, shift left=0.5ex, bend left= 25,"d"] \arrow[d,"e"] \arrow[d, shift right=0.5ex, bend right=25,"f"]\\
4 \arrow[u, shift left=0.5ex, bend left= 25,"j"] \arrow[u,"k"] \arrow[u, shift right=0.5ex, bend right=25,"l"] & 3 \arrow[l, shift left=0.5ex, bend left= 25,"g"] \arrow[l,"h"] \arrow[l, shift right=0.5ex, bend right=25,"i"]
\end{tikzcd}
\]
and the potential is
$$W_\Delta = jgea + lhfb + kidc - kgda - jheb - lifc.$$
Looking at some of the relations we get from this potential, we have for example
$$\frac{\partial W_\Delta}{\partial a} = jge-kgd \,, \quad \frac{\partial W_\Delta}{\partial f} = blh-cli \,, \quad \frac{\partial W_\Delta}{\partial j} = gea-heb$$
and we can see from \Cref{pic:contil} that indeed the path $jge$ is homotopic to $kgd$, $blh$ is homotopic to $cli$, and $gea$ is homotopic to $heb$. 
\end{eg}

To rectify the issue of homotopic paths not being identified in $\r{Jac}(Q_\Delta, W_\Delta)$ we use a cut. Since each term in $W_\Delta$ is the minimum cycle $c_v$ around a vertex $v \in \Delta$ we can see that a choice of a dimer for $\Delta$ corresponds to a choice of a cut for $(Q_\Delta, W_\Delta)$ whereby we just take the duals of the edges in the dimer to be the cut; and indeed this choice makes $W_\Delta$ homogeneous of degree 1 due to the fact that every vertex is contained in the dimer exactly once.

It turns out that issues only arise from the minimal cycles $c_v$ not being trivial paths in the Jacobi algebra and so by taking a selection of arrows such that $W_\Delta$ is homogeneous of degree 1 in those arrows, when considering the 2D Jacobi algebra we end up removing the minimal cycles. This entirely gets rid of the problem of having non-trivial paths in the Jacobi algebra which are not null-homotopic when viewed as paths in $\Sigma_g$. More precisely we have the following result from \cite{dav1}.

\begin{prop}[\cite{dav1} Proposition 5.4]\label{prop1.02}
Let $\Delta$ be a brane tiling of the Riemann surface $\Sigma_g$ with dual quiver $Q_\Delta$ and potential $W_\Delta$. Take a cut $E$ for $W_\Delta$. Then
$$\r{Jac}(\widetilde{Q}_\Delta, W_\Delta, E) \cong \r{Mat}_r(\m C[\pi_1(\Sigma_g)])$$
where $r= |Q_{\Delta, 0}|$.
\end{prop}

To tidy up the fact that we must take matrices over the fundamental group algebra in \Cref{prop1.02} consider a maximal tree $T \subset Q_\Delta \setminus E$. Let $Q$ denote the quiver with one vertex and arrow set indexed by elements in $Q_{\Delta,1} \setminus (E \cup T)$ (equivalently $Q$ is obtained by contracting the tree $T$ in the quiver $Q_\Delta \setminus E$). Then for all $i,j \in Q_{\Delta,0}$ let $t_{i,j}: i \rightarrow j$ be the unique path in the localised path algebra $\m C (\widetilde{Q_\Delta \setminus E})$ comprised solely from arrows in $T$ or their inverses. We can define an algebra homomorphism 
$$\sigma: \m C (\widetilde{Q_\Delta \setminus E}) \longrightarrow \r{Mat}_r(\m C \widetilde{Q})$$
by
\begin{align*}
    \r{const}_i &\longmapsto E_{i,i}(1)\\
    a: i \rightarrow j \in Q_{\Delta,1} \setminus (E \cup T) &\longmapsto E_{j,i}(a)\\
    t_{i,j} &\longmapsto E_{j,i}(1)
\end{align*}
where $\r{const}_i$ is the constant path in $\m C (\widetilde{Q_\Delta \setminus E})$ at the vertex $i \in Q_{\Delta,0}$, and $E_{i,j}(p)$ is the matrix with entries $p \in \m C \widetilde{Q}$ in the $(i,j)$-th place and zeroes everywhere else. Let $W$ be the potential on $Q_\Delta \setminus T$ obtained by removing all instances of the arrows that belong to $T$ from the cycles in $W_\Delta$. It is clear that $E$ remains as a cut for $(Q_\Delta \setminus T,\, W)$. Then $\sigma$ descends to a homomorphism
$$\overline{\sigma}: \r{Jac}(\widetilde{Q}_\Delta, W_\Delta, E) \longrightarrow \r{Mat}_r(\r{Jac}(\widetilde{Q}, W, E))$$
since $\sigma(\partial W_\Delta/\partial a) = E_{i,j}(\partial W/\partial a)$.

\begin{lem}\label{lem1.03}
The homomorphisms $\sigma$ and $\overline{\sigma}$ are isomorphisms.
\end{lem}

\begin{proof}
We define a homomorphism $\sigma^{-1}: \r{Mat}_r(\m C \widetilde{Q}) \rightarrow \m C (\widetilde{Q_\Delta \setminus E})$ by
$$E_{i,j}(a) \longmapsto t_{t(a),i}\, a\, t_{j,s(a)}.$$
For a path $p = a_m \ldots a_1 \in \m C \widetilde{Q}$, $\sigma^{-1}$ will send $E_{i,j}(p)$ to 
$$t_{t(a_m),i}\, a_m\, t_{t(a_{m-1}),s(a_m)}\, a_{m-1}\, t_{t(a_{m-2}),s(a_{m-1})} \ldots t_{t(a_1),s(a_2)}\, a_1\, t_{j,s(a_1)}.$$
Hence $\sigma^{-1}$ will descend to the 2D Jacobi algebras since $\sigma^{-1}(E_{i,j}(\partial W/\partial a)) = t_{s(a),i}\,(\partial W_\Delta/\partial a) \, t_{j,t(a)}$. It is clear that $\sigma$ and $\sigma^{-1}$ are inverses.
\end{proof} 

In particular \Cref{lem1.03} and \Cref{prop1.02} imply that
\begin{align}
    \r{Jac}(\widetilde{Q}, W, E) \cong \m C[\pi_1(\Sigma_g)]. \label{eq:2.001}
\end{align}

\begin{eg}\label{eg:1.02}
Returning to \Cref{eg:1.01} we could take $E=\{a,b,c\}$ as a cut for the potential $W_\Delta$ (we could also take $\{d,e,f\}$ or $\{g,h,i\}$ or $\{j,k,l\}$ among many others). Using the cut $E$ we can choose $T=\{e,h,k\}$ for our maximal tree, then we have that the localised 2D Jacobi algebra of $(Q, W, E)$ is
$$\r{Jac}(\widetilde{Q}, W, E) = \frac{\m C \langle d^{\pm 1}, f^{\pm 1}, g^{\pm 1}, i^{\pm 1}, j^{\pm 1}, l^{\pm 1}\rangle}{(jg-gd,\, lf-j,\, id-lif)}.$$
Note
\begin{align*}
    & jg-gd = 0 \quad \Rightarrow \quad j = gdg^{-1}\\
    & id-lif = 0 \quad \Rightarrow \quad f = i^{-1}l^{-1}id.
\end{align*}
Therefore $lf-j = 0$ becomes
$$li^{-1}l^{-1}id-gdg^{-1} = 0$$
or equivalently
$$li^{-1}l^{-1}idgd^{-1}g^{-1} - 1 = 0$$
and hence
\begin{align*}
    \r{Jac}(\widetilde{Q}, W, E) &= \frac{\m C \langle d^{\pm 1}, f^{\pm 1}, g^{\pm 1}, i^{\pm 1}, j^{\pm 1}, l^{\pm 1}\rangle}{(j-gdg^{-1},\, li^{-1}l^{-1}idgd^{-1}g^{-1} - 1,\, f-i^{-1}l^{-1}id)}\\
    &= \frac{\m C \langle d^{\pm 1}, g^{\pm 1}, i^{\pm 1}, l^{\pm 1}\rangle}{(li^{-1}l^{-1}idgd^{-1}g^{-1} - 1)}\\
    & \cong \m C[\pi_1(\Sigma_2)].
\end{align*}
\end{eg}

\section{Cohomological Hall algebras}

\subsection{The 2D CoHA multiplication}
Building on Schiffman and Vasserot's work from \cite{sv}, and in the same vein as the appendix in \cite{rs}, we shall describe a version of the 2D CoHA construction for $\bigoplus_{n \in \m N} \r{H}_{c,G}(Z_n, \m Q)$ the equivariant compactly supported cohomology of spaces $Z_n$ which can be written as the zero-loci of a specific type of function $f_n:X_n \rightarrow \m A^m$ on some smooth ambient space $X_n$.

Let us first fix some notation for various morphisms and natural transformations of sheaves. Our main reference for this will be \cite{ks}. All functors will be derived unless otherwise stated and when underived functors are necessary the same notation will be used as the derived versions, but we shall make it clear from the context which we mean. 

So let $X$ and $Y$ be varieties over $\m C$ and let $\r{D}(X)$ and $\r{D}(Y)$ denote the derived categories of sheaves of $\m Q$-vector spaces on $X$ and $Y$ respectively. For a map $f: X \rightarrow Y$ write
\begin{align*}
    & \eta^f: \r{id}_{\r{D}(Y)} \rightarrow f_* f^*\\
    & \sigma^f : f^*f_* \rightarrow \r{id}_{\r{D}(X)}
\end{align*}
and
\begin{align*}
    & \nu^f: f_! f^! \rightarrow \r{id}_{\r{D}(Y)}\\ 
    & \theta^f : \r{id}_{\r{D}(X)} \rightarrow f^! f_!
\end{align*}
for the canonical unit-counit pairs for the adjunctions $f^* \dashv f_*$ and $f_! \dashv f^!$ respectively. Let
$$\kappa^f: f^* \otimes \omega_{X/Y} \rightarrow f^!$$
be the natural transformation (see [\cite{ks} equation (3.1.6)]) where $\omega_{X/Y} := f^! \m Q_Y$ is the dualising sheaf. For the Cartesian square
\[\begin{tikzcd}[column sep = 6em, row sep = 6em]
X' \arrow[r, "f'"] & Y'\\
X \arrow[u, "\varphi"] \arrow[r, "f"] & Y \arrow[u, "\psi"] 
\end{tikzcd}\]
let
\begin{align*}
\epsilon^{\psi,f'}: \psi^* f'_! \xrightarrow{\,\, \sim \,\,} f_! \varphi^*
\end{align*}
be the base-change natural isomorphism.

\begin{rmk}\label{rmk3.3}
[\cite{ks} Theorem 3.1.5] says that the upper-shriek functors are unique up to isomorphism. If $X$ and $Y$ are smooth and hence complex manifolds, there exists a canonical choice of $f^!$ such that
$$\omega_{X/Y} = f^! \m Q_{Y} = \m Q_{X}[2d]$$
(where $d= \r{dim}(f) =  \r{dim}(X) - \r{dim}(Y)$) and so $\kappa^{f}$ becomes the natural transformation
$$\kappa^{f} : f^{*}[2d] \longrightarrow f^{!}.$$
If the map $f$ itself is smooth then we also get that
$$\kappa^{f} : f^{*}[2d] \xrightarrow{\,\,=\,\,} f^{!}$$
is an equality by [\cite{ks} Proposition 3.3.2 (ii)].
\end{rmk}

\begin{rmk}\label{rmk3.2}
If $f$ is proper then $f_* = f_!$ and so for such maps we shall use $f_!$ to denote all pushforward functors. If $\varphi$ and $\psi$ are closed immersions (and hence proper) we have that the natural transformations $\sigma^\psi, \sigma^\varphi, \theta^\psi, \theta^\varphi$ are the identity. We also have an equality of natural isomorphisms
\begin{align}
    \epsilon^{\psi,f'} &= (\sigma^\psi)f_! \varphi^* \circ \psi^*f'_!(\eta^\varphi) \label{eq:1.0001}
\end{align}
because $\psi_*=\psi_!$ and $\varphi_*=\varphi_!$ (see [\cite{ks} Proposition 2.5.11]).
\end{rmk}

To begin the description of the 2D CoHA multiplication take smooth finitely generated $\m C$-algebras $A$ and $B$ and an algebra homomorphism $\widehat{f}:B \rightarrow A$. Fix presentations of $A$ and $B$
\begin{align}
    A &\cong \m C \langle a_1 \ldots a_k \rangle/I \label{eq:1.00015}\\
    B &\cong \m C \langle b_1 \ldots b_l \rangle/J \nonumber
\end{align}
and a vector space $V$ of dimension $n$, then consider the spaces
\begin{align*}
    M_n(A) &= \left\{(R_i) \in \bigoplus_{i=1}^k \r{End}(V) \,\, \Big| \,\, p(R_i) = 0 \,\,\,\, \r{for all} \,\, p \in I \right\}\\
    M_n(B) &= \left\{(S_i) \in \bigoplus_{i=1}^l \r{End}(V) \,\, \Big| \,\, q(S_i) = 0 \,\,\,\, \r{for all} \,\, q \in J \right\}
\end{align*}
of representations with underlying vector space $V$ of $A$ and $B$ respectively, along with their natural $\r{GL}(V)$-actions. Let $\r{Rep}_n(A)$ and $\r{Rep}_n(B)$ denote the stack of representations with underlying vector space $V$ of $A$ and $B$ respectively i.e.
\begin{align*}
    \r{Rep}_n(A) &\cong  M_n(A)/\r{GL}(V)\\
    \r{Rep}_n(B) &\cong  M_n(B)/\r{GL}(V).
\end{align*}
From now on take $V= \m C^n$, hence $M_n(A)$ and $M_n(B)$ will be spaces of matrices. We obtain natural maps $f_n : M_n(A) \rightarrow M_n(B)$ induced from $\widehat{f}$ and so let $Z_n=f_n^{-1}(0)$, where $0 \in M_n(B)$ is the representation for which every element of $B$ acts on $\m C^n$ as zero. 

We want to define a multiplication on the dual of the vector space of $\r{GL}_n$-equivariant compactly supported cohomology of $Z_n$, or equivalently on the dual of the compactly supported cohomology of the stack $Z_n/\r{GL}_n$
$$\bigoplus_{n \in \m N} \r{H}_{c, \r{GL}_n}(Z_n, \m Q)^\vee = \bigoplus_{n \in \m N} \r{H}_{c}(Z_n/\r{GL}_n, \m Q)^\vee.$$
Let $M_{m,n}(A)$ denote the space of short-exact sequences of representations of $A$
$$0 \rightarrow \rho' \rightarrow \rho \rightarrow \rho'' \rightarrow 0$$
such that $\r{dim}(\rho')=m$ and $\r{dim}(\rho'')=n$. Then $M_{m,n}(A)$ naturally embeds into $M_{m+n}(A)$. Let $Z_{m,n}=Z_{m+n} \cap M_{m,n}(A)$ and let $\r{GL}_{m,n} \subset \r{GL}_{m+n}$ denote those automorphisms that preserve the subspace $\m C^m \subset \m C^{m+n}$. Then $\r{GL}_{m,n}$ acts naturally via conjugation on $M_{m,n}(A)$. Viewing a square matrix $R \in \r{Mat}_{m+n}$ in block-form as
$$R =
\begin{pmatrix}
R^{(1)} & R^{(3)}\\
R^{(4)} & R^{(2)}
\end{pmatrix}
$$
where $R^{(1)}$ is an $m \times m$-matrix, $R^{(2)}$ is an $n \times n$-matrix, $R^{(3)}$ is an $m \times n$-matrix, and $R^{(4)}$ is an $n \times m$ matrix, let
\begin{align*}
   & \pi^A_{m+n,m}: M_{m+n}(A) \rightarrow M_m(A)\\
   & \pi^A_{m+n,n}: M_{m+n}(A) \rightarrow M_n(A)\\
   & \pi^B_{m+n,m}: M_{m+n}(B) \rightarrow M_m(B)\\
   & \pi^B_{m+n,n}: M_{m+n}(B) \rightarrow M_n(B)
\end{align*}
denote the projections
\begin{align*}
    (R_i) &\longmapsto (R_i^{(1)})\\
    (R_i) &\longmapsto (R_i^{(2)})\\
    (S_i) &\longmapsto (S_i^{(1)})\\
    (S_i) &\longmapsto (S_i^{(2)})
\end{align*}
and we use the same notation for their restrictions to $M_{m,n}(A)$ and $M_{m,n}(B)$. Then let $Y_{m,n} \subset M_m(A) \times M_n(A) \times M_{m,n}(B)$ denote the space
$$Y_{m,n} = \Big\{ (\rho', \rho'', \sigma) : f_m(\rho')= \pi^B_{m+n,m}(\sigma),\, f_n(\rho'')=\pi^B_{m+n,n}(\sigma) \Big\}.$$
We can draw the following diagram of $\r{GL}_{m,n}$-equivariant maps
\begin{equation}\label{fig:2.1}
\begin{tikzcd}[column sep = 6em, row sep = 6em]
M_{m+n}(A) & M_{m,n}(A) \arrow[l, hook', "h"'] \arrow[r, "f"] & Y_{m,n}\\
Z_{m+n} \arrow[u, hook, "i_{m+n}"] & Z_{m,n} \arrow[u, hook, "i_{m,n}"] \arrow[l, hook', "\widetilde{h}"'] \arrow[r, "\widetilde{f}"] & Z_m \times Z_n \arrow[u, hook, "i"]
\end{tikzcd}
\end{equation}
where $h$ and $i_{m+n}$ and $i_{m,n}$ are the natural inclusions, $i = (i_m,i_n,0)$, the tilde denotes restriction to $Z_{m,n}$, and
$$f(\rho) = \big(\pi^A_{m+n,m}(\rho),\, \pi^A_{m+n,n}(\rho),\, f_{m+n}(\rho)\big).$$

\begin{rmk}
From the assumption that $A$ and $B$ are smooth, every space on the top row of \eqref{fig:2.1} is smooth. Also both squares in \eqref{fig:2.1} are Cartesian; the left-hand square is clear, and for the right-hand square this is why we needed to define $Y_{m,n}$ as above instead of taking $M_m(A) \times M_n(A)$ as one might expect from the standard convolution product on (cohomological) Hall algebras.
\end{rmk}

\begin{rmk}
When normally defining a cohomological Hall algebra one would usually try to use just the bottom line of \eqref{fig:2.1} as the correspondence diagram and use the pullback along $\widetilde{h}$ and the pushforward along $\widetilde{f}$ to induce the multiplication. Unfortunately however, in our case this will not work because the spaces $Z_n$ are not smooth. In particular this implies that $\widetilde{f}^! \m Q_{Z_m \times Z_n}$ is not necessarily equal to $\m Q_{Z_{m,n}}[2 \r{dim}(\widetilde{f})]$ and so there is no natural morphism we can take that would induce a pushforward along $\widetilde{f}$ to complete the definition of the multiplication. The following will describe how we alleviate this issue using the top row in \eqref{fig:2.1} along with the facts that each space in the top row is smooth and the right-hand square is Cartesian.
\end{rmk}

Let $p:Z_{m+n}/\r{GL}_{m+n} \rightarrow \r{pt}$ and $p':(Z_m \times Z_n)/(\r{GL}_m \times \r{GL}_n) \rightarrow \r{pt}$ be the respective structure maps. Recall that compactly supported cohomology is defined by taking the compactly supported pushforward of the constant sheaf along the structure map. The multiplication on the dual of the equivariant compactly supported cohomology is then given as follows:
\vspace{1em} \\

Firstly the inclusion $\r{GL}_{m,n} \hookrightarrow \r{GL}_{m+n}$ induces the quotient map $q: Z_{m+n}/\r{GL}_{m,n} \rightarrow Z_{m+n}/\r{GL}_{m+n}$ which gives the natural morphism of sheaves
$$\eta^q(\m Q) : \m Q_{Z_{m+n}/\r{GL}_{m+n}} \rightarrow q_* q^* \m Q_{Z_{m+n}/\r{GL}_{m+n}}.$$
Because $q$ is proper and so $q_* = q_!$ we have that $q_* q^* \m Q_{Z_{m+n}/\r{GL}_{m+n}} = q_! \m Q_{Z_{m+n}/\r{GL}_{m,n}}$ and so applying $p_!$ to $\eta^q(\m Q)$ gives the pullback morphism
\begin{align}
    q^\star: \r{H}_{c, \r{GL}_{m+n}}(Z_{m+n}, \m Q) \rightarrow \r{H}_{c, \r{GL}_{m,n}}(Z_{m+n}, \m Q). \label{eq:1.000001}
\end{align}

Next we obtain a natural pullback induced by $\widetilde{h}$. Because $\widetilde{h}$ is $\r{GL}_{m,n}$-equivariant we get an induced map on quotient stacks $\overline{\widetilde{h}}: Z_{m,n}/\r{GL}_{m,n} \rightarrow Z_{m+n}/\r{GL}_{m,n}$. We then have the natural morphism 
$$\eta^{\overline{\widetilde{h}}}(\m Q): \m Q_{Z_{m+n}/\r{GL}_{m,n}} \rightarrow \overline{\widetilde{h}}_* \overline{\widetilde{h}}^* \m Q_{Z_{m+n}/\r{GL}_{m,n}}$$
where again, because $\overline{\widetilde{h}}$ is a closed immersion and hence proper, we have $\overline{\widetilde{h}}_*= \overline{\widetilde{h}}_!$ and therefore $\overline{\widetilde{h}}_* \overline{\widetilde{h}}^* \m Q_{Z_{m+n}/\r{GL}_{m,n}} = \overline{\widetilde{h}}_! \m Q_{Z_{m,n}/\r{GL}_{m,n}}$. So after pushing forward along $p_!$ this gives a morphism
\begin{align}
    \overline{\widetilde{h}}^\star: \r{H}_{c, \r{GL}_{m,n}}(Z_{m+n}, \m Q) \rightarrow \r{H}_{c, \r{GL}_{m,n}}(Z_{m,n}, \m Q). \label{eq:1.000002}
\end{align}

Then we have a natural pushforward induced by $f$ (here we utilise the fact that in \eqref{fig:2.1} the right-hand square is Cartesian and the top row consists of smooth spaces). As $f$ is $\r{GL}_{m,n}$-equivariant we get the induced map $\overline{f}: M_{m,n}(A)/\r{GL}_{m,n} \rightarrow Y_{m,n}/\r{GL}_{m,n}$ and so have the morphism of sheaves 
$$\nu^{\overline{f}}(\m Q) : \overline{f}_!\overline{f}^! \m Q_{Y_{m,n}/\r{GL}_{m,n}} \rightarrow \m Q_{Y_{m,n}/\r{GL}_{m,n}}.$$
Because $M_{m,n}(A)/\r{GL}_{m,n}$ is smooth we have $\overline{f}_!\overline{f}^! \m Q_{Y_{m,n}/\r{GL}_{m,n}} = \overline{f}_! \m Q_{M_{m,n}(A)/\r{GL}_{m,n}}[2d]$ where $d=\r{dim}(\overline{f}) = \r{dim}(f)$. Restricting this morphism to $(Z_m \times Z_n)/\r{GL}_{m,n}$ and precomposing with the base-change inverse $(\epsilon^{\overline{i},\overline{f}})^{-1}: \overline{\widetilde{f}}_! \, \overline{i}_{m,n}^* \xrightarrow{\sim} \overline{i}^* \, \overline{f}_!$ (where $\overline{i}$ and $\overline{i}_{m,n}$ denote the maps on the quotient stacks induced by $i$ and $i_{m,n}$ respectively) gives
\begin{align*}
    & \overline{\widetilde{f}}_! \m Q_{Z_{m,n}/\r{GL}_{m,n}}[2d] = \overline{\widetilde{f}}_! \overline{i}_{m,n}^* \m Q_{M_{m,n}(A)/\r{GL}_{m,n}}[2d] \xrightarrow{\epsilon^{\overline{i}, \overline{f}}(\m Q)[2d]^{-1}} \overline{i}^* \overline{f}_! \m Q_{M_{m,n}(A)/\r{GL}_{m,n}}[2d]\\
    & = \overline{i}^* \overline{f}_! \overline{f}^! \m Q_{Y_{m,n}/\r{GL}_{m,n}} \xrightarrow{\overline{i}^*(\nu^{\overline{f}}(\m Q))} \overline{i}^* \m Q_{Y_{m,n}/\r{GL}_{m,n}} = \m Q_{(Z_m \times Z_n)/\r{GL}_{m,n}}
\end{align*}
and so pushing forward along $p'_!$ gives us a morphism
\begin{align}
    \overline{f}_\star : \r{H}_{c, \r{GL}_{m,n}}^{+2d}(Z_{m,n}, \m Q) \rightarrow \r{H}_{c, \r{GL}_{m,n}}(Z_m \times Z_n, \m Q). \label{eq:1.000003}
\end{align}

Finally we have the isomorphisms
\begin{align}
    \r{H}_{c, \r{GL}_{m,n}}^{+2mn}(Z_m \times Z_n, \m Q) \cong \r{H}_{c, \r{GL}_m \times \r{GL}_n}(Z_m \times Z_n, \m Q) \cong \r{H}_{c, \r{GL}_m}(Z_m, \m Q) \otimes \r{H}_{c,\r{GL}_n}(Z_n, \m Q) \label{eq:1.000004}
\end{align}
- the first induced by the affine fibration $(Z_m \times Z_n) /\r{GL}_{m,n} \rightarrow (Z_m \times Z_n) /(\r{GL}_m \times \r{GL}_n)$ with fibre $\r{Mat}_{m \times n} \cong \m A^{mn}$ which comes from the affine fibration $\r{GL}_{m,n} \rightarrow \r{GL}_m \times \r{GL}_n$, and the second isomorphism is the K\"unneth isomorphism for equivariant compactly supported cohomology.

Our multiplication is then the dual of the composition of \eqref{eq:1.000001}, \eqref{eq:1.000002} and \eqref{eq:1.000003} with the isomorphism \eqref{eq:1.000004}
$$(\overline{f}_\star \circ \overline{\widetilde{h}}^\star \circ q^\star)^\vee : \r{H}_{c, \r{GL}_m}(Z_m, \m Q)^\vee \otimes \r{H}_{c,\r{GL}_n}(Z_n, \m Q)^\vee \longrightarrow \r{H}_{c, \r{GL}_{m+n}}^{-2(d+mn)}(Z_{m+n}, \m Q)^\vee.$$
\\

\begin{rmk}
The diagram \eqref{fig:2.1} depends upon our initial choices \eqref{eq:1.00015} of presentations of $A$ and $B$, and hence so does this multiplication. This underpins the central question we are trying to answer- when do different presentations result in the same multiplication?
\end{rmk}

\begin{rmk}
We have defined things above using sheaves on quotient stacks. Now if $G$ is an algebraic group acting on a variety $X$ then
$$X \rightarrow X/G$$
is an atlas for the quotient stack $X/G$, and the derived category of sheaves on $X/G$ is equivalent to the derived category of $G$-equivariant sheaves on $X$. Because the morphisms $\overline{\widetilde{h}}^\star$ and $\overline{f}_\star$ above are induced from the equivariant maps $\widetilde{h}$ and $f$ on the atlases of the quotient stacks, we in turn have induced morphisms $\widetilde{h}^\star$ and $f_\star$ in the derived category of equivariant sheaves on the atlases that are compatible with $\overline{\widetilde{h}}^\star$ and $\overline{f}_\star$. Hence it is enough to instead view everything at the level of equivariant sheaves on the atlases of the quotient stacks and check the commutativity of diagrams there. Therefore in the following sections we work only at the level of sheaves on varieties and do not work explicitly with the compactly supported cohomology of the quotient stacks/equivariant compactly supported cohomology. In particular instead of the morphism on equivariant compactly supported cohomology $\overline{\widetilde{h}}^\star: \r{H}_{c, \r{GL}_{m,n}}(Z_{m+n}, \m Q) \rightarrow \r{H}_{c, \r{GL}_{m,n}}(Z_{m,n}, \m Q)$ defined above, we consider the morphism of sheaves
\begin{align}
    \widetilde{h}^\star := \eta^{\widetilde{h}}(\m Q): \m Q_{Z_{m+n}} \rightarrow \widetilde{h}_* \widetilde{h}^* \m Q_{Z_{m+n}} = \widetilde{h}_! \m Q_{Z_{m,n}} \label{eq:1.0000045}
\end{align}
and instead of $\overline{f}_\star: \r{H}_{c, \r{GL}_{m,n}}^{+2d}(Z_{m,n}, \m Q) \rightarrow \r{H}_{c, \r{GL}_{m,n}}(Z_m \times Z_n, \m Q)$ we consider $f_\star$ defined as the composition
\begin{align}
    & \widetilde{f}_! \m Q_{Z_{m,n}}[2d] = \widetilde{f}_! i_{m,n}^* \m Q_{M_{m,n}(A)}[2d] \xrightarrow{\epsilon^{i, f}(\m Q[2d])^{-1}} i^* f_! \m Q_{M_{m,n}(A)}[2d] \label{eq:1.000005}\\
    & = i^* f_! f^! \m Q_{Y_{m,n}} \xrightarrow{i^*(\nu^f(\m Q))} i^* \m Q_{Y_{m,n}} = \m Q_{Z_m \times Z_n} \nonumber
\end{align}
\end{rmk}

Since the definition of the pushforward along $f$ is slightly unorthodox we verify that the pushforward of a composition is the composition of pushforwards.

\begin{lem}\label{lem3.11}
Suppose we have a diagram of Cartesian squares
\[
\begin{tikzcd}[column sep = 6em, row sep = 5em]
X_1 \arrow[r, "f"] & X_2 \arrow[r, "g"] & X_3\\
Z_1 \arrow[u, hook, "i_1"] \arrow[r, "\widetilde{f}"] & Z_2 \arrow[u, hook, "i_2"] \arrow[r, "\widetilde{g}"] & Z_3 \arrow[u, hook, "i_3"]
\end{tikzcd}
\]
where $X_1,\, X_2,\, X_3$ are smooth, $i_1,\, i_2,\, i_3$ are closed immersions, and $d= \r{dim}(f),\, e= \r{dim}(g)$. Then the morphism
$$(g \circ f)_\star: (\widetilde{g} \circ \widetilde{f})_! \m Q_{Z_1}[2(d+e)] \longrightarrow \m Q_{Z_3}$$
equals the morphism
$$g_\star \circ \widetilde{g}_!(f_\star): \widetilde{g}_! \widetilde{f}_! \m Q_{Z_1}[2(d+e)] \longrightarrow \m Q_{Z_3}.$$
\end{lem}

\begin{proof}
From the definition \eqref{eq:1.000005} we see that $g_\star \circ \widetilde{g}_!(f_\star)$ is the composition
\begin{align}
    & (\widetilde{g} \circ \widetilde{f})_! i_1^* (g \circ f)^! \m Q_{X_3} = \widetilde{g}_! \widetilde{f}_! i_1^* f^! g^! \m Q_{X_3} \xrightarrow{\widetilde{g}_! (\epsilon^{i_2,f}(f^! g^! \m Q))^{-1}} \widetilde{g}_! i_2^* f_! f^! g^! \m Q_{X_3} \xrightarrow{\widetilde{g}_! i_2^*(\nu^f(g^! \m Q))} \label{eq:3.61}\\
    & \widetilde{g}_! i_2^* g^! \m Q_{X_3} \xrightarrow{\epsilon^{i_3,g}(g^! \m Q)^{-1}} i_3^* g_! g^! \m Q_{X_3} \xrightarrow{i_3^*(\nu^g(\m Q))} i_3^* \m Q_{X_3} = \m Q_{Z_3}. \nonumber
\end{align}
Applying the natural transformation $\epsilon^{i_3, g}$ to the morphism 
$$i_3^*g_! f_! f^! g^! \m Q_{X_3} \xrightarrow{i_3^* g_!(\nu^f(g^! \m Q))} i_3^*g_! g^! \m Q_{X_3}$$
gives the commutative square
\[\begin{tikzcd}[column sep = 8em, row sep = 7em]
i_3^*g_! f_! f^! g^! \m Q_{X_3} \arrow[d, "\epsilon^{i_3,g}(f_! f^! g^! \m Q)"] \arrow[r, "i_3^* g_!(\nu^f(g^! \m Q_{X_3}))"] & i_3^*g_! g^! \m Q_{X_3} \arrow[d, "\epsilon^{i_3,g}(g^! \m Q)"]\\
\widetilde{g}_! i_2^* f_! f^! g^! \m Q_{X_3} \arrow[r, "\widetilde{g}_! i_2^*(\nu^f(g^! \m Q_{X_3}))"] & \widetilde{g}_! i_2^* g^! \m Q_{X_3}
\end{tikzcd}\]
and so \eqref{eq:3.61} is equal to
\begin{align}
    & (\widetilde{g} \circ \widetilde{f})_! i_1^* (g \circ f)^! \m Q_{X_3} = \widetilde{g}_! \widetilde{f}_! i_1^* f^! g^! \m Q_{X_3} \xrightarrow{\widetilde{g}_!(\epsilon^{i_2, f}(f^! g^! \m Q))^{-1}} \widetilde{g}_! i_2^* f_! f^! g^! \m Q_{X_3} \xrightarrow{(\epsilon^{i_3,g}(f_! f^! g^! \m Q))^{-1}} \label{eq:3.62}\\
    & i_3^* g_! f_! f^! g^! \m Q \xrightarrow{i_3^* g_!(\nu^f(g^! \m Q))}  i_3^* g_! g^! \m Q_{X_3} \xrightarrow{i_3^*(\nu^g(\m Q))} i_3^* \m Q_{X_3} = \m Q_{Z_3}. \nonumber
\end{align}
Now since
$$i_3^* g_! f_! = i_3^* (g \circ f)_! \xrightarrow[\sim]{\epsilon^{i_3, g \circ f}} (\widetilde{g} \circ \widetilde{f})_! i_1^* = \widetilde{g}_! \widetilde{f}_! i_1^* $$
equals
$$i_3^* g_! f_! \xrightarrow[\sim]{\widetilde{g}_!(\epsilon^{i_2,f})} \widetilde{g}_! i_2^* f_! \xrightarrow[\sim]{(\epsilon^{i_3,g}) f_!} \widetilde{g}_! \widetilde{f}_! i_1^*$$
because these natural transformations are defined using the underlying maps $f,\,g,\, i_1,\, i_2,\, i_3$ directly, we have that \eqref{eq:3.62} equals
\begin{align}
    & (\widetilde{g} \circ \widetilde{f})_! i_1^* (g \circ f)^! \m Q_{X_3} \xrightarrow{\epsilon^{i_3, g \circ f}(f^! g^! \m Q))^{-1}} i_3^* (g \circ f)_! f^! g^! \m Q_{X_3} =  i_3^* g_! f_! f^! g^! \m Q_{X_3} \label{eq:3.63}\\
    & \xrightarrow{i_3^* g_!(\nu^f(g^! \m Q))}  i_3^* g_! g^! \m Q_{X_3} \xrightarrow{i_3^*(\nu^g(\m Q))} i_3^* \m Q_{X_3} = \m Q_{Z_3}. \nonumber
\end{align}
Then as
$$g_! f_! f^! g^! = (g \circ f)_! (g \circ f)^! \xrightarrow{\nu^{g \circ f}} \r{id}_{X_3}$$
equals
$$g_! f_! f^! g^! \xrightarrow{g_!(\nu^f) g^!} g_! g^! \xrightarrow{\nu^g} \r{id}_{X_3}$$
we get \eqref{eq:3.63} is equal to
\begin{align}
    (\widetilde{g} \circ \widetilde{f})_! i_1^* (g \circ f)^! \m Q_{X_3} \xrightarrow{\epsilon^{i_3, g \circ f}(f^! g^! \m Q))^{-1}} i_3^* (g \circ f)_! (g \circ f)^! \m Q_{X_3} \xrightarrow{i_3^*(\nu^{g \circ f}(\m Q))} i_3^* \m Q_{X_3} = \m Q_{Z_3} \label{eq:1.205}
\end{align}
which is exactly $(g \circ f)_\star$. 
\end{proof}

\subsection{2D CoHA multiplications for the character variety}
We seek to compare two 2D CoHAs for the character variety, i.e. to compare two multiplications induced by diagrams as in \eqref{fig:2.1} on the dual of the compactly supported cohomology of the stack $\r{Rep}_n(\m C[\pi_1(\Sigma_g)])$ of $n$-dimensional representations of the algebra over the fundamental group of a Riemann surface $\Sigma_g$ of genus g. The first of these will come from the standard presentation of the fundamental group algebra of a Riemann surface
$$\m C[\pi_1(\Sigma_g)] \cong \frac{\m C \langle x_1^{\pm 1},y_1^{\pm 1}, \ldots, x_g^{\pm 1},y_g^{\pm 1} \rangle}{(\lambda - 1)}$$
where $\lambda = \prod_{i=1}^g x_i y_i x_i^{-1} y_i^{-1}$. This presentation allows us to define the space $M_n(\m C[\pi_1(\Sigma_g)])$ of $n$-dimensional representations of $\m C[\pi_1(\Sigma_g)]$ as the zero-locus of the function $\lambda_n: \r{GL}_n^{2g} \rightarrow \r{Mat}_n$ given by 
$$(A_i,B_i) \longmapsto \prod_{i=1}^g A_iB_iA_i^{-1}B_i^{-1}-\r{Id}_n$$
which has a natural $\r{GL}_n$-action via simultaneous conjugation. Then we have that
$$\r{Rep}_n(\m C[\pi_1(\Sigma_g)]) \cong M_n(\m C[\pi_1(\Sigma_g)])/\r{GL}_n.$$
Recall that $\r{Mat}_{m,n}$ is the space of all $(m+n)$-dimensional matrices whose lower-left $n \times m$-block is 0. We shall call such matrices \i{$m,n$-block matrices}. For such an $m,n$-block matrix $R$ we write
$$R= \begin{pmatrix}
        R^{(1)} & R^{(3)}\\
        0 & R^{(2)}
\end{pmatrix}$$
where $R^{(1)}$ is an $m \times m$-matrix, $R^{(2)}$ is an $n \times n$-matrix and $R^{(3)}$ is an $m \times n$-matrix. Define
$$Y_{m,n}= \bigg\{\Big((A_i,B_i),(C_i, D_i), R\Big) \in \r{GL}_m^{2g} \times \r{GL}_n^{2g} \times \r{Mat}_{m,n} \,\, \Big| \,\, \big(R^{(1)}, R^{(2)} \big) = \big(\lambda_m(A_i,B_i),\lambda_n(C_i,D_i) \big)\bigg\}$$
then this space is defined analogously to the $Y_{m,n}$ from diagram \eqref{fig:2.1}. Let
$$V_n = \lambda_n^{-1}(0) = M_n(\m C[\pi_1(\Sigma_g)]).$$
We get a new diagram, in the form of \eqref{fig:2.1}, of $\r{GL}_{m,n}$-equivariant maps
\begin{equation}\label{fig:2}
\begin{tikzcd}[column sep = 6em, row sep = 6em]
\r{GL}_{m+n}^{2g} & \r{GL}_{m,n}^{2g} \arrow[l, hook', "h"'] \arrow[r, "f"] & Y_{m,n}\\
V_{m+n} \arrow[u, hook, "j_{m+n}"] & V_{m,n} \arrow[u, hook, "j_{m,n}"] \arrow[l, hook', "\widetilde{h}"'] \arrow[r, "\widetilde{f}"] & V_m \times V_n \arrow[u, hook, "i"]
\end{tikzcd}
\end{equation}
where $j_{m+n}: V_{m+n} \hookrightarrow \r{GL}_{m+n}^{2g}$ is the natural inclusion and
$$f(A_i,B_i)= \Big((A_i^{(1)},B_i^{(1)}),\, (A_i^{(2)},B_i^{(2)}),\,\, \lambda_{m+n}(A_i,B_i)\Big).$$
The diagram \eqref{fig:2} induces the first multiplication on $\bigoplus_n \r{H}_c(\r{Rep}_n(\m C[\pi_1(\Sigma_g)]), \m Q)^\vee$ as described in Section 3.1.

\vspace{3em} 

For our second multiplication let $\Delta$ be a brane tiling of the surface $\Sigma_g$. As in Section 2 denote by $Q_\Delta$ the dual quiver with associated potential $W_\Delta$, $E$ a cut for $W_\Delta$, and $T \subset Q_\Delta \setminus E$ a maximal tree. Let $Q$ be the quiver with one vertex obtained by contracting $T$ in $Q_\Delta \setminus E$ and let $W$ be corresponding potential on $Q_\Delta \setminus T$. For $n \in \m N$ let
$$M_n = M_n(\m C \widetilde{Q}) = \bigoplus_{a \in Q_1} \r{GL}_n$$
be the space of $n$-dimensional representations of $\m C \widetilde{Q}$, where recall $\m C \widetilde{Q}$ is the localised path algebra of $Q$. Let $M_{m,n} =  M_{m,n}(\m C \widetilde{Q})$ be the space of short exact sequences of $m$-dimensional and $n$-dimensional representations. Then $M_{m,n}$ can be viewed as the space of $m,n$-block representations of $\m C \widetilde{Q}$. Define
$$Y'_{m,n} = \left\{\big(\rho', \rho'', (R_e) \big) \in M_m \times M_n \times \r{Mat}_{m,n}^{|E|} \,\, \bigg| \,\, \big(R_e^{(1)},R_e^{(2)} \big) = \left(\frac{\partial W}{\partial e}(\rho'), \frac{\partial W}{\partial e}(\rho'') \right) \,\, \r{for all $e \in E$}\, \right\}.$$
Let 
$$Z_n = \left\{ \rho \in M_n \,\, \bigg| \,\, \frac{\partial W}{\partial e}(\rho) =0 \,\,\, \r{for all} \,\,\, e \in E \right\} \subset M_n$$
then clearly $Z_n = M_n(\r{Jac}(\widetilde{Q}, W, E))$ and so $$\r{Rep}_n(\r{Jac}(\widetilde{Q},W,E)) \cong Z_n/\r{GL}_n.$$
As in \eqref{fig:2.1}, we get the following diagram of $\r{GL}_{m,n}$-equivariant maps
\begin{equation}\label{fig:1}
\begin{tikzcd}[column sep = 6em, row sep = 6em]
M_{m+n} & M_{m,n} \arrow[l, hook', "h'"'] \arrow[r, "f''"] & Y'_{m,n}\\
Z_{m+n} \arrow[u, hook, "j'_{m+n}"] & Z_{m,n} \arrow[u, hook, "j'_{m,n}"] \arrow[l, hook', "\widetilde{h'}"'] \arrow[r, "\widetilde{f''}"] & Z_m \times Z_n \arrow[u, hook, "i'"]
\end{tikzcd}
\end{equation}
where $j'_{m+n}: Z_{m+n} \hookrightarrow M_{m+n}$ is the natural inclusion and 
$$f''(\rho) = \left(\rho^{(1)}, \rho^{(2)}, \left(\frac{\partial W}{\partial e}(\rho) \right)_{e \in E} \right).$$
The diagram \eqref{fig:1} then induces the second 2D CoHA multiplication on \begin{align*}
    \bigoplus_{n \in \m N} \r{H}_{c, \r{GL}_n}(Z_n, \m Q)^\vee &= \bigoplus_{n \in \m N} \r{H}_{c, \r{GL}_n}(M_n(\r{Jac}(\widetilde{Q},W,E)), \m Q)^\vee\\
    &\cong \bigoplus_{n \in \m N} \r{H}_c(\r{Rep}_n(\r{Jac}(\widetilde{Q},W,E)), \m Q)^\vee\\
    &\cong \bigoplus_{n \in \m N} \r{H}_c(\r{Rep}_n(\m C[\pi_1(\Sigma_g)]), \m Q)^\vee
\end{align*}
where the final isomorphism is induced from the isomorphism of algebras \eqref{eq:2.001}. Note, a priori, this multiplication depends upon the choices for the cut $E$ and the maximal tree $T$.

\section{Comparing morphisms on compactly supported cohomology}

In this section we accumulate results that will allow us to compare the morphisms on compactly supported cohomology used in defining our two prior multiplications on the 2D CoHA.

Consider the diagram
\begin{equation}\label{fig:1.001}
\begin{tikzcd}[column sep = 5em, row sep = 5em]
X \arrow[r, "f"] & Y\\
X \times S \arrow[u, "\pi_X"] \arrow[r, "f'"] & Y \times S \arrow[u, "\pi_Y"]\\
X \arrow[u, hook, "\varphi"] \arrow[r, "f"] & Y \arrow[u, hook, "\psi"]
\end{tikzcd}
\end{equation}
where
\begin{itemize}
    \item $X,Y,S$ are smooth\\
    \item $\pi_X$ and $\pi_Y$ are the respective projections\\
    \item $\varphi$ and $\psi$ are closed immersions such that $\pi_X \circ \varphi = \r{id}_X$ and $\pi_Y \circ \psi = \r{id}_Y$\\
    \item both squares are Cartesian (and hence $\r{dim}(f) = \r{dim}(f') = d$).
\end{itemize}

\begin{rmk}\label{rmk3.1}
This last condition is satisfied if and only if we can write $f'$ as $(f \times \r{id}_S) \circ \zeta$ where $\zeta: X \times S \xrightarrow{\sim} X \times S$ is an isomorphism.
\end{rmk}

For convenience we write $X'= X \times S$ and $Y'= Y \times S$ and let $m= \r{dim}(S)$. From \Cref{rmk3.3} we have that $\pi_X^*[-2m] = \pi_X^!$ and $\pi_Y^*[-2m] = \pi_Y^!$ since $\pi_X$ and $\pi_Y$ are smooth, and so
$$f^!= \r{id}_{\r{D}(X)}^* f^! = \varphi^* \pi_X^* f^! = \varphi^* \pi_X^![-2m] f^! = \varphi^* f^{\prime !} \pi_Y^![-2m] = \varphi^* f^{\prime !} \pi_Y^*.$$
We refer back to the beginning of Section 3.1 for a reminder on the notation of various natural transformations between functors of sheaves.
\begin{lem}\label{lem1.002}
The pair of compositions
\begin{align}
    \r{id}_{\r{D}(X)} = \varphi^* \pi_X^* \xrightarrow{\varphi^*(\theta^{f'}) \pi_X^*} \varphi^* f^{\prime !} f'_! \pi_X^* \xrightarrow{\varphi^* f^{\prime !}(\epsilon^{\pi_Y,f})^{-1}} \varphi^* f^{\prime !} \pi_Y^* f_! = f^! f_! \label{eq:1.00315}
\end{align}
and
\begin{align}
   f_! f^! = f_! \varphi^* f^{\prime !} \pi_Y^* \xrightarrow{(\epsilon^{\psi,f'})^{-1} f^{\prime !} \pi_Y^*} \psi^* f'_! f^{\prime !} \pi_Y^* \xrightarrow{\psi^*(\nu^{f'}) \pi_Y^*} \psi^* \pi_Y^* = \r{id}_{\r{D}(Y)} \label{eq:1.00305}
\end{align}
form a unit-counit adjunction for $f_! \dashv f^!$.
\end{lem}

\begin{proof}
Denote \eqref{eq:1.00315} by $\alpha$ and \eqref{eq:1.00305} by $\beta$. Then we must show that the following compositions
\begin{align*}
    & \r{a)} \qquad \quad f_! \xrightarrow{f_!(\alpha)} f_! f^! f_! \xrightarrow{(\beta)f_!} f_!\\
    & \r{b)} \qquad \quad f^! \xrightarrow{(\alpha)f^!} f^! f_! f^! \xrightarrow{f^!(\beta)} f^!
\end{align*}
are the identity natural transformations.

\vspace{2em}

For a) we have that
$$f_! \xrightarrow{f_!(\alpha)} f_! f^! f_! \xrightarrow{(\beta)f_!} f_!$$
is
\begin{align}
    & f_! = f_! \varphi^* \pi_X^* \xrightarrow{f_!\varphi^*(\theta^{f'}) \pi_X^*} f_! \varphi^* f^{\prime !} f'_! \pi_X^* \xrightarrow{f_! \varphi^* f^{\prime !}(\epsilon^{\pi_Y,f})^{-1}} f_! \varphi^* f^{\prime !} \pi_Y^* f_! = f_! f^! f_! \label{eq:1.002}\\
    & = f_! \varphi^* f^{\prime !} \pi_Y^* f_! \xrightarrow{(\epsilon^{\psi,f'})^{-1} f^{\prime !} \pi_Y^* f_!} \psi^* f'_! f^{\prime !} \pi_Y^* f_! \xrightarrow{\psi^*(\nu^{f'}) \pi_Y^* f_!} \psi^* \pi_Y^*  f_! = f_!. \nonumber
\end{align}
Then because these are natural transformations \eqref{eq:1.002} is equal to
\begin{align}
    & f_! = f_! \varphi^* \pi_X^* \xrightarrow{(\epsilon^{\psi,f'})^{-1} \pi_X^*} \psi^* f'_! \pi_X^* \xrightarrow{\psi^* f'_!(\theta^{f'}) \pi_X^*} \psi^* f'_! f^{\prime !} f'_! \pi_X^* \xrightarrow{\psi^*(\nu^{f'}) f'_! \pi_X^*} \label{eq:1.00025}\\
    & \psi^* f'_! \pi_X^* \xrightarrow{\psi^*(\epsilon^{\pi_Y,f})^{-1}} \psi^* \pi_Y^*  f_! = f_!. \nonumber
\end{align}
Indeed by evaluating at some object $\c F \in \r{D}(X)$ we can first apply the commutativity property of the natural transformation $(\epsilon^{\psi, f'})^{-1}$ to the morphism $f_! \varphi^*(\theta^{f'}(\pi_X^* \c F))$ giving us the commutative square
\begin{equation}
\begin{tikzcd}[column sep = 7em, row sep = 5em]
f_! \varphi^* \pi_X^* \c F \arrow[d, "(\epsilon^{\psi, f'})^{-1}(\pi_X^* \c F)"] \arrow[r, "f_! \varphi^*(\theta^{f'}(\pi_X^* \c F))"] & f_! \varphi^* f^{\prime !} f'_! \pi_X^* \c F \arrow[d, "(\epsilon^{\psi, f'})^{-1}(f^{\prime !} f'_! \pi_X^* \c F)"]\\
\psi^* f'_! \pi_X^* \c F \arrow[r, "\psi^* f'_!(\theta^{f'}(\pi_X^* \c F))"] & \psi^* f'_! f^{\prime !} f'_! \pi_X^* \c F
\end{tikzcd}    
\end{equation}
We get similar commutative squares for the natural transformation $(\epsilon^{\psi, f'})^{-1}$ and morphism $f_! \varphi^* f^{\prime !}(\epsilon^{\pi_Y, f}( \c F))^{-1}$, and the natural transformation $\nu^{f'}$ and morphism $f'_! f^{\prime !}(\epsilon^{\pi_Y, f}( \c F))^{-1}$ resulting in commutative squares
\begin{equation}
\begin{tikzcd}[column sep = 9em, row sep = 5em]
f_! \varphi^* f^{\prime !} f'_! \pi_X^* \c F \arrow[d, "(\epsilon^{\psi, f'})^{-1}(f^{\prime !} f'_! \pi_X^* \c F)"] \arrow[r, "f_! \varphi^* f^{\prime !}(\epsilon^{\pi_Y,f}(\c F)^{-1})"] & f_! \varphi^* f^{\prime !} \pi_Y^* f_! \c F \arrow[d, "(\epsilon^{\psi, f'})^{-1}(f^{\prime !} \pi_Y^* f_! \c F)"]\\
\psi^* f'_! f^{\prime !} f'_! \pi_X^* \c F \arrow[r, "\psi^* f'_! f^{\prime !}(\epsilon^{\pi_Y,f}(\c F)^{-1})"] & \psi^* f'_! f^{\prime !} \pi_Y^* f_! \c F
\end{tikzcd}    
\end{equation}
and
\begin{equation}
\begin{tikzcd}[column sep = 8em, row sep = 5em]
f'_! f^{\prime !} f'_! \pi_X^* \c F \arrow[d, "\nu^{f}(f'_! \pi_X^* \c F)"] \arrow[r, "f'_! f^{\prime !}(\epsilon^{\pi_Y,f}(\c F)^{-1})"] & f'_! f^{\prime !} \pi_Y^* f_! \c F \arrow[d, "\nu^{f}(\pi_Y^* f_! \c F)"]\\
f'_! \pi_X^* \c F \arrow[r, "\epsilon^{\pi_Y,f}(\c F)^{-1}"] & \pi_Y^* f_! \c F
\end{tikzcd}    
\end{equation}
Combining these three squares we can see that \eqref{eq:1.00025} equals \eqref{eq:1.002}.

Now by the adjunction $f'_! \dashv f^{\prime !}$ \eqref{eq:1.00025} just equals
\begin{align}
    f_! = f_! \varphi^* \pi_X^* \xrightarrow{(\epsilon^{\psi,f'})^{-1} \pi_X^*} \psi^* f'_! \pi_X^* \xrightarrow{\psi^*(\epsilon^{\pi_Y,f})^{-1}} \psi^* \pi_Y^* f_! = f_!. \label{eq:1.00026}
\end{align}
Then because $(\epsilon^{\psi,f'})\pi_X^* \circ \psi^*(\epsilon^{\pi_Y, f}) = \epsilon^{\pi_Y \circ \psi, f} = \epsilon^{\r{id}_X,f} = \r{id}$ \eqref{eq:1.00026} is the identity as required.

\vspace{2em}

For b) we have that
$$f^! \xrightarrow{(\alpha)f^!} f^! f_! f^! \xrightarrow{f^!(\beta)} f^!$$
is
\begin{align}
    & f^! = \varphi^* \pi_X^* f^! \xrightarrow{\varphi^*(\theta^{f'}) \pi_X^* f^!} \varphi^* f^{\prime !} f'_! \pi_X^* f^! \xrightarrow{\varphi^* f^{\prime !}(\epsilon^{\pi_Y,f})^{-1} f^!} \varphi^* f^{\prime !} \pi_Y^* f_! f^! = f^! f_! f^! \label{eq:1.003}\\
    & = f^! f_! \varphi^* f^{\prime !} \pi_Y^* \xrightarrow{f^!(\epsilon^{\psi,f'})^{-1} f^{\prime !} \pi_Y^*} f^! \psi^* f'_! f^{\prime !} \pi_Y^* \xrightarrow{f^!\psi^*(\nu^{f'}) \pi_Y^*} f^! \psi^* \pi_Y^* = f^!. \nonumber
\end{align}
Now because
$$\pi_Y^* f_! \xrightarrow{\epsilon^{\pi_Y,f}} f'_! \pi_X^*$$
equals
$$\pi_Y^* f_! = \pi_Y^* \psi^* \pi_Y^* f_! \xrightarrow{\pi_Y^* \psi^*(\epsilon^{\pi_Y, f}))} \pi_Y^* \psi^* f'_! \pi_X^* \xrightarrow{\pi_Y^*(\epsilon^{\psi, f'})\pi_X^*} \pi_Y^* f_! \varphi^* \pi_X^* \xrightarrow{(\epsilon^{\pi_Y,f})\varphi^* \pi_X^*} f'_! \pi_X^* \varphi^* \pi_X^* = f'_! \pi_X^*$$
this implies that
$$f'_! \pi_X^* \xrightarrow{(\epsilon^{\pi_Y,f})^{-1}} \pi_Y^* f_! \xrightarrow{\epsilon^{\pi_Y,f}} f'_! \pi_X^*$$
equals
\begin{align*}
    & f'_! \pi_X^* \xrightarrow{(\epsilon^{\pi_Y,f})^{-1}} \pi_Y^* f_! = \pi_Y^* \psi^* \pi_Y^* f_! \xrightarrow{\pi_Y^* \psi^*(\epsilon^{\pi_Y, f}))} \pi_Y^* \psi^* f'_! \pi_X^*\\
    & \xrightarrow{\pi_Y^*(\epsilon^{\psi, f'})\pi_X^*} \pi_Y^* f_! \varphi^* \pi_X^* \xrightarrow{(\epsilon^{\pi_Y,f})\varphi^* \pi_X^*} f'_! \pi_X^* \varphi^* \pi_X^* = f'_! \pi_X^*
\end{align*}
i.e.
$$f'_! \pi_X^* \xrightarrow{\r{id}} f'_! \pi_X^*$$
equals
$$\pi_Y^* \psi^* f'_! \pi_X^*  \xrightarrow{\pi_Y^*(\epsilon^{\psi, f'})\pi_X^*} \pi_Y^* f_! \varphi^* \pi_X^* \xrightarrow{(\epsilon^{\pi_Y,f})\varphi^* \pi_X^*} f'_! \pi_X^* \varphi^* \pi_X^* = f'_! \pi_X^*$$
and therefore \eqref{eq:1.003} equals
\begin{align}
    f^! = \varphi^* f^{\prime !} \pi_Y^* \xrightarrow{\varphi^*(\theta^{f'}) f^{\prime !} \pi_Y^*} \varphi^* f^{\prime !} f'_! f^{\prime !} \pi_Y^* \xrightarrow{\varphi^* f^{\prime !} (\nu^{f'}) \pi_Y^*} \varphi^* f^{\prime !} \pi_Y^* = f^! \label{eq:1.0035}
\end{align}
which is once again the identity by the adjunction $f'_! \dashv f^{\prime !}$.\\
\end{proof}

\begin{lem}\label{lem1.0021}
The natural isomorphism
$$\epsilon^{\pi_Y, f} : \pi_Y^* f_! \xrightarrow{\sim} f'_! \pi_X^*$$
is the identity.
\end{lem}

\begin{proof}
By \Cref{rmk3.1} we have that $f' = (f \times \r{id}_S) \circ \gamma$ where $\gamma: X \times S \xrightarrow{\sim} X \times S$ is an isomorphism. For the moment assume that $f' = f \times \r{id}_S$.

Working at the level of sheaves (so for the moment everything is underived) the isomorphism $\pi_Y^* f_! \xrightarrow{\sim} f'_! \pi_X^*$ is given in [\cite{ks} Proposition 2.5.11]. It is enough to check this is an equality on a basis for the topology of $Y \times S$. Let $\c F \in \r{Sh}(X)$ and $U \times W \subset Y \times S$ for $U \subset Y$ and $W \subset S$ open. Then because the maps $\pi_X$ and $\pi_Y$ are open we have
\begin{align*}
    \pi_Y^* f_! \c F(U \times W) &= f_! \c F(\pi_Y(U \times W))\\
    &= f_! \c F(U)\\
    &= \big\{t \in \c F(f^{-1}(U)) \,\,\, \big| \,\,\, f|_{\r{supp}(t)} \,\, \r{is proper} \big\}\\
    \pi_Y^* f_! \pi_{X *} \pi_X^* \c F(U \times W) &= f_! \pi_{X *} \pi_X^* \c F(\pi_Y(U \times W))\\
    &= f_! \pi_{X *} \pi_X^* \c F(U)\\
    &= \big\{r \in \pi_{X *} \pi_X^* \c F(f^{-1}(U)) \,\,\, \big| \,\,\, f|_{\r{supp}(r)} \,\, \r{is proper} \big\}\\
    &= \big\{s \in \pi_X^* \c F((\pi_X \circ f)^{-1}(U)) \,\,\, \big| \,\,\, f|_{\pi_X(\r{supp}(s))} \,\, \r{is proper} \big\}\\
    &= \big\{t \in \c F((\pi_X(f^{-1}(U) \times S)) \,\,\, \big| \,\,\, f|_{\r{supp}(t)} \,\, \r{is proper} \big\}\\
    &= \big\{t \in \c F(f^{-1}(U) ) \,\,\, \big| \,\,\, f|_{\r{supp}(t)} \,\, \r{is proper} \big\}\\
    \pi_Y^* \pi_{Y *} f'_! \pi_X^* \c F(U \times W) &= \pi_{Y *} f'_! \pi_X^* \c F(\pi_Y(U \times W))\\
    &= f'_! \pi_X^* \c F(\pi_Y^{-1}(U))\\
    &= \big\{s \in \pi_X^* \c F(f^{\prime -1}(U \times S)) \,\,\, \big| \,\,\, f'|_{\r{supp}(s)} \,\, \r{is proper} \big\}\\
    &= \big\{t \in \c F(f^{-1}(U)) \,\,\, \big| \,\,\, (f \times \r{id}_S)|_{\r{supp}(t) \times S} \,\, \r{is proper} \big\}\\
    &= \big\{t \in \c F(f^{-1}(U)) \,\,\, \big| \,\,\, f|_{\r{supp}(t)} \,\, \r{is proper} \big\}\\
    f'_! \pi_X^* \c F(U \times W) &= \big\{s \in \pi_X^* \c F(f^{\prime -1}(U \times W)) \,\,\, \big| \,\,\, f'|_{\r{supp}(s)} \,\, \r{is proper} \big\}\\
    &= \big\{t \in \c F(f^{-1}(U)) \,\,\, \big| \,\,\, (f \times \r{id}_S)|_{\r{supp}(t) \times S} \,\, \r{is proper} \big\}\\
    &= \big\{t \in \c F(f^{-1}(U)) \,\,\, \big| \,\,\, f|_{\r{supp}(t)} \,\, \r{is proper} \big\}
\end{align*}
where the final equalities for both $\pi_Y^* \pi_{Y *} f'_! \pi_X^* \c F(U \times W)$ and $f'_! \pi_X^* \c F(U \times W)$ arise due to the fact that properness is preserved under base extension. The isomorphism given in [\cite{ks} Proposition 2.5.11] is then the composition
\begin{align}
    \pi_Y^* f_! \longrightarrow \pi_Y^* f_! \pi_{X *} \pi_X^* \longrightarrow \pi_Y^* \pi_{Y *} f'_! \pi_X^* \longrightarrow f'_! \pi_X^* \label{eq:1.0036}
\end{align}
where the first morphism comes from the natural transformation $\r{id}_{\r{Sh}(X)} \rightarrow \pi_{X *} \pi_X^*$ for the adjunction $\pi_X^* \dashv \pi_{X *}$, the second morphism becomes the identity following the description in the proof of [\cite{ks} Proposition 2.5.11], and the third morphism comes from the natural transformation $\pi_Y^* \pi_{Y *} \rightarrow \r{id}_{\r{Sh}(Y \times S)}$ for the adjunction $\pi_Y^* \dashv \pi_{Y *}$. 

Note that for a projection $\pi: P \times Q \rightarrow P$ the natural transformation $\r{id}_{\r{Sh}(P)} \rightarrow \pi_* \pi^*$ is the identity, since for $\c F \in \r{Sh}(P)$ and $U \subset P$ open
\begin{align*}
    \pi_* \pi^* \c F(U) &=  \pi^* \c F(\pi^{-1}(U))\\
    &= \pi^* \c F(U \times T)\\
    &= \c F(\pi(U \times T))\\
    &= \c F(U)
\end{align*}
with restriction maps for $V \subset U$
\begin{align*}
    \r{res}^{\pi_* \pi^* \c F}_{U, V} &= \r{res}^{\pi^* \c F}_{U \times S, V \times S}\\
    &= \r{res}^{\c F}_{U, V}.
\end{align*} 
Hence the first morphism in \eqref{eq:1.0036} is also the identity. For $\c G \in \r{Sh}(Y \times S)$ and $U \times W \subset Y \times S$ open we have
$$\pi_Y^* \pi_{Y *} \c G(U \times W) = \c G(U \times S)$$
and so 
$$\sigma^{\pi_Y}(\c G(U \times W)) = \r{res}^{\c G}_{U \times S, U \times W}$$
hence
\begin{align*}
    \sigma^{\pi_Y}(f'_! \pi_X^*(\c F(U \times W)) &= \r{res}^{f'_! \pi_X^* \c F}_{U \times S, U \times W}\\
    &= \r{res}^{\pi_X^* \c F}_{f^{-1}(U) \times S, f^{-1}(U) \times W}|_{\r{sections with proper support}}\\
    &= \r{res}^{\c F}_{f^{-1}(U), f^{-1}(U)}|_{\r{sections with proper support}}\\
    &= \r{id}_{\c F(f^{-1}(U))}|_{\r{sections with proper support}}
\end{align*}
showing that on a basis \eqref{eq:1.0036} is the identity morphism.
Therefore, since by [\cite{ks} Proposition 2.6.7] the natural isomorphism $\epsilon^{\pi_Y,f}$ of derived functors is induced by the natural isomorphism \eqref{eq:1.0036} at the level of sheaves, upgrading to derived functors implies that $\epsilon^{\pi_Y,f}$ is the identity natural isomorphism too.

In the case that $f' = (f \times \r{id}_S) \circ \gamma$ we have
\begin{align*}
    f'_! (\pi_X \circ \gamma)^* = (f \times \r{id}_S)_! \gamma_! \gamma^* \pi_X^* = (f \times \r{id}_S)_! \pi_X^* = \pi_Y^* f_!
\end{align*}
because for an isomorphism $\gamma : X \times S \xrightarrow{\sim} X \times S$ it is clear that the functor
$$\gamma_! \gamma^* : \r{D}(X \times S) \rightarrow \r{D}(X \times S)$$
is simply the identity functor $\r{id}_{\r{D}(X \times S)}$.
\end{proof}

\begin{lem}\label{lem1.0022}
The composition of natural transformations
\begin{align}
    \varphi^! f^{\prime !}\Big( f'_! \pi_X^! \xrightarrow{(\theta^{\pi_Y}) f'_! \pi_X^!} \pi_Y^! \pi_{Y !} f'_! \pi_X^! = \pi_Y^! f_! \pi_{X !} \pi_X^! \xrightarrow{\pi_Y^! f_!(\nu^{\pi_X})} \pi_Y^! f_! \Big) \label{eq:1.00385}
\end{align}
is the identity.
\end{lem}

\begin{proof}
To show this natural transformation is the identity it suffices to show its Verdier dual
\begin{align}
    \varphi^* f^{\prime *}\Big(\pi_Y^* f_* \xrightarrow{\pi_Y^* f_*(\eta^{\pi_X})} \pi_Y^* f_* \pi_{X *} \pi_X^* = \pi_Y^* \pi_{Y *} f'_* \pi_X^* \xrightarrow{(\sigma^{\pi_Y}) f'_* \pi_X^*} f'_* \pi_X^* \Big). \label{eq:1.0039}
\end{align}
is the identity. But showing this can be done using almost exactly the same method as the proof of \Cref{lem1.0021}, with the only changes being that we do not consider proper supports and the middle morphism here is explicitly the identity.
\end{proof}

\begin{prop}\label{prop1.003}
The canonical natural transformation $\nu^f: f_! f^! \rightarrow \r{id}_{\r{D}(Y)}$ from the adjunction $f_! \dashv f^!$ is equal to the composition \eqref{eq:1.00305} i.e.
$$f_! f^! = f_! \varphi^* f^{\prime !} \pi_Y^* \xrightarrow{(\epsilon^{\psi,f'})^{-1} f^{\prime !} \pi_Y^*} \psi^* f'_! f^{\prime !} \pi_Y^* \xrightarrow{\psi^*(\nu^{f'}) \pi_Y^*} \psi^* \pi_Y^* = \r{id}_{\r{D}(Y)}$$
\end{prop}

\begin{proof}
We shall show that \eqref{eq:1.00315} is equal to the unit $\theta^f$ implying that the counits must also be equal, namely \eqref{eq:1.00305} equals $\nu^f$.

First note the following: from the adjunction $\pi_{X !} \dashv \pi_X^!$ we have that
$$\pi_X^! \xrightarrow{\pi_X^!(\theta^{\pi_X})} \pi_X^! \pi_{X !} \pi_X^! \xrightarrow{(\nu^{\pi_X}) \pi_X^!} \pi_X^!$$
is the identity natural transformation. Applying $\varphi^!$ then post-composing with $\theta^f$ gives
\begin{align}
    \r{id}_{\r{D}(X)} = \varphi^! \pi_X^! \xrightarrow{\varphi^! \pi_X^!(\theta^{\pi_X})} \varphi^! \pi_X^! \pi_{X !} \pi_X^! \xrightarrow{\varphi^!(\nu^{\pi_X}) \pi_X^!} \varphi^! \pi_X^! = \r{id}_{\r{D}(X)} \xrightarrow{\theta^f} f^! f_!. \label{eq:1.00365}
\end{align}
Because these are natural transformations we can rearrange the order in which we do them (using a similar argument as in the proof of \Cref{lem1.002}) making \eqref{eq:1.00365} equal to
\begin{align}
    \r{id}_{\r{D}(X)} = \varphi^! \pi_X^! \xrightarrow{\varphi^! \pi_X^!(\theta^{\pi_X})} \varphi^! \pi_X^! \pi_{X !} \pi_X^! \xrightarrow{\varphi^! \pi_X^!(\theta^f) \pi_{X !} \pi_X^!} \varphi^! \pi_X^! f^! f_! \pi_{X !} \pi_X^! \xrightarrow{\varphi^! \pi_X^! f^! f_!(\nu^{\pi_X})} \varphi^! \pi_X^! f^! f_! =  f^! f_!. \label{eq:1.0037}
\end{align}
Then because
$$\pi_X^!(\theta^f) \pi_{X !} \circ \theta^{\pi_X} = \theta^{f \circ \pi_X} = \theta^{\pi_Y \circ f'} = f^{\prime !} (\theta^{\pi_Y}) f_! \circ \theta^{f'}$$
we get that \eqref{eq:1.0037} equals
\begin{align}
    & \r{id}_{\r{D}(X)} = \varphi^! \pi_X^! \xrightarrow{\varphi^!(\theta^{f'}) \pi_X^!} \varphi^! f^{\prime !} f'_! \pi_X^! \xrightarrow{\varphi^! f^{\prime !} (\theta^{\pi_Y}) f'_! \pi_X^!} \varphi^! f^{\prime !} \pi_Y^! \pi_{Y !} f'_! \pi_X^! \label{eq:1.0038}\\
    & = \varphi^! f^{\prime !} \pi_Y^! f_! \pi_{X !} \pi_X^! \xrightarrow{\varphi^! f^{\prime !} \pi_Y^! f_!(\nu^{\pi_X})} \varphi^! \pi_X^! f^! f_! =  f^! f_!. \nonumber
\end{align}
so in particular $\theta^f$ equals \eqref{eq:1.0038}.

We construct the following diagram whose top row is \eqref{eq:1.00315} and whose bottom row is \eqref{eq:1.0038}.
\begin{equation}\label{fig:1.0031}
\begin{tikzcd}[column sep = 6em, row sep = 5em]
\r{id}_{\r{D}(X)} = \varphi^*[-2m] \pi_X^*[2m] \arrow[d, "(\kappa^\varphi)\pi_X^*{[2m]}"] \arrow[r, "\varphi^*(\theta^{f'}) \pi_X^*"] & \varphi^*[-2m] f^{\prime !} f'_! \pi_X^*[2m] \arrow[d, "(\kappa^\varphi)f^{\prime !} f'_! \pi_X^*{[2m]}"] \arrow[r, "\varphi^* f^{\prime !}(\epsilon^{\pi_Y, f})^{-1}"] & \varphi^*[-2m] f^{\prime !} \pi_Y^*[2m] f_! \arrow[d, "(\kappa^\varphi)f^{\prime !} \pi_Y^*{[2m]} f_!"]\\
\varphi^! \pi_X^*[2m] \arrow[d, "\varphi^!(\kappa^{\pi_X})", "="' sloped] \arrow[r, "\varphi^!(\theta^{f'}) \pi_X^*{[2m]}"] & \varphi^! f^{\prime !} f'_! \pi_X^*[2m] \arrow[d, "\varphi^! f^{\prime !} f'_!(\kappa^{\pi_X})", "="' sloped] \arrow[r, "\varphi^! f^{\prime !}(\epsilon^{\pi_Y, f}{[2m]})^{-1}"] & \varphi^! f^{\prime !} \pi_Y^*[2m] f_! \arrow[d, "\varphi^! f^{\prime !}(\kappa^{\pi_Y}) f_!", "="' sloped]\\
\r{id}_{\r{D}(X)} = \varphi^! \pi_X^! \arrow[r, "\varphi^!(\theta^{f'}) \pi_X^!"] & \varphi^! f^{\prime !} f'_! \pi_X^! \arrow[r, "\varphi^! f^{\prime !} (\xi^{\pi_X, f'})"] & \varphi^! f^{\prime !} \pi_Y^! f_!
\end{tikzcd}    
\end{equation}
where $\xi^{\pi_X, f'}: f'_! \pi_X^! \rightarrow \pi_Y^! f_!$
is the composition
$$f'_! \pi_X^! \xrightarrow{(\theta^{\pi_Y}) f'_! \pi_X^!} \pi_Y^! \pi_{Y !} f'_! \pi_X^! = \pi_Y^! f_! \pi_{X !} \pi_X^! \xrightarrow{\pi_Y^! f_!(\nu^{\pi_X})} \pi_Y^! f_!.$$
The composition of the natural transformations in the left-hand column is the identity since 
$$\varphi^!(\kappa^{\pi_X}) \circ (\kappa^\varphi) \pi_X^*[2m] = \kappa^{\pi_X \circ \varphi} =  \kappa^{\r{id}_X} = \r{id}.$$
For the right-hand column we have that
$$(\kappa^\varphi) f^{\prime !} \pi_Y^*[2m] f_! = (\kappa^\varphi) \pi_X^*[2m] f^! f_!$$
and
$$\varphi^! f^{\prime !}(\kappa^{\pi_Y}) f_! = \varphi^!(\kappa^{\pi_X}) f^! f_!$$
(simply using the fact that $\kappa^{\pi_X} = \r{id}$ and $\kappa^{\pi_Y}= \r{id}$) and so, for the same reason as the left-hand column, the composition in the right-hand column is also the identity.
Therefore it suffices to show this diagram commutes in order to prove that the two units are equal and hence conclude the proof of the proposition. Commutativity of the top two squares and bottom left square in \eqref{fig:1.0031} are clear because once evaluated on some $\c F \in \r{D}(X)$ these once again become the squares obtained by applying the natural transformations to some morphism, and hence they must commute by definition of a natural transformation. Finally for the bottom right square; we already know that $\kappa^{\pi_X}$ and $\kappa^{\pi_Y}$ are just the identity natural transformations, and by \Cref{lem1.0021} and \Cref{lem1.0022} we have that the two horizontal arrows in this square are the identity as well, and hence it also commutes.
\end{proof}
\vspace{3em}
Now consider the following commutative diagram
\begin{equation}\label{fig:1.004}
\begin{tikzcd}
& X' \arrow[hookleftarrow]{dd}[near end]{j'} \arrow[rr, "f'"] && Y' \arrow[hookleftarrow]{dd}{i'}\\
    X \arrow[hookleftarrow]{dd}{j} \arrow[ru, hook, "\varphi"] \arrow[rr, "f" near end, crossing over] && Y \arrow[ru, hook, "\psi"]\\
& V' \arrow[rr, "\widetilde{f'}" near start] && Z'\\
V \arrow[ru, "\widetilde{\varphi}", "\sim"' sloped] \arrow[rr, "\widetilde{f}"] && Z \arrow[uu, hook', "i"' near end, crossing over] \arrow[ru, "\widetilde{\psi}", "\sim"' sloped]
\end{tikzcd}
\end{equation}
where 
\begin{itemize}
    \item the top face of this cube fits as the bottom square into a diagram as in \eqref{fig:1.001} and satisfies all the conditions proceeding \eqref{fig:1.001}\\
    \item $i,i',j,j'$ are closed immersions\\
    \item $\widetilde{\varphi}$ and $\widetilde{\psi}$ are isomorphisms\\
    \item all squares are Cartesian.
\end{itemize}

\vspace{2em}

We wish to show that the morphisms
$$\widetilde{\psi}_!(f_\star): \widetilde{\psi}_! \widetilde{f}_! \m Q_{V}[2d] \longrightarrow \widetilde{\psi}_! \m Q_Z$$
and 
$$f'_\star: \widetilde{f}'_! \m Q_{V'}[2d] \longrightarrow \m Q_{Z'}$$
are equal via the isomorphisms on sheaves between $V$ and $V'$ and between $Z$ and $Z'$ induced by pulling-back along $\widetilde{\varphi}$ and $\widetilde{\psi}$ respectively. Using the definition \eqref{eq:1.000005} from Section 3.1 we have that $\widetilde{\psi}_!(f_\star)$ is given by
\begin{align}
    & \widetilde{\psi}_! \widetilde{f}_! \m Q_{V}[2d] = \widetilde{\psi}_! \widetilde{f}_! j^* \m Q_{X}[2d] \xrightarrow[\sim]{\widetilde{\psi}_!(\epsilon^{i,f}(\m Q_{X}[2d])^{-1})} \widetilde{\psi}_! i^* f_! \m Q_{X}[2d] \label{eq:1.005}\\
    & = \widetilde{\psi}_! i^* f_! f^! \m Q_{Y} \xrightarrow{\widetilde{\psi}_!i^*(\nu^f(\m Q_{Y}))} \widetilde{\psi}_! i^* \m Q_{Y} = \widetilde{\psi}_! \m Q_Z \nonumber  
\end{align}
and $f'_\star$ is given by
\begin{align}
   & \widetilde{f}'_! \m Q_{V'}[2d] = \widetilde{f}'_! j^{\prime *} \m Q_{X'}[2d] \xrightarrow[\sim]{\epsilon^{i',f'}(\m Q_{X'}[2d])^{-1}} i^{\prime *} f'_! \m Q_{X'}[2d] \label{eq:1.006}\\
   & = i^{\prime *} f'_! f^{\prime !} \m Q_{Y'} \xrightarrow{i'^*(\nu^{f'}(\m Q_{Y'}))} i'^* \m Q_{Y'} = \m Q_{Z'}. \nonumber
\end{align}

\begin{thm}\label{thm1.04}
The following diagram of sheaves in $\r{D}(Z')$ commutes
\begin{equation}\label{fig:1.007}
\begin{tikzcd}[column sep = 7em, row sep=6em]
\widetilde{f}'_! \m Q_{V'}[2d] \arrow[r, "\widetilde{f}'_!(\widetilde{\varphi}^\star)", "\sim"'] \arrow[d, "f'_\star"] & \widetilde{\psi}_! \widetilde{f}_! \m Q_{V}[2d] \arrow[d, "\widetilde{\psi}_!(f_\star)"]\\
\m Q_{Z'} \arrow[r, "\widetilde{\psi}^\star", "\sim"'] & \widetilde{\psi}_! \m Q_{Z}
\end{tikzcd}
\end{equation}
\end{thm}
Using the descriptions \eqref{eq:1.005} and \eqref{eq:1.006} and the description of the pullback on compactly supported cohomology given in Section 3.1, we get that diagram \eqref{fig:1.007} is equal to the outer square of the following diagram and hence to prove \Cref{thm1.04} it suffices to show that \eqref{fig:1.008} commutes.
\begin{equation}\label{fig:1.008}
\begin{tikzcd}[column sep = 8em, row sep=8em]
\widetilde{f}'_! j^{\prime *} f^{\prime !} \m Q_{Y'} \arrow[r, "\widetilde{f}'_!(\eta^{\widetilde{\varphi}}(f^{\prime !} \m Q_{Y'}))", "\sim"'] \arrow[d, "\epsilon^{i',f'}(f^{\prime !} \m Q_{Y'})^{-1}", "\sim"' sloped] & \widetilde{f}'_! \widetilde{\varphi}_! \widetilde{\varphi}^* j^{\prime *} f^{\prime !} \m Q_{Y'} = \widetilde{\psi}_! \widetilde{f}_! j^* f^! \m Q_{Y} \arrow[d, "\widetilde{\psi}_!(\epsilon^{i,f}(f^! \m Q_{Y})^{-1})", "\sim"' sloped]\\
i^{\prime *} f'_! f^{\prime !} \m Q_{Y'} \arrow[r, "\alpha", "\sim "'] \arrow[d, "i^{\prime *}(\nu^{f'}(\m Q_{Y'}))"] & \widetilde{\psi}_! i^* f_! f^! \m Q_{Y} \arrow[d, "\widetilde{\psi}_! i^*(\nu^f(\m Q_{Y}))"]\\
i^{\prime *} \m Q_{Y'} = \m Q_{Z'} \arrow[r, "\eta^{\widetilde{\psi}}(\m Q_{Z'})", "\sim"'] & \widetilde{\psi}_! \widetilde{\psi}^* \m Q_{Z'} = \widetilde{\psi}_! i^* \m Q_Y
\end{tikzcd}
\end{equation}
where $\alpha$ is the isomorphism
\begin{align*}
    & i^{\prime *} f'_! f^{\prime !} \m Q_{Y'} \xrightarrow[\sim]{\eta^{\widetilde{\psi}}(i^{\prime *} f'_! f^{\prime !}\m Q_{Y'})} \widetilde{\psi}_! \widetilde{\psi}^* i^{\prime *} f'_! f^{\prime !} \m Q_{Y'} = \widetilde{\psi}_! \widetilde{\psi}^* (\widetilde{\psi}^{-1})^* i^* \psi^* f'_! f^{\prime !} \m Q_{Y'} =\\
    & \widetilde{\psi}_! i^* \psi^* f'_! f^{\prime !} \m Q_{Y'} \xrightarrow[\sim]{\widetilde{\psi}_! i^*(\epsilon^{\psi,f'}(f^{\prime !} \m Q_{Y'}))} \widetilde{\psi}_! i^* f_! \varphi^* f^{\prime !} \m Q_{Y'} = \widetilde{\psi}_! i^* f_! \varphi^* f^{\prime !} \pi_Y^* \m Q_{Y} = \widetilde{\psi}_! i^* f_! f^! \m Q_{Y}.
\end{align*}

\begin{lem}\label{lem1.051}
The top square in diagram \eqref{fig:1.008} commutes.
\end{lem}

\begin{proof}
This is just a simple matter of rearranging the order we do the morphisms that comprise either side of the square and cancelling.
In particular by applying the natural transformation $\eta^{\widetilde{\psi}}$ to the morphism $\epsilon^{i',f'}(f^{\prime !} \m Q_{Y'})^{-1}$ we get the following commutative square
\begin{equation}\label{fig:1.0081}
\begin{tikzcd}[column sep = 9em, row sep=6em]
\widetilde{f}'_! j^{\prime *} f^{\prime !} \m Q_{Y'} \arrow[d, "\eta^{\widetilde{\psi}}(\widetilde{f}'_! j^{\prime *} f^{\prime !} \m Q_{Y'})"] \arrow[r, "\epsilon^{i',f'}(f^{\prime !} \m Q_{Y'})^{-1}"] & i^{\prime *} f'_! f^{\prime !} \m Q_{Y'} \arrow[d, "\eta^{\widetilde{\psi}}(i^{\prime *} f'_! f^{\prime !} \m Q_{Y'})"]\\
\widetilde{\psi}_! \widetilde{\psi}^* \widetilde{f}'_! j^{\prime *} f^{\prime !} \m Q_{Y'} \arrow[r, "\widetilde{\psi}_! \widetilde{\psi}^*(\epsilon^{i',f'}(f^{\prime !} \m Q_{Y'})^{-1})"] & \widetilde{\psi}_! \widetilde{\psi}^* i^{\prime *} f'_! f^{\prime !} \m Q_{Y'} = \widetilde{\psi}_! i^* \psi^* f'_! f^{\prime !} \m Q_{Y'}
\end{tikzcd}
\end{equation}
Hence \eqref{fig:1.0081} tells us that going along the left-hand and bottom sides of the top square in \eqref{fig:1.008} equals
\begin{align}
    & \widetilde{f}'_! j^{\prime *} f^{\prime !} \m Q_{Y'} \xrightarrow{\eta^{\widetilde{\psi}}(\widetilde{f}'_! j^{\prime *} f^{\prime !} \m Q_{Y'})} \widetilde{\psi}_! \widetilde{\psi}^* \widetilde{f}'_! j^{\prime *} f^{\prime !} \m Q_{Y'} \xrightarrow{\widetilde{\psi}_! \widetilde{\psi}^*(\epsilon^{i',f'}(f^{\prime !} \m Q_{Y'})^{-1})} \label{eq:1.011}\\
    & \widetilde{\psi}_! \widetilde{\psi}^* i^{\prime *} f'_! f^{\prime !} \m Q_{Y'} = \widetilde{\psi}_! i^* \psi^* f'_! f^{\prime !} \m Q_{Y'} \xrightarrow{\widetilde{\psi}_! i^*(\epsilon^{\psi,f'}(f^{\prime !} \m Q_{Y'}))} \widetilde{\psi}_! i^* f_! \varphi^* f^{\prime !} \m Q_{Y'}. \nonumber
\end{align}
Then because
$$\epsilon^{i',f'} = \epsilon^{\psi \circ i \circ \widetilde{\psi}^{-1}, f'} = (\epsilon^{\widetilde{\psi}^{-1}, \widetilde{f}})j^* \varphi^* \circ (\widetilde{\psi}^{-1})^*(\epsilon^{i,f}) \varphi^* \circ (\widetilde{\psi}^{-1})^* i^*(\epsilon^{\psi,f'})$$
we get that \eqref{eq:1.011} equals
\begin{align}
    & \widetilde{f}'_! j^{\prime *} f^{\prime !} \m Q_{Y'} \xrightarrow{\eta^{\widetilde{\psi}}(\widetilde{f}'_! j^{\prime *} f^{\prime !} \m Q_{Y'})} \widetilde{\psi}_! \widetilde{\psi}^* \widetilde{f}'_! j^{\prime *} f^{\prime !} \m Q_{Y'} = \label{eq:1.012}\\ 
    & \widetilde{\psi}_! \widetilde{\psi}^* \widetilde{f}'_! (\widetilde{\varphi}^{-1})^* j^* \varphi^* f^{\prime !} \m Q_{Y'} \xrightarrow{\widetilde{\psi}_! \widetilde{\psi}^*(\epsilon^{\widetilde{\psi}^{-1},\widetilde{f}}(j^* \varphi^* f^{\prime !} \m Q_{Y'})^{-1})} \widetilde{\psi}_! \widetilde{\psi}^* (\widetilde{\psi}^{-1})^* \widetilde{f}_! j^* \varphi^* f^{\prime !} \m Q_{Y'} = \nonumber\\
    &  \widetilde{\psi}_! \widetilde{f}_! j^* \varphi^* f^{\prime !} \m Q_{Y'} \xrightarrow{ \widetilde{\psi}_!(\epsilon^{i,f}(\varphi^* f^{\prime !} \m Q_{Y'})^{-1})}  \widetilde{\psi}_! i^* f_! i^* \varphi^* f^{\prime !} \m Q_{Y'} \nonumber
\end{align}
Now recall the description of $\epsilon^{\widetilde{\psi}^{-1},\widetilde{f}}$ given in \eqref{eq:1.0001}, namely it equals
$$(\widetilde{\psi}^{-1})^* \widetilde{f}_! \xrightarrow{(\widetilde{\psi}^{-1})^* \widetilde{f}_!(\eta^{\widetilde{\varphi}^{-1}})} (\widetilde{\psi}^{-1})^* \widetilde{f}_! (\widetilde{\varphi}^{-1})_! (\widetilde{\varphi}^{-1})^* = (\widetilde{\psi}^{-1})^* (\widetilde{\psi}^{-1})_! \widetilde{f}'_! (\widetilde{\varphi}^{-1})^* \xrightarrow{(\sigma^{\widetilde{\psi}^{-1}})\widetilde{f}'_! (\widetilde{\varphi}^{-1})^*} \widetilde{f}'_! (\widetilde{\varphi}^{-1})^*.$$
But since
$$(\widetilde{\varphi}^{-1})_! (\widetilde{\varphi}^{-1})^* \xrightarrow{(\eta^{\widetilde{\varphi}^{-1}})^{-1}} \r{id}_{\r{D}(V)}$$
equals
$$(\widetilde{\varphi}^{-1})_! (\widetilde{\varphi}^{-1})^* \xrightarrow{(\widetilde{\varphi}^{-1})_!(\eta^{\widetilde{\varphi}}) (\widetilde{\varphi}^{-1})^*} (\widetilde{\varphi}^{-1})_! \widetilde{\varphi}_! \widetilde{\varphi}^* (\widetilde{\varphi}^{-1})^* = \r{id}_{\r{D}(V)}$$
and similarly
$$\r{id}_{\r{D}(Z')} \xrightarrow{(\sigma^{\widetilde{\psi}^{-1}})^{-1}} (\widetilde{\psi}^{-1})^* (\widetilde{\psi}^{-1})_!$$
equals
$$\r{id}_{\r{D}(Z')} = (\widetilde{\psi}^{-1})^* \widetilde{\psi}^* \widetilde{\psi}_! (\widetilde{\psi}^{-1})_! \xrightarrow{(\widetilde{\psi}^{-1})^*(\sigma^{\widetilde{\psi}})(\widetilde{\psi}^{-1})_!} (\widetilde{\psi}^{-1})^* (\widetilde{\psi}^{-1})_!$$
we get that \eqref{eq:1.012} equals
\begin{align}
    &  \widetilde{f}'_! j^{\prime *} f^{\prime !} \m Q_{Y'} \xrightarrow{\eta^{\widetilde{\psi}}(\widetilde{f}'_! j^{\prime *} f^{\prime !} \m Q_{Y'})}  \widetilde{\psi}_! \widetilde{\psi}^* \widetilde{f}'_! j^{\prime *} f^{\prime !} \m Q_{Y'} =  \widetilde{\psi}_! \widetilde{\psi}^* \widetilde{\psi}_! (\widetilde{\psi}^{-1})_! \widetilde{f}'_! (\widetilde{\varphi}^{-1})^* j^* \varphi^* f^{\prime !} \m Q_{Y'} \label{eq:1.013}\\
    & \xrightarrow{\widetilde{\psi}_!(\sigma^{\widetilde{\psi}}( (\widetilde{\psi}^{-1})_! \widetilde{f}_! (\widetilde{\varphi}^{-1})^* j^* \varphi^* f^{\prime !} \m Q_{Y'}))}  \widetilde{\psi}_! (\widetilde{\psi}^{-1})_! \widetilde{f}'_! (\widetilde{\varphi}^{-1})^* j^* \varphi^* f^{\prime !} \m Q_{Y'} =  \widetilde{\psi}_! \widetilde{f}_! (\widetilde{\varphi}^{-1})_! (\widetilde{\varphi}^{-1})^* j^* \varphi^* f^{\prime !} \m Q_{Y'} \nonumber\\
    & \xrightarrow{ \widetilde{\psi}_! \widetilde{f}_! (\widetilde{\varphi}^{-1})_!(\eta^{\widetilde{\varphi}}((\widetilde{\varphi}^{-1})^* j^* \varphi^* f^{\prime !} \m Q_{Y'}))}  \widetilde{\psi}_! \widetilde{f}_! (\widetilde{\varphi}^{-1})_! \widetilde{\varphi}_! \widetilde{\varphi}^* (\widetilde{\varphi}^{-1})^* j^* \varphi^* f^{\prime !} \m Q_{Y'} \nonumber\\
    &=  \widetilde{\psi}_! \widetilde{f}_! j^* \varphi^* f^{\prime !} \m Q_{Y'} \xrightarrow{ \widetilde{\psi}_! i^*(\epsilon^{i,f}(\varphi^* f^{\prime !} \m Q_{Y'})^{-1})}  \widetilde{\psi}_! i^* f_! \varphi^* f^{\prime !} \m Q_{Y'}. \nonumber
\end{align}
The composition of first two morphisms in \eqref{eq:1.013} is the identity due to the adjunction $\widetilde{\psi}^* \dashv \widetilde{\psi}_*$ and so we are left with the top and right-hand sides of the top square in \eqref{fig:1.008}.
\end{proof}

\begin{lem}\label{lem1.052}
The bottom square in diagram \eqref{fig:1.008} commutes.
\end{lem}

\begin{proof}
This follows from \Cref{prop1.003} and more rearranging. Indeed \Cref{prop1.003} says that
$$\nu^f = \psi^*(\nu^{f'}) \pi_Y^* \circ (\epsilon^{\psi, f'})^{-1} f^{\prime !} \pi_Y^*$$
and so we can immediately see that the composition of top and right-hand sides of the bottom square in \eqref{fig:1.008} gives
\begin{align}
    &  i^{\prime *} f'_! f^{\prime !} \m Q_{Y'} \xrightarrow{\eta^{\widetilde{\varphi}}(i^{\prime *} f'_! f^{\prime !} \m Q_{Y'})}  \widetilde{\psi}_! \widetilde{\psi}^* i^{\prime *} f'_! f^{\prime !} \m Q_{Y'} =  \widetilde{\psi}_! i^* \psi^* f'_! f^{\prime !} \m Q_{Y'} \label{eq:1.014}\\
    & \xrightarrow{ \widetilde{\psi}_! i^* \psi^*(\nu^{f'}(\m Q_{Y'}))}   \widetilde{\psi}_! i^* \psi^* \m Q_{Y'} =  \widetilde{\psi}_! \m Q_{Z}. \nonumber
\end{align}
Now applying the natural transformation $\eta^{\widetilde{\psi}}$ to the morphism $i^{\prime *} f'_! f^{\prime !} \m Q_{Y'} \xrightarrow{i^{\prime *}(\nu^{f'}(\m Q_{Y'}))} i^{\prime *} \m Q_{Y'}$ gives the commutative square
\begin{equation}\label{fig:1.0082}
\begin{tikzcd}[column sep = 8em, row sep=6em]
i^{\prime *} f'_! f^{\prime !} \m Q_{Y'} \arrow[d, "i^{\prime *}(\nu^{f'}(\m Q_{Y'}))"] \arrow[r, "\eta^{\widetilde{\psi}}(i^{\prime *} f'_! f^{\prime !} \m Q_{Y'})"] & \widetilde{\psi}_! \widetilde{\psi}^* i^{\prime *} f'_! f^{\prime !} \m Q_{Y'} \arrow[d, "\widetilde{\psi}_! i^* \psi^*(\nu^{f'}(\m Q_{Y'}))"]\\
i^{\prime *} \m Q_{Y'} \arrow[r, "\eta^{\widetilde{\psi}}(i^{\prime *} \m Q_{Y'})"] & \widetilde{\psi}_! \widetilde{\psi}^* i^{\prime *} \m Q_{Y'}
\end{tikzcd}
\end{equation}
and we can see that the left-hand and bottom sides of \eqref{fig:1.0082} are exactly those in the bottom square in \eqref{fig:1.008} whilst the top and right-hand sides are \eqref{eq:1.014} proving the result.
\end{proof}

The two preceding lemmas then directly imply \Cref{thm3.6}. We end this section with a more general reformulation of \Cref{thm3.6}, in which we slightly generalise what the inclusions $\psi$ and $\varphi$ can look like, and more importantly we only require the squares in diagram \eqref{fig:1.001} to be Cartesian up to an isomorphism of $Y \times S$. This will allow us to more easily apply the results in this section to our character variety CoHAs.

\begin{cor}\label{cor1.06}
Consider diagram \eqref{fig:1.004} and let $\alpha:Y \xrightarrow{\sim} Y$ and $\beta:X \xrightarrow{\sim} X$ be isomorphisms, and let $\delta: Y \times S \xrightarrow{\sim} Y \times S$ be an isomorphism such that the restriction of $\delta$ to $Z'$ via $i'$ is $\r{id}_{Z'}$. Define
\begin{align*}
    & \psi'= \psi \circ \alpha\\
    & \varphi'= \varphi \circ \beta\\
    & \pi_Y'=\alpha^{-1} \circ \pi_Y\\
    & \pi_X'= \beta^{-1} \circ \pi_X \\
    & f''= \delta \circ f'.
\end{align*}
Suppose that in diagram \eqref{fig:1.001} we replace $\psi$ with $\psi'$, $\varphi$ with $\varphi'$, $\pi_Y$ with $\pi'_Y$, and $\pi_X$ with $\pi'_X$, and in \eqref{fig:1.004} we replace $\psi$ with $\psi'$, and $\varphi$ with $\varphi'$. Assume that all of the conditions subsequent to \eqref{fig:1.004} hold with these replacements. Then the diagram
\begin{equation}\label{fig:1.009}
\begin{tikzcd}[column sep = 7em, row sep=6em]
\widetilde{f}'_! \m Q_{V'}[2d] \arrow[r, "\widetilde{f}'_!(\widetilde{\varphi}^{\prime \star})", "\sim"'] \arrow[d, "f''_\star"] & \widetilde{\psi}'_! \widetilde{f}_! \m Q_{V}[2d] \arrow[d, "\widetilde{\psi}'_!(f_\star)"]\\
\m Q_{Z'} \arrow[r, "\widetilde{\psi}^{\prime \star}", "\sim"'] & \widetilde{\psi}'_! \m Q_{Z}
\end{tikzcd}
\end{equation}
commutes.
\end{cor}

\begin{rmk}
Note that in \eqref{fig:1.009} we are using the morphism induced by $f''$ instead of $f'$. Also if we let $\widetilde{f}''$ denote the restriction of $f''$ to $V'$ then since $\delta|_{Z'} =  \r{id}_{Z'}$ we have that
$$\widetilde{f}'' = (\delta \circ f')|_{V'} = \delta_{Z'} \circ f'_{V'} = \widetilde{f}'.$$
\end{rmk}

\begin{proof}
First consider the case in which $\delta= \r{id}_{Y \times S}$. Then the proofs for \Cref{lem1.002}, \Cref{lem1.0021}, \Cref{lem1.0022}, \Cref{prop1.003}, \Cref{lem1.051} and \Cref{lem1.052} all follow when replacing $\psi$ and $\varphi$ with $\psi \circ \alpha$ and $\varphi \circ \beta$ and $\pi_Y$ and $\pi_X$ with $\alpha^{-1} \circ \pi_Y$ and $\beta^{-1} \circ \pi_X$.

If $\delta$ is arbitrary then for $f'' = \delta \circ f'$ the pushforward $f''_\star$ is given by the composition
\begin{align*}
    &  \widetilde{f}''_! j^{\prime *} f^{\prime \prime !} \m Q_{Y'}  =  \r{id}_{Z' !} \widetilde{f}'_! j^{\prime *} f^{\prime \prime !} \m Q_{Y'} \xrightarrow{ \r{id}_{Z' !}(\epsilon^{i',f'}(f^{\prime \prime !} \m Q_{Y'})^{-1})}   \r{id}_{Z' !} i^{\prime *} f'_! f^{\prime \prime !} \m Q_{Y'} \xrightarrow{\epsilon^{i',\delta}(f'_! f^{\prime \prime !} \m Q_{Y'})^{-1}}\\
    &   i^{\prime *} \delta_! f'_! f^{\prime \prime !} \m Q_{Y'} =  i^{\prime *} \delta_! f'_! f^{\prime !} \delta^! \m Q_{Y'} \xrightarrow{ i^{\prime *} \delta_!(\nu^{f'}(\delta^! \m Q_{Y'}))}  i^{\prime *} \delta_! \delta^! \m Q_{Y'} \xrightarrow{ i^{\prime *}(\nu^\delta(
    \m Q_{Y'}))}  i^{\prime *} \m Q_{Y'}.
\end{align*}
Hence if we can show that $\epsilon^{i',\delta}$ and $\nu^\delta$ are equal to the identity natural transformations then we are done by \Cref{prop1.003} and the proofs of \Cref{lem1.051} and \Cref{lem1.052}, as we have reduced to the case of $f'$. Clearly $\nu^\delta$ is the identity transformation because for an isomorphism $\delta$ all the units and counits for both adjunctions can be taken to be the identity. As for $\epsilon^{i', \delta}$, we again initially work at the level of sheaves and underived functors and consult [\cite{ks} Proposition 2.5.11]. The natural isomorphism $i^{\prime *} \delta_! \xrightarrow{\sim} \r{id}_{Z' !} i^{\prime *}$ is the composition
$$i^{\prime *} \delta_! \longrightarrow i^{\prime *} \delta_! i'_* i^{\prime *} \longrightarrow i^{\prime *} i'_* \r{id}_{Z' !} i^{\prime *} \longrightarrow \r{id}_{Z' !}  i^{\prime *}$$
where the first morphism comes from the natural transformation $\r{id}_{\r{Sh}(Y \times S)} \rightarrow i'_* i^{\prime *}$ and the third morphism comes from the natural transformation $i^{\prime *} i'_* \rightarrow \r{id}_{\r{Sh}(Z')}$ for the adjunction $i^{\prime *} \dashv i'_*$. For the second morphism, since $\delta$ is an isomorphism and so $\delta_! = \delta_*$, the natural transformation 
$$\delta_! i'_* \rightarrow i'_* \r{id}_{Z' !}$$
described in the proof of [\cite{ks} Proposition 2.5.11] is the identity. Similarly to the derived case, the natural transformation $i^{\prime *} i'_* \rightarrow \r{id}_{\r{Sh}(Z')}$ is the identity natural isomorphism since $i'$ is a closed immersion. Hence it remains to consider $i^{\prime *} \delta_!(\eta^{i'})$. We have that for $\c F \in \r{Sh}(Y')$
$$(i'_* i^{\prime *} \c F)_p =
\begin{cases}
    \c F_p \quad \r{if}\,\, p \in Z'\\
    0 \quad \,\,\,\, \r{if}\,\, p \in Y' \setminus Z' 
\end{cases}
$$
and so
$$\eta^{i'}(\c F)_p = 
\begin{cases}
    \r{id}_{\c F_p} \quad \r{if}\,\, p \in Z'\\
    0 \quad \quad \,\, \r{if}\,\, p \in Y' \setminus Z'.
\end{cases}
$$
Therefore, because $\delta|_{Z'} = \r{id}_{Z'}$, by taking stalks at $p \in Z'$ we get that $i^{\prime *} \delta_!(\eta^{i'})$ is also the identity. Upgrading to the level of derived categories then implies that the natural isomorphism $\epsilon^{i',\delta}$ is also the identity.
\end{proof}

\section{An isomorphism of 2D CoHAs for the character variety}

We now utilise the results from Section 4 to show that the two multiplications on the dual of the compactly supported cohomology of the character variety described in Section 3 are equal, and in the process show that the multiplication derived from the brane tiling picture is independent of the choices of the cut $E$ and maximal tree $T$. We first explain the setup that will allow us to apply the results from Section 4.

\begin{lem}\label{lem3.411}
Let $G$ be the free group $F(y_1, \ldots, y_m)$ and $G'$ be the free group $F(x_1 \ldots, x_m)$, and let $\lambda \in G$ and $\lambda' \in G'$ be words. Suppose we have the following commutative diagram of algebras
\begin{equation}\label{fig:4.01}
\begin{tikzcd}[column sep = 5em, row sep = 5em]
\m C[G'] \arrow[r, "\tau", "\sim"'] \arrow[d, two heads, "p"] & \m C[G] \arrow[d, two heads, "q"]\\
\m C[G']/\big(\lambda'-1 \big) \arrow[r, "\widetilde{\tau}", "\sim"'] & \m C[G]/\big(\lambda -1 \big)
\end{tikzcd}
\end{equation}
where $\lambda$ and $\lambda'$ exist as elements of the group algebras in the natural way, $p$ and $q$ are the obvious quotient maps, $\tau$ is induced from an isomorphism of the free groups 
$$F(x_1, \ldots, x_m) \xrightarrow[\,\, \sim \,\,]{\epsilon} F(y_1, \ldots, y_m),$$
and $\widetilde{\tau}$ is also an isomorphism. Then
$$\tau(\lambda') = u \, \lambda^c \, u^{-1}$$
for some monomial $u \in \m C \langle x_1^{\pm 1},y_1^{\pm 1}, \ldots, x_g^{\pm 1},y_g^{\pm 1} \rangle$ and $c \in \{1,-1\}$.
\end{lem}

\begin{proof}
First we claim that
$$\big(\tau(\lambda') -1 \big) = \big(\lambda -1 \big)$$
as ideals in $\m C\langle y_1^{\pm 1}, \ldots, y_m^{\pm 1} \rangle$. Suppose not, then either there exists some element $a \in \big(\tau(\lambda') -1 \big)$ such that $a \notin \big(\lambda -1 \big)$ or there exists some $a \in \big(\lambda -1 \big)$ such that $a \notin \big(\tau(\lambda') -1 \big)$. Assume the first case holds. Write $a= \tau(b)$ for some $b \in \big(\lambda' -1 \big)$. Then $p(b)=0$ hence $\widetilde{\tau}(p(b))=0$ and so by the commutativity of \eqref{fig:4.01} we get that $q(\tau(b)) =  q(a) = 0$. But this implies that $a \in \big(\lambda -1 \big)$ a contradiction to the initial assumption. The second case follows from a similar argument.

As $\tau$ is induced from the group isomorphism $\epsilon$ consider the elements $\lambda,\, \epsilon(\lambda') \in F(y_1, \ldots, y_m)$. Let $N(\lambda)$ and $N(\epsilon(\lambda'))$ denote the normal closures in $F(y_1, \ldots, y_m)$ of $\lambda$ and $\epsilon(\lambda')$ respectively i.e. the smallest normal subgroups of $F(y_1, \ldots, y_m)$ that contain $\lambda$ or $\epsilon(\lambda')$ respectively. Now for an arbitrary group $H$ we have an injective mapping from the set of all normal subgroups of $H$ to the set of two-sided ideals of the group algebra $\m C[H]$ that sends a normal subgroup $N \lhd H$ to the ideal generated by all elements of the form $h -1$ for $h \in N$ (see [\cite{ms} Proposition 3.3.6]). Then clearly the ideal generated by $N(\lambda)$ is $\big(\lambda -1 \big)$ and the ideal generated by $N(\epsilon(\lambda'))$ is $\big(\tau(\lambda')-1 \big)$, and hence the equality of ideals $\big(\tau(\lambda')-1 \big) = \big(\lambda -1 \big)$ in $\m C[G]$ implies the equality of the normal closures
$$N(\epsilon(\lambda')) = N(\lambda)$$
in $F(y_1, \ldots, y_m)$.

Then [\cite{ls} Proposition 5.8] says that two elements of a free group have equal normal closures if and only if one is conjugate to the other or its inverse, i.e. we get that
$$\epsilon(\lambda')=u \lambda^c u^{-1}$$
for some $c \in \{1,-1\}$ and some word $u \in F(x_1,y_1,\ldots, x_g,y_g)$ which immediately implies the result.
\end{proof}

For an arbitrary quiver with one vertex $Q$ note that the localised path algebra $\m C \widetilde{Q}$ is the group algebra of the free group $F(Q) := F(a \,:\, a \in Q_1)$ whose generators are the arrows in $Q$.

\begin{prop}\label{lem3.412}
For a Riemann surface $\Sigma_g$ fix a brane tiling $\Delta$, a cut $E$ for $(Q_\Delta, W_\Delta)$ and a maximal tree $T \subset Q_\Delta \setminus E$. Let $Q$ denote the quiver given by contracting $T$ in $Q_\Delta \setminus E$. Then there exists a subquiver $Q' \subset Q$ such that
\begin{enumerate}[\normalfont i)]
    \item There exists an equality of ideals in $\m C \widetilde{Q}$
    \begin{align}
        \Big(h - p_h \,:\, h \in Q_1 \setminus Q'_1  \Big)  = \Bigg(\frac{\partial W}{\partial e} \,:\, e \neq e_0 \Bigg) \label{eq:5.1}
    \end{align}
    where $p_h$ is a path in the localised path algebra $\m C\widetilde{Q}'$ and $e_0$ is some nominated element of the cut $E$\\
    \item $|Q'_1| = 2g$\\
    \item There exists an isomorphism of free groups
    $$\epsilon: F(Q'_1) \xrightarrow{\,\, \sim \,\,} F(x_1,y_1, \ldots, x_g,y_g)$$
    which induces the isomorphism of group algebras
    $$\tau'': \m C \widetilde{Q}' \xrightarrow{\,\, \sim \,\,} \m C\langle x_1^{\pm 1}, y_1^{\pm 1}, \ldots, x_g^{\pm 1},y_g^{\pm 1} \rangle$$
    such that the surjection
    $$\tau': \m C \widetilde{Q} \xrightarrow{\,\, \tau \,\,} \m C \widetilde{Q}' \xrightarrow[\sim]{\,\, \tau'' \,\,} \m C\langle x_1^{\pm 1}, y_1^{\pm 1}, \ldots, x_g^{\pm 1},y_g^{\pm 1} \rangle \,$$
    induces the isomorphism
    $$\r{Jac}(\widetilde{Q}, W, E) \xrightarrow[\sim]{\,\, \widetilde{\tau}' \,\,} \m C[\pi_1(\Sigma_g)]$$
    from \Cref{prop1.02}, where
    $$\tau(a) =
    \begin{cases}
    \, a \,\, & \r{if} \,\, a \in Q'_1,\\
    \, p_h \,\, & \r{if} \,\, a = h \in Q_1 \setminus Q'_1 \,.
    \end{cases}$$
\end{enumerate}
\end{prop}

\begin{proof}
For an arrow $a \in Q_{\Delta, 1}$ let $\widehat{a} \in \Delta$ denote its dual edge. Let
$$\widehat{E} =  \big\{ \widehat{e} \in \Delta \,\, | \,\, e \in E \big\}$$
be the dimer corresponding to the cut $E$, i.e. the set of edges in the brane tiling which are dual to the arrows in $E$. Extend $\widehat{E}$ to a maximal tree $\widehat{F}$ inside $\Delta$. Let $\widehat{H}$ denote the set of edges in the complement $\widehat{F} \setminus \widehat{E}$ and let $r = |\widehat{H}|$. Note that $r=|E|-1$, since to extend $|\widehat{E}| = |E|$ disjoint edges that touch every vertex into a maximal tree we require $|E|-1$ additional edges. Let 
$$H =  \big\{ h \in Q_1 \,\, | \,\, \widehat{h} \in \widehat{H} \big\}.$$
then take $Q' \subset Q$ to be the quiver with arrow set $Q'_1 = Q_1 \setminus H$.

Let us consider what a relation $\partial W/\partial e$ for an arrow $e \in E$ means. Since $E$ is a cut we have that $\partial W/\partial e$ identifies two paths in $Q$. For any $h \in H$ there exists some $e \in E$ (specifically $e$ is dual to $\widehat{e}$ where $\widehat{e}$ and $\widehat{h}$ are connected in $\Delta$) such that we can use the relation $\partial W/ \partial e$ to identify $h$ with a path in $\m C(\widetilde{Q \setminus \{h\}})$. As $\widehat{H}$ turns $\widehat{E}$ into a tree for each $h \in H$ we can choose a distinct $e \in E$ giving pairs $(h,e)$ such that $\partial W/\partial e$ allows us to identify $h$ with a path in $\m C(\widetilde{Q \setminus \{h\}})$, and since $|H| = r = |E| -1$ we will always have one extra arrow $e_0 \in E$ that has not been paired with any $h \in H$. The edges on the ends of the tree $\widehat{F}$ must be in $\widehat{E}$, so if we take $\widehat{e}_1$ to be an edge on the end of $\widehat{F}$ such that it is only connected to one $\widehat{h}_1 \in \widehat{H}$ we get that the relation $\partial W/\partial e_1$ identifies a path in $\m C(Q \setminus H)$ with a path in $\m C Q$, and so in particular we can use $\partial W/\partial e_1$ to identify $h_1$ with a path in $\m C \widetilde{Q}'$. This can be repeated for all $h \in H$ by following the edge $\widehat{h}$ to the ends of the tree $\widehat{F}$ and utilising the relations $\partial W/\partial e$ for any $\widehat{e}$ we encounter along this route (see \eqref{pic:1.1}, \eqref{pic:1.2} and \eqref{pic:1.3} for an illustration of this).

To be more precise we go by induction. First fix pairs $(h_i, e_i)$ where $\widehat{e}_i$ is connected to $\widehat{h}_i$ in $\Delta$, for $ i=1, \ldots, r$. As $|E|=r+1 =|H|+1$ we have one remaining element of the cut which we shall denote by $e_0$. Define the \i{route length} of $h_i$ to be 1 plus the sum of the route lengths of $h_j$ where $\widehat{e}_i$ is connected to $\widehat{h}_j$ for $i \neq j$ (see \eqref{pic:1.3}). Note that there must exist some $h_i$ of route length 1 as otherwise this means that every $\widehat{e} \in \widehat{E}$ is connected to at least two $\widehat{h} \in \widehat{H}$, but as $\widehat{E}$ and $\widehat{H}$ together form a tree this would mean that $|\widehat{H}|>|\widehat{E}|$ which is a contradiction. We go by induction on the route length of the $h_i$. For the base case suppose that $h_i$ has route length 1. Then up to a sign we have that
$$\frac{\partial W}{\partial e_i} = q'_i h_i q''_i - p'_i$$
for paths $p'_i,\, q'_i,\, q''_i$ in $Q'$, since from the definition of the potential $W$ the route length of $h_i$ being 1 implies that no other arrow $h \in H$ appears in $\partial W/\partial e_i$. Therefore in $\m C \widetilde{Q}$ we can write
$$h_i = p_i + q_i^{\prime -1}\left(\frac{\partial W}{\partial e_i} \right) q_i^{\prime \prime -1}$$
where $p_i = q_i^{\prime -1} p_i' q_i^{\prime \prime -1}$ is a path in $\m C \widetilde{Q}'$, and
$$\frac{\partial W}{\partial e_i} = q'_i \big(h_i - p_i \big) q''_i$$
as required. So now suppose it is true that for all $h_j$ of route length at most $m$ we have the equality of ideals
$$I_m := \Big(h_j - p_j \,:\,\, \r{route length}(h_j) \leq m \Big)  = \Bigg(\frac{\partial W}{\partial e_k} \,:\,\, \r{$e_k$ paired with $h_k$ for route length}(h_k) \leq m \Bigg).$$
Then for $h_i$ of route length $m+1$ consider (again up to a sign)
$$\frac{\partial W}{\partial e_i} = q'_i h_i q''_i - p'_i$$
where this time $p'_i,\, q'_i,\, q''_i$ may now contain arrows in $H$ too. However any arrow $h_j \in H$ that appears in either of $p'_i$ or $q'_i$ or $q''_i$ must have route length strictly less than the route length of $h_i$. Hence if we suppose that
$$q_i^{\prime -1} p'_i q_i^{\prime \prime -1} = \prod_j q'_j h_j^{b_j} q''_j$$
where $q'_j$ and $q''_j$ are paths in $\m C \widetilde{Q}'$ and $b_j \in \{-1,1\}$ ($h_j$ can appear at most once in $\partial W/\partial e_i$ due to how $W_\Delta$ was defined), by the induction hypothesis we can write in $\m C \widetilde{Q}/I_m$ (where $\overline{p} \in \m C \widetilde{Q}/I_m$ denotes the image of $p \in \m C \widetilde{Q}$  under the natural quotient map)
$$\overline{h}_i = \prod_j \overline{q}'_j \overline{p}_j^{b_j} \overline{q}''_j + \overline{q}_i^{\prime -1} \left(\overline{\frac{\partial W}{\partial e_i}} \right) \overline{q}_i^{\prime \prime -1}$$
with $\prod_j q'_j p_j^{b_j} q''_j = p_i$ a path in $\m C \widetilde{Q}'$, and
$$\overline{\frac{\partial W}{\partial e_i}} = \overline{q}'_i \Big(\overline{h}_i - \overline{p}_i \Big) \overline{q}''_i.$$ It follows that that
$$\frac{\m C \widetilde{Q}}{\big(h_i - p_i, I_m \big)} = \frac{\m C \widetilde{Q}}{\big(\partial W/\partial e_i, I_m \big)}$$
and hence as ideals in $\m C \widetilde{Q}$
$$\big(h_i - p_i, I_m \big) = \left(\frac{\partial W}{\partial e_i}, I_m \right)$$
as required for i).

This implies the following isomorphism of algebras
$$\r{Jac}(\widetilde{Q},W,E) = \frac{\m C \widetilde{Q}}{\big(\partial W/\partial e : e \in E \big)} \xrightarrow[\sim]{\quad \widetilde{\tau} \quad} \frac{\m C \widetilde{Q}'}{\big(\tau(\partial W/\partial e_0) \big)}$$
where $\tau: \m C \widetilde{Q} \longrightarrow \m C \widetilde{Q}'$ is the surjective homomorphism 
$$a \in Q_1 \mapsto
\begin{cases}
\, a \,\, & \r{if} \,\, a \in Q'_1,\\
\, p_i \,\, & \r{if} \,\, a = h_i \in H
\end{cases}$$
and recall $e_0 \in E$ is the remaining element of the cut not paired with any $h \in H$. Let 
$$\tau\left(\frac{\partial W}{\partial e_0} \right) = p_0 - q_0$$
then write
\begin{align}
    \frac{\m C \widetilde{Q}'}{\big(\tau(\partial W/\partial e_0) \big)} = \m C[G'] \label{eq:5.2}
\end{align}
where $G'$ is the group with presentation $\langle a \in Q'_1 \,\, |\,\, p_0 q_0^{-1} \rangle$. We get the following diagram of algebras in which the left-hand square commutes
\begin{equation}\label{fig10.45}
\begin{tikzcd}[column sep = 6em, row sep = 5em]
\m C \widetilde{Q} \arrow[r, two heads, "\tau"] \arrow[d, two heads, "p"] & \m C \widetilde{Q}' \arrow[r, dashed, "\tau''", "\sim"'] \arrow[d, two heads, "q'"] & \m C\langle x_1^{\pm 1}, y_1^{\pm 1}, \ldots, x_g^{\pm 1},y_g^{\pm 1} \rangle \arrow[d, two heads, "q"]\\
\r{Jac}(\widetilde{Q}, W, E) \arrow[r, "\widetilde{\tau}", "\sim"'] & \m C[G'] \arrow[r,"\widetilde{\tau}''", "\sim"'] & \m C[\pi_1(\Sigma_g)]
\end{tikzcd}
\end{equation}
where $\widetilde{\tau}''$ is chosen such that $\widetilde{\tau}'' \circ \widetilde{\tau} =  \widetilde{\tau}'$ is the isomorphism given in \Cref{prop1.02}, and the vertical maps are the quotients. We aim to find an isomorphism $\tau''$ which makes the right-hand square commute.

We first note that the description of the isomorphism $\widetilde{\tau}': \r{Jac}(\widetilde{Q},W,E) \xrightarrow{\sim} \m C[\pi_1(\Sigma_g)]$ from [\cite{dav1} Sections 4 and 5] tells us that $\widetilde{\tau}'$ is induced from a homomorphism of groups
$$G_{Q,W,E} \longrightarrow \pi_1(\Sigma_g)$$
where $G_{Q,W,E}$ is the group with presentation
$$\langle a \in Q_1 \,\,|\,\, p_e q_e^{-1} \,:\, e \in E \rangle$$
for $\partial W/\partial e = p_e - q_e$. Indeed the quiver $Q$ embeds into the Riemann surface $\Sigma_g$ as explained in \cite{dav1} and the multiplication of paths in the quiver $Q$ (which are in fact loops as $Q$ only has one vertex) corresponds to concatenation of loops in the surface. The relations from the cut and potential then correspond to the relation in $\pi_1(\Sigma_g)$. This implies that the isomorphism $\widetilde{\tau}'': \m C[G'] \xrightarrow{\sim} \m C[\pi_1(\Sigma_g)]$ is induced from a homomorphism of groups 
$$G' \xrightarrow{\,\, \xi \,\,} \pi_1(\Sigma_g)$$
since $\widetilde{\tau}$ is induced from the isomorphism of groups
$$G_{Q,W,E} \xrightarrow{\,\, \sim \,\,} G'$$
that sends $a \in Q_1 \setminus H \mapsto a$ and $h \in H \mapsto p_h$.

We claim that the group homomorphism $\xi$ is an isomorphism. Suppose $\xi$ is not surjective, then there exists some $y \in \pi_1(\Sigma_g)$ not in the image of $\xi$. An arbitrary element of $\m C[G']$ is of the form $x=\sum_{i=1}^m c_i x_i$ and therefore
$$\widetilde{\tau}''(x) = \sum_{i=1}^m c_i\, \xi(x_i).$$
As $\widetilde{\tau''}$ is surjective suppose that $\widetilde{\tau}''(x) = y$. Hence for at least one $i$ we must have that $\xi(x_i)=y$ giving a contradiction to the assumption that $y$ does not lie in the image of $\xi$. Now suppose $\xi$ is not injective. Then there exists an $x \in G'$ such that $\xi(x)$ is null-homotopic. But then $\widetilde{\tau}''(x) = \xi(x)$ is null-homotopic and hence equal to $1$ in the fundamental group algebra $\m C[\pi_1(\Sigma_g)]$ which gives a contradiction to the injectivity of $\widetilde{\tau}''$. 

So we have an isomorphism of one-relator groups $G' \xrightarrow[\sim]{\,\, \xi \,\,} \pi_1(\Sigma_g)$ and hence [\cite{ls} Proposition 5.11] tells us that the number of generators of $G'$ must be equal to $2g$, i.e. $|Q'_1| = 2g$ as required by ii).

Finally we show the existence of the isomorphism $\tau''$. Let $F(z_1, \ldots, z_m)$ be the free group on the generators $z_i$. Define a \emph{marking} of $\pi_1(\Sigma_g)$ of size $m$ to be a surjective group homomorphism $\pi: F(z_1, \ldots, z_m) \rightarrow \pi_1(\Sigma_g)$, hence a marking of size $m$ is equivalent to a set $\{\alpha_1, \ldots, \alpha_m\} \subset \pi_1(\Sigma_g)$ of generators for $\pi_1(\Sigma_g)$. Then [\cite{ll} Theorem 1.1] says that any two markings $\pi$ and $\pi'$ of $\pi_1(\Sigma_g)$ of the same size are Nielsen equivalent, i.e. that there exists an automorphism $\epsilon$ of the free group $F(z_1, \ldots, z_m)$ such that the following diagram commutes
\[\begin{tikzcd}[column sep = 6em, row sep = 6em]
F(z_1, \ldots, z_m) \arrow[rd, "\pi"] \arrow[r, "\epsilon", "\sim"'] & F(z_1, \ldots, z_m) \arrow[d, "\pi'"]\\
& \pi_1(\Sigma_g)
\end{tikzcd}\]
In our case we take the generators $\{\xi(a) : a \in Q'_1 \}$ to define the marking $\pi$ of size $2g$ and the standard generators $\{x_1, y_1, \ldots, x_g, y_g\}$ to define $\pi'$, giving us the automorphism $\epsilon$. Upgrading $\epsilon$ into an isomorphism of group algebras provides us with $\tau''$. 
\end{proof}

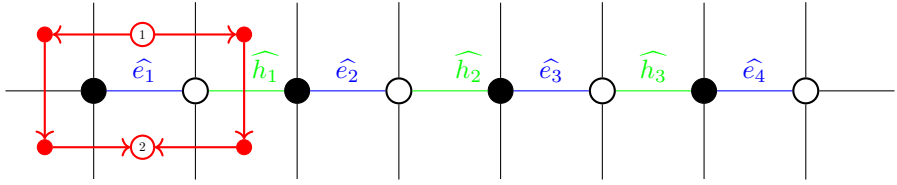
\begin{figure}
    \centering
\begin{tikzpicture}
[black/.style={circle, draw=black!120, fill=black!120, thin, minimum size=3mm},
white/.style={circle, draw=black!120, thick, minimum size=3mm},
empty/.style={circle, draw=black!0, thin, minimum size=0.1mm},
red/.style={circle, draw=red!120, thick, scale=0.5},
redsmall/.style={circle, draw=red!120, fill=red!120, thin, scale=0.6},]

\node[black] (1) {};
\node[white] (2) [right=of 1] {};
\node[black] (3) [right=of 2] {};
\node[white] (4) [right=of 3] {};
\node[black] (5) [right=of 4] {};
\node[white] (6) [right=of 5] {};
\node[black] (7) [right=of 6] {};
\node[white] (8) [right=of 7] {};

\node[empty] (01) [above=of 1] {};
\node[empty] (02) [above=of 2] {};
\node[empty] (03) [above=of 3] {};
\node[empty] (04) [above=of 4] {};
\node[empty] (05) [above=of 5] {};
\node[empty] (06) [above=of 6] {};
\node[empty] (07) [above=of 7] {};
\node[empty] (08) [above=of 8] {};

\node[empty] (11) [below=of 1] {};
\node[empty] (12) [below=of 2] {};
\node[empty] (13) [below=of 3] {};
\node[empty] (14) [below=of 4] {};
\node[empty] (15) [below=of 5] {};
\node[empty] (16) [below=of 6] {};
\node[empty] (17) [below=of 7] {};
\node[empty] (18) [below=of 8] {};

\node[empty] (0) [left=of 1] {};
\node[empty] (9) [right=of 8] {};

\node[red] (051) at (0.65, 0.75) {1};
\node[red] (151) at (0.65, -0.75) {2};
\node[redsmall] (050) at (-0.65, 0.75) {};
\node[redsmall] (150) at (-0.65, -0.75) {};
\node[redsmall] (052) at (2.0, 0.75) {};
\node[redsmall] (152) at (2.0, -0.75) {};

\draw[blue, -] (1.east) -- (2.west) node [above, midway] {$\widehat{e_1}$};
\draw[green, -] (2.east) -- (3.west) node [above, near end] {$\widehat{h_1}$};
\draw[blue, -] (3.east) -- (4.west) node [above, midway] {$\widehat{e_2}$};
\draw[green, -] (4.east) -- (5.west) node [above, near end] {$\widehat{h_2}$};
\draw[blue, -] (5.east) -- (6.west) node [above, midway] {$\widehat{e_3}$};
\draw[green, -] (6.east) -- (7.west) node [above, midway] {$\widehat{h_3}$};
\draw[blue, -] (7.east) -- (8.west) node [above, midway] {$\widehat{e_4}$};

\draw[-] (1.north) -- (01.south);
\draw[-] (2.north) -- (02.south);
\draw[-] (3.north) -- (03.south);
\draw[-] (4.north) -- (04.south);
\draw[-] (5.north) -- (05.south);
\draw[-] (6.north) -- (06.south);
\draw[-] (7.north) -- (07.south);
\draw[-] (8.north) -- (08.south);
\draw[-] (1.south) -- (11.north);
\draw[-] (2.south) -- (12.north);
\draw[-] (3.south) -- (13.north);
\draw[-] (4.south) -- (14.north);
\draw[-] (5.south) -- (15.north);
\draw[-] (6.south) -- (16.north);
\draw[-] (7.south) -- (17.north);
\draw[-] (8.south) -- (18.north);
\draw[-] (1.west) -- (0.east);
\draw[-] (8.east) -- (9.west);

\draw[red, ->] (051.west) -- (050.east);
\draw[red, ->] (051.east) -- (052.west);
\draw[red, ->] (050.south) -- (150.north);
\draw[red, ->] (150.east) -- (151.west);
\draw[red, ->] (052.south) -- (152.north);
\draw[red, ->] (152.west) -- (151.east);
\end{tikzpicture}
    \caption{An example of a brane tiling with black and white vertices, with the edges in the dimer $\widehat{E}$ in blue and the edges in $\widehat{H}$ in green. The tree $\widehat{F}$ then consists of all the blue and green edges and so its ends are $\widehat{e}_1$ and $\widehat{e}_4$. For illustrative purposes we do not contract the tree $T$ in $Q_\Delta$ and so our quivers will have more than one vertex. We can see that the relation $\partial W/\partial e_1$ identifies the two paths in the quiver (in red) that go from vertex $1$ to vertex $2$ in $Q_\Delta$. Since $\widehat{e_1}$ is at an end of the tree one path does not contain any arrows which are dual to edges in $\widehat{H}$.}
    \label{pic:1.1}
\end{figure}

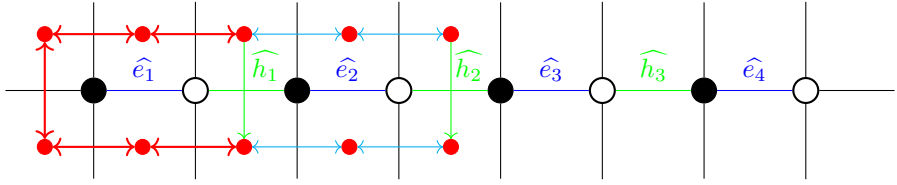
\begin{figure}
    \centering
\begin{tikzpicture}
[black/.style={circle, draw=black!120, fill=black!120, thin, minimum size=3mm},
white/.style={circle, draw=black!120, thick, minimum size=3mm},
empty/.style={circle, draw=black!0, thin, minimum size=0.1mm},
red/.style={circle, draw=red!120, thick, scale=0.5},
redsmall/.style={circle, draw=red!120, fill=red!120, thin, scale=0.6},]

\node[black] (1) {};
\node[white] (2) [right=of 1] {};
\node[black] (3) [right=of 2] {};
\node[white] (4) [right=of 3] {};
\node[black] (5) [right=of 4] {};
\node[white] (6) [right=of 5] {};
\node[black] (7) [right=of 6] {};
\node[white] (8) [right=of 7] {};

\node[empty] (01) [above=of 1] {};
\node[empty] (02) [above=of 2] {};
\node[empty] (03) [above=of 3] {};
\node[empty] (04) [above=of 4] {};
\node[empty] (05) [above=of 5] {};
\node[empty] (06) [above=of 6] {};
\node[empty] (07) [above=of 7] {};
\node[empty] (08) [above=of 8] {};

\node[empty] (11) [below=of 1] {};
\node[empty] (12) [below=of 2] {};
\node[empty] (13) [below=of 3] {};
\node[empty] (14) [below=of 4] {};
\node[empty] (15) [below=of 5] {};
\node[empty] (16) [below=of 6] {};
\node[empty] (17) [below=of 7] {};
\node[empty] (18) [below=of 8] {};

\node[empty] (0) [left=of 1] {};
\node[empty] (9) [right=of 8] {};

\node[redsmall] (051) at (0.65, 0.75) {};
\node[redsmall] (151) at (0.65, -0.75) {};
\node[redsmall] (050) at (-0.65, 0.75) {};
\node[redsmall] (150) at (-0.65, -0.75) {};
\node[redsmall] (052) at (2.0, 0.75) {};
\node[redsmall] (152) at (2.0, -0.75) {};
\node[redsmall] (053) at (3.4, 0.75) {};
\node[redsmall] (153) at (3.4, -0.75) {};
\node[redsmall] (054) at (4.75, 0.75) {};
\node[redsmall] (154) at (4.75, -0.75) {};

\draw[blue, -] (1.east) -- (2.west) node [above, midway] {$\widehat{e_1}$};
\draw[green, -] (2.east) -- (3.west) node [above, near end] {$\widehat{h_1}$};
\draw[blue, -] (3.east) -- (4.west) node [above, midway] {$\widehat{e_2}$};
\draw[green, -] (4.east) -- (5.west) node [above, near end] {$\widehat{h_2}$};
\draw[blue, -] (5.east) -- (6.west) node [above, midway] {$\widehat{e_3}$};
\draw[green, -] (6.east) -- (7.west) node [above, midway] {$\widehat{h_3}$};
\draw[blue, -] (7.east) -- (8.west) node [above, midway] {$\widehat{e_4}$};

\draw[-] (1.north) -- (01.south);
\draw[-] (2.north) -- (02.south);
\draw[-] (3.north) -- (03.south);
\draw[-] (4.north) -- (04.south);
\draw[-] (5.north) -- (05.south);
\draw[-] (6.north) -- (06.south);
\draw[-] (7.north) -- (07.south);
\draw[-] (8.north) -- (08.south);
\draw[-] (1.south) -- (11.north);
\draw[-] (2.south) -- (12.north);
\draw[-] (3.south) -- (13.north);
\draw[-] (4.south) -- (14.north);
\draw[-] (5.south) -- (15.north);
\draw[-] (6.south) -- (16.north);
\draw[-] (7.south) -- (17.north);
\draw[-] (8.south) -- (18.north);
\draw[-] (1.west) -- (0.east);
\draw[-] (8.east) -- (9.west);

\draw[red, <->] (051.west) -- (050.east);
\draw[red, <->] (051.east) -- (052.west);
\draw[red, <->] (050.south) -- (150.north);
\draw[red, <->] (150.east) -- (151.west);
\draw[green, ->] (052.south) -- (152.north);
\draw[red, <->] (152.west) -- (151.east);

\draw[cyan, <->] (053.west) -- (052.east);
\draw[cyan, <->] (053.east) -- (054.west);
\draw[cyan, <->] (152.east) -- (153.west);
\draw[green, ->] (054.south) -- (154.north);
\draw[cyan, <->] (154.west) -- (153.east);
\end{tikzpicture}
    \caption{By adding inverses to the arrows in $Q_\Delta$ we can use the relation $\partial W/ \partial e_1$ to write the arrow in the quiver dual to $\widehat{h_1}$ (in green) in terms of the red path and so in particular as a path in $\m C \widetilde{Q}'_\Delta = \m C(\widetilde{Q_{\Delta} \setminus H})$. Similarly for the arrow in the quiver dual to $\widehat{h_2}$ (also in green), by following the edge $\widehat{h_2}$ to an end of the tree in the brane tiling (here we go to the left and end up at $\widehat{e_1}$) we can also write it as a path in $\m C \widetilde{Q}'_{\Delta}$; first by using $\partial W/ \partial e_2$ we can write $h_2$ in terms of the cyan arrows and $h_1$, then using $\partial W/ \partial e_1$ we can write it as the concatenation of the cyan path along the top, then the red path, then the cyan path along the bottom.}
    \label{pic:1.2}
\end{figure}

\begin{figure}
    \centering
\begin{tikzpicture}
[black/.style={circle, draw=black!120, fill=black!120, thin, minimum size=3mm},
white/.style={circle, draw=black!120, thick, minimum size=3mm},
empty/.style={circle, draw=black!0, thin, minimum size=0.1mm},
red/.style={circle, draw=red!120, thick, scale=0.5},
redsmall/.style={circle, draw=red!120, fill=red!120, thin, scale=0.6},]

\node[black] (1) {};
\node[white] (2) [right=of 1] {};
\node[black] (3) [right=of 2] {};
\node[white] (4) [right=of 3] {};
\node[black] (5) [right=of 4] {};
\node[white] (6) [right=of 5] {};
\node[black] (7) [right=of 6] {};
\node[white] (8) [right=of 7] {};

\node[empty] (01) [above=of 1] {};
\node[black] (02) [above=of 2] {};
\node[white] (03) [above=of 3] {};
\node[empty] (04) [above=of 4] {};
\node[empty] (05) [above=of 5] {};
\node[empty] (06) [above=of 6] {};
\node[empty] (07) [above=of 7] {};
\node[empty] (08) [above=of 8] {};

\node[empty] (11) [below=of 1] {};
\node[empty] (12) [below=of 2] {};
\node[empty] (13) [below=of 3] {};
\node[empty] (14) [below=of 4] {};
\node[empty] (15) [below=of 5] {};
\node[empty] (16) [below=of 6] {};
\node[empty] (17) [below=of 7] {};
\node[empty] (18) [below=of 8] {};
\node[empty] (002) [above=of 02] {};
\node[empty] (003) [above=of 03] {};

\node[empty] (0) [left=of 1] {};
\node[empty] (9) [right=of 8] {};

\node[redsmall] (051) at (0.65, 0.75) {};
\node[redsmall] (151) at (0.65, -0.75) {};
\node[redsmall] (050) at (-0.65, 0.75) {};
\node[redsmall] (150) at (-0.65, -0.75) {};
\node[redsmall] (052) at (2.0, 0.75) {};
\node[redsmall] (152) at (2.0, -0.75) {};
\node[redsmall] (053) at (3.4, 0.75) {};
\node[redsmall] (153) at (3.4, -0.75) {};
\node[redsmall] (054) at (4.75, 0.75) {};
\node[redsmall] (154) at (4.75, -0.75) {};
\node[redsmall] (0051) at (0.65, 2.05) {};
\node[redsmall] (0052) at (2.0, 2.05) {};
\node[redsmall] (0053) at (3.4, 2.05) {};

\draw[blue, -] (1.east) -- (2.west) node [above, midway] {$\widehat{e_1}$};
\draw[green, -] (2.east) -- (3.west) node [above, midway] {$\widehat{h_1}$};
\draw[blue, -] (3.east) -- (4.west) node [above, midway] {$\widehat{e_2}$};
\draw[green, -] (4.east) -- (5.west) node [above, near end] {$\widehat{h_2}$};
\draw[blue, -] (5.east) -- (6.west) node [above, midway] {$\widehat{e_3}$};
\draw[green, -] (6.east) -- (7.west) node [above, midway] {$\widehat{h_3}$};
\draw[blue, -] (7.east) -- (8.west) node [above, midway] {$\widehat{e_5}$};
\draw[blue, -] (02.east) -- (03.west) node [above, midway] {$\widehat{e_4}$};
\draw[green, -] (3.north) -- (03.south) node [right, midway] {$\widehat{h_4}$};

\draw[-] (1.north) -- (01.south);
\draw[-] (2.north) -- (02.south);
\draw[-] (4.north) -- (04.south);
\draw[-] (5.north) -- (05.south);
\draw[-] (6.north) -- (06.south);
\draw[-] (7.north) -- (07.south);
\draw[-] (8.north) -- (08.south);
\draw[-] (1.south) -- (11.north);
\draw[-] (2.south) -- (12.north);
\draw[-] (3.south) -- (13.north);
\draw[-] (4.south) -- (14.north);
\draw[-] (5.south) -- (15.north);
\draw[-] (6.south) -- (16.north);
\draw[-] (7.south) -- (17.north);
\draw[-] (8.south) -- (18.north);
\draw[-] (1.west) -- (0.east);
\draw[-] (8.east) -- (9.west);
\draw[-] (01.east) -- (02.west);
\draw[-] (03.east) -- (04.west);
\draw[-] (02.north) -- (002.south);
\draw[-] (03.north) -- (003.south);

\draw[red, ->] (051.west) -- (050.east);
\draw[red, ->] (050.south) -- (150.north);
\draw[red, ->] (150.east) -- (151.west);
\draw[red, <-] (152.west) -- (151.east);
\draw[red, <-] (053.east) -- (054.west);
\draw[red, ->] (152.east) -- (153.west);
\draw[red, ->] (153.east) -- (154.west);
\draw[green, ->] (054.south) -- (154.north);
\draw[red, ->] (053.north) -- (0053.south);
\draw[red, ->] (0053.west) -- (0052.east);
\draw[red, ->] (0052.west) -- (0051.east);
\draw[red, ->] (0051.south) -- (051.north);
\end{tikzpicture}
    \caption{Fixing the pairs $(e_1,h_1),\, (e_2,h_2),\, (e_3,h_3),\, (e_4,h_4)$ we get that $\r{route length}(h_1) = 1$, $\r{route length}(h_2) = 3$, $\r{route length}(h_3) = 4$, $\r{route length}(h_4) = 1$. Route length 1 arrows in $H$ exist as those paired to arrows $e \in E$ with the property that $\widehat{e}$ only connects to a single $\widehat{h} \in \widehat{H}$. We can use the relations $\partial W/\partial e_j$, where $\widehat{e}_j$ is on the route from $\widehat{h}_i$ to an end of the tree in the brane tiling, to write $h_i$ as a path in $\m C \widetilde{Q}'_\Delta$ e.g. we see that the arrow $h_2$ in green is equal to the red path using $\partial W/\partial e_2$ then $\partial W/\partial e_1$ and $\partial W/\partial e_4$.}
    \label{pic:1.3}
\end{figure}
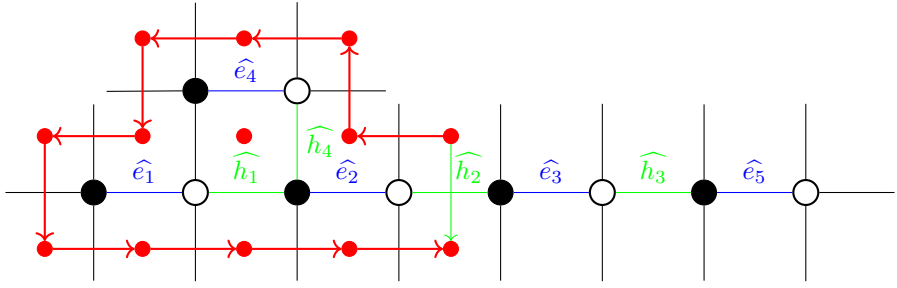

\newpage
We are now finally in a position to begin applying the results from Section 4 to the character variety. We fix the brane tiling $\Delta$ on our Riemann surface $\Sigma_g$, as well as the cut $E$, maximal tree $T$, and subquiver $Q' \subset Q$ as in \Cref{lem3.412}. Recall that $|E| = r+1$, $|Q'_1| = 2g$ and $|Q_1| = 2g+r$. Let $e_i \in E$ for $i=0, \ldots, r$. Also note that
\begin{align*}
    M_n &= M_n(\m C \widetilde{Q}) = \r{GL}_n^{2g+r}\\
    Y'_{m,n} &\cong M_m \times M_n \times \r{Mat}_{m \times n}^{|E|}\\
    &= \r{GL}_m^{2g+r} \times \r{GL}_n^{2g+r} \times \r{Mat}_{m \times n}^{r+1}\\
    &= \big(\r{GL}_m^{2g} \times \r{GL}_n^{2g} \times \r{Mat}_{m \times n} \big) \times \big(\r{GL}_m^r \times \r{GL}_n^r \times \r{Mat}_{m \times n}^r \big)\\
    &\cong Y_{m,n} \times \r{GL}_{m,n}^r\\
    M_{m,n} &= M_{m,n}(\m C \widetilde{Q}) = \r{GL}_{m,n}^{2g+r}\\
    &= \r{GL}_{m,n}^{2g} \times \r{GL}_{m,n}^{r}\\
    &= M_{m,n}(\m C \widetilde{Q}') \times \r{GL}_{m,n}^{r}.
\end{align*}
It will be much more convenient to view $Y_{m,n}$ as $\r{GL}_m^{2g} \times \r{GL}_n^{2g} \times \r{Mat}_{m \times n}$ and $Y'_{m,n}$ as $M_m \times M_n \times \r{Mat}_{m \times n}^{r+1}$. We shall index matrices in the $M_{m,n}(\widetilde{Q}') = \r{GL}_{m,n}^{2g}$ part using $a \in Q'_1$ and matrices in the $\r{GL}_{m,n}^r$ part using $h \,\, \r{or} \,\, h_i \in Q_1 \setminus Q'_1$ for $i=1, \ldots, r$. Then recalling our notation for $m,n$-block matrices
$$R =
\begin{pmatrix}
R^{(1)} & R^{(3)}\\
0 & R^{(2)}
\end{pmatrix}$$
we have that in diagrams \eqref{fig:2.1} and \eqref{fig:1} for $(A_i,B_i) \in \r{GL}_{m,n}^{2g} = M_{m,n}(\m C \langle x_1^{\pm 1}, y_1^{\pm 1}, \ldots, x_g^{\pm 1}, y_g^{\pm 1} \rangle)$ and for $\rho \in M_{m,n}$
\begin{align*}
    f(A_i, B_i) &= \Big((A_i^{(1)}, B_i^{(1)}),\, (A_i^{(2)}, B_i^{(2)}),\,\, \lambda_{m+n}(A_i, B_i)^{(3)} \Big)\\
    f''(\rho) &= \bigg(\rho^{(1)},\, \rho^{(2)},\,\, \left(\frac{\partial W}{\partial e_i}(\rho)^{(3)} \right)_{i=1, \ldots, r} \bigg).
\end{align*}
So consider the following diagram whose front face is \eqref{fig:2} and whose back face is \eqref{fig:1}:
\begin{equation}\label{fig:10.5}
\begin{tikzcd}
& M_{m+n} \arrow[hookleftarrow]{dd}[near end]{j'_{m+n}} && M_{m,n} \arrow[hookleftarrow]{dd}[near end]{j'_{m,n}} \arrow[ll, hook', "h'"'] \arrow[rr, "f''"] && Y'_{m,n} \arrow[hookleftarrow]{dd}{i'}\\
\r{GL}_{m+n}^{2g} \arrow[ru, hook, "\varphi'_{m+n}"] && \r{GL}_{m,n}^{2g} \arrow[ru, hook, "\varphi'_{m,n}"] \arrow[ll, hook', "h"' near start, crossing over] \arrow[rr, "f" near end, crossing over] && Y_{m,n} \arrow[ru, hook, "\psi'_{m,n}"]\\
& Z_{m+n} && Z_{m,n} \arrow[ll, hook', "\widetilde{h'}"' near end] \arrow[rr, "\widetilde{f''}" near start] && Z_m \times Z_n\\
V_{m+n} \arrow[uu, hook', "j_{m+n}"] \arrow[ru, "\widetilde{\varphi}'_{m+n}" near end, "\sim"' sloped] && V_{m,n} \arrow[uu, hook', crossing over, "j_{m,n}" near start] \arrow[ru, "\widetilde{\varphi}'_{m,n}"' near end, "\sim" sloped] \arrow[ll, hook', "\widetilde{h}"'] \arrow[rr, "\widetilde{f}" near end] && V_m \times V_n\arrow [uu, hook', crossing over, "i" near start] \arrow[ru, "\widetilde{\psi}'_{m,n} = \widetilde{\varphi}'_m \times \widetilde{\varphi}'_n"' near end, "\sim" sloped]
\end{tikzcd}
\end{equation}
where the inclusion $\varphi'_n: \r{GL}_n^{2g} \rightarrow M_n$ is induced from the algebra homomorphism $\tau'$ given in \Cref{lem3.412}, the inclusion $\psi'_{m,n}: Y_{m,n} \rightarrow Y'_{m,n}$ is given by
$$\big(\rho',\, \rho'',\, R \big) \longmapsto \big(\varphi'_m(\rho'),\, \varphi'_n(\rho''),\, R,\, 0, \ldots, 0 \big)$$
and the isomorphism $\widetilde{\varphi}'_n : V_n \xrightarrow{\sim} Z_n$ is induced from the algebra isomorphism $\widetilde{\tau}'$ from \Cref{prop1.02}.

All the squares that go into the page are Cartesian due to the commutativity of diagram \eqref{fig10.45}. The rest of the squares, except for the top-right square, are all easily checked to be Cartesian as well. The top-right square may not even be commutative in fact, since we have that
\begin{align*}
    f'' \circ \varphi'_{m,n}(A_i, B_i)  &= \left(\varphi'_{m,n}(A_i, B_i)^{(1)},\, \varphi'_{m,n}(A_i, B_i)^{(2)},\, \left(\frac{\partial W}{\partial e_j}(\varphi'_{m,n}(A_i, B_i))^{(3)} \right)\right)\\
    &= \left(\varphi'_m(A_i^{(1)},B_i^{(1)}),\, \varphi'_n(A_i^{(2)}, B_i^{(2)}),\, \left(\frac{\partial W}{\partial e_j}(\varphi'_{m,n}(A_i, B_i))^{(3)} \right)\right)
\end{align*}
whilst
$$\psi_{m,n} \circ f(A_i, B_i) = \Big(\varphi'_m(A_i^{(1)},B_i^{(1)}),\, \varphi'_n(A_i^{(2)}, B_i^{(2)}),\,\, \lambda_{m+n}(A_i,B_i)^{(3)},\, 0, \ldots, 0 \Big)$$
and therefore this square is commutative if and only if (up to a re-ordering of the arrows in $E$) we have $\frac{\partial W}{\partial e_0}(\varphi'_{m,n}(A_i, B_i))^{(3)} = \lambda_{m+n}(A_i,B_i)^{(3)}$ and $\frac{\partial W}{\partial e_j}(\varphi'_{m,n}(A_i, B_i))^{(3)} = 0$ for $j = 1, \ldots, r$. This means that we cannot apply \Cref{cor1.06} to the right-hand cube in \eqref{fig:10.5} because there most likely will not exist an isomorphism $\delta: Y_{m,n}' \xrightarrow{\sim} Y_{m,n}'$ that will make this square Cartesian after replacing $f''$ with $f'= \delta^{-1} \circ f''$.

\begin{eg}\label{eg5.3}
Continuing the example of the brane tiling for the genus 2 surface from \Cref{eg:1.01} and \Cref{eg:1.02}, we took the cut as $E = \{a,b,c\}$ and the maximal tree as $T=\{e,h,k\}$. To extend the dimer $\widehat{E}$ to a maximal tree in the tiling we can take $\widehat{H}=\{\widehat{f}, \widehat{j}\}$ and so $H=\{f,j\}$. Pairing $f$ with $c$ and $j$ with $a$ we can use $\partial W/\partial c$ and $\partial W/\partial a$ to write
\begin{align*}
    f &= i^{-1}l^{-1}id\\
    j &= gdg^{-1}
\end{align*}
as we saw before. Then $e_0 = b$ and $\partial W/\partial b$ becomes
$$li^{-1}l^{-1}idgd^{-1}g^{-1} - 1$$
and so defining $\tau''$ by sending $l \mapsto x_1$, $i^{-1} \mapsto y_1$, $d \mapsto x_2$, $g \mapsto y_2$ gives us the isomorphism and surjection of algebras from \Cref{lem3.412}. It follows that $\varphi'_n$ sends
$$\big(A_1, B_1, A_2, B_2 \big) \longmapsto \big(D= A_2,\,\, F= B_1A_1^{-1}B_1^{-1}A_2,\,\, G=B_2,\,\, I=B_1^{-1},\,\, J= B_2A_2B_2^{-1},\,\, L= A_1 \big)$$
and we see that
\begin{align*}
    f'' \circ \varphi'_{m,n}(A_1, B_1, A_2, B_2) &= \Big(B_2A_2B_2^{-1}B_2-B_2A_2, \quad A_1B_1A_1^{-1}B_1^{-1}A_2-B_2A_2B_2^{-1},\\
    &\qquad B_1^{-1}A_2-A_1B_1^{-1}B_1A_1^{-1}B_1^{-1}A_2 \Big)\\
    &= \Big(0,\,\, A_1B_1A_1^{-1}B_1^{-1}A_2-B_2A_2B_2^{-1},\,\, 0 \Big)
\end{align*}
whilst
\begin{align*}
    \psi_{m,n} \circ f(A_1, B_1, A_2, B_2) &= \Big(0,\,\, A_1B_1A_1^{-1}B_1^{-1}A_2B_2A_2B_2^{-1}- \r{Id}_{m+n},\,\, 0 \Big).
\end{align*}
Hence in this case the top-right square in \eqref{fig:10.5} does not commute and there does not exist an isomorphism $\delta$ that can rectify this.
\end{eg}

To deal with the non-commutativity of the top-right square in \eqref{fig:10.5} we extend the right-hand cube into:
\begin{equation}\label{fig:10.6}
\begin{tikzcd}[column sep=1.5em]
& M_{m,n} \arrow[rrr, "f_0''", crossing over] &&& X'_{m,n} \arrow[rr, "\pi_1'"] && Y'_{m,n}\\
\r{GL}_{m,n}^{2g} \arrow[ru, hook, "\varphi'_{m,n}"] &&& X_{m,n} \arrow[ru, hook, "\psi'_{m,n} \times \r{id}"] && Y_{m,n} \arrow[ru, hook, "\psi'_{m,n}"]\\
& Z_{m,n} \arrow[uu, hook', "j'_{m,n}" near start] \arrow[rrr, "\widetilde{f_0''}" near start] &&& Z \arrow[uu, hook', crossing over, "j'" near start] \arrow[rr, "\widetilde{\pi}_1'" near start] && Z_m \times Z_n \arrow [uu, hook', crossing over, "i'"']\\
V_{m,n} \arrow[uu, hook', "j_{m,n}"] \arrow[ru, "\widetilde{\varphi}'_{m,n}"' near end, "\sim" sloped] \arrow[rrr, "\widetilde{f_0}"'] &&& V \arrow[uu, hook', crossing over, "j" near start] \arrow[ru, "\widetilde{\psi}'_{m,n} \times \r{id}"' near end, "\sim" sloped] \arrow[rr, "\widetilde{\pi}_1"'] && V_m \times V_n \arrow [uu, hook', crossing over, "i" near start] \arrow[ru, "\widetilde{\psi}'_{m,n}"' near end, "\sim" sloped]
\arrow[from=2-4,to=2-6, "\pi_1" near end, crossing over]
\arrow[from=2-1,to=2-4, "f_0" near end, crossing over] 
\end{tikzcd}
\end{equation}
where
\begin{align*}
    X_{m,n} &= Y_{m,n} \times \r{Mat}_{m \times n}^{2g+r}\\
    X'_{m,n} &= Y'_{m,n} \times \r{Mat}_{m \times n}^{2g+r}\\
    V &= V_m \times V_n \times \r{Mat}_{m \times n}^{2g+r}\\
    Z &= Z_m \times Z_n \times \r{Mat}_{m \times n}^{2g+r}
\end{align*}
and
\begin{align*}
    f_0(A_i, B_i) &= \Big(f(A_i, B_i),\,\, \varphi'_{m,n}(A_i, B_i)^{(3)} \Big)\\
    f_0''(\rho) &= \Big(f''(\rho),\,\, \rho^{(3)} \Big)
\end{align*}
the maps $\pi_1,\, \pi_1',\, \widetilde{\pi}_1,\, \widetilde{\pi}'_1$ are the respective projections, and $j,\, j'$ are the obvious inclusions.

Hence \eqref{fig:10.6} is indeed just the right-hand cube of \eqref{fig:10.5} since $\pi_1 \circ f_0 = f$, $\pi_1' \circ f_0'' = f''$ and $\widetilde{\pi}_1 \circ \widetilde{f}_0 = \widetilde{f}$, $\widetilde{\pi}_1' \circ \widetilde{f}_0'' = \widetilde{f}''$. It is also clear that all the squares in diagram \eqref{fig:10.6} except the top-left square are Cartesian. The top-left square of \eqref{fig:10.6} will fail to be commutative due to the top-right square of \eqref{fig:10.5} not being commutative. 

\begin{prop}\label{prop1.08}
The right-hand cube in diagram \eqref{fig:10.6} satisfies all the conditions needed to apply \Cref{cor1.06}.
\end{prop}

\begin{proof}
Looking at the conditions from \eqref{fig:1.004} needed to apply \Cref{cor1.06}, all are easily verified apart from the condition that the top face of this cube can be fitted into a diagram of the form \eqref{fig:1.001} using maps $\psi' = \psi \circ \alpha$, $\varphi' = \varphi \circ \beta$ and $f'' = \delta \circ f'$ such that $\pi_Y \circ \psi = \r{id}_{Y_{m,n}}$ and $(\pi_Y \times \r{id}_{\r{Mat}_{m \times n}^{2g+r}}) \circ \varphi = \r{id}_{X_{m,n}}$, where $\psi,\, \alpha, \varphi,\, \beta, f',\, \delta$ and $\pi_Y$ need to be determined.

For this cube, we have $\varphi' = \psi'_{m,n} \times \r{id}_{\r{Mat}_{m \times n}^{2g+r}}$ and $\psi' = \psi'_{m,n}$. The inclusion $\psi'_{m,n}$ is defined using $\varphi'_m$ and $\varphi'_n$ which are in turn determined by the algebra surjection $\tau'$ given by \Cref{lem3.412}. $\tau'$ was constructed such that $\tau' = \tau'' \circ \tau$ which gives us the factorisation of $\varphi'$ and of $\psi'$ we are looking for. In particular, we write $\varphi'_n = \varphi_n \circ \alpha_n$ where $\varphi_n: \r{GL}_n^{2g} \hookrightarrow \r{GL}_n^{2g+r}$ is the inclusion
$$\big(A_a \big)_{a \in Q'_1} \longmapsto \big(A_a, p_i(A_a) \big)_{a \in Q'_1,\,\, i = 1, \ldots, r}$$
which is induced from the surjection $\tau$, and $\alpha_n: \r{GL}_n^{2g} \xrightarrow{\sim} \r{GL}_n^{2g}$ is the isomorphism induced from the isomorphism $\tau''$. Also consider the projection $\pi_n: \r{GL}_n^{2g+r} \longrightarrow \r{GL}_n^{2g}$
$$\big(A_a, B_{h_i} \big)_{a \in Q'_1,\, i=1, \ldots r} \longmapsto \big(A_a \big)_{a \in Q'_1}.$$
Then we take $\psi = \psi_{m,n}: Y_{m,n} \longrightarrow Y'_{m,n}$ to be
$$\big(\rho',\, \rho'',\, R \big) \longmapsto \big(\varphi_m(\rho'),\, \varphi_n(\rho''),\, R,\,\, 0, \ldots, 0 \big)$$
and $\pi_Y: Y'_{m,n} \longrightarrow Y_{m,n}$ to be
$$\big(\rho',\, \rho'',\, R_0, \ldots, R_r \big) \longmapsto \big(\pi_m(\rho'),\, \pi_n(\rho''),\, R_0 \big)$$
and
\begin{align*}
    \alpha &= \alpha_m \times \alpha_n \times \r{id}_{\r{Mat}_{m \times n}}\\
    \varphi &= \psi_{m,n} \times \r{id}_{\r{Mat}_{m \times n}^{2g+r}}\\
    \beta &= \alpha \times \r{id}_{\r{Mat}_{m \times n}^{2g+r}}\\
    \delta &= \r{id}_{Y'_{m,n}}\\
    f' &= f'' = \pi'_1.
\end{align*}
The squares from diagram \eqref{fig:1.001} then become
\[\begin{tikzcd}[column sep = 8em, row sep = 6em]
Y'_{m,n} \times \r{Mat}_{m \times n}^{2g+r} \arrow[r, "\pi_1'"] & Y'_{m,n}\\
Y_{m,n} \times \r{Mat}_{m \times n}^{2g+r} \arrow[u, hook, "\psi_{m,n} \times \r{id}"] \arrow[r, "\pi_1"] & Y_{m,n} \arrow[u, hook, "\psi_{m,n}"]
\end{tikzcd}\]
and
\[\begin{tikzcd}[column sep = 8em, row sep = 6em]
Y_{m,n} \times  \r{Mat}_{m \times n}^{2g+r} \arrow[r, "\pi_1"] & Y_{m,n}\\
Y'_{m,n} \times  \r{Mat}_{m \times n}^{2g+r} \arrow[u, "\pi_Y \times \r{id}"] \arrow[r, "\pi_1'"] & Y'_{m,n} \arrow[u, "\pi_Y"]
\end{tikzcd}\]
both of which are clearly Cartesian, and we have that $\pi_Y \circ \psi = \r{id}_{Y_{m,n}}$ and $(\pi_Y \times \r{id}_{\r{Mat}_{m \times n}^{2g+r}}) \circ \varphi = \r{id}_{X_{m,n}}$ as required.
\end{proof}

The reason we have constructed diagram \eqref{fig:10.6} is that, whilst for \eqref{fig:10.5} there most likely will not exist an isomorphism $\delta_Y: Y'_{m,n} \xrightarrow{\sim} Y'_{m,n}$ that will make the top-right square Cartesian when replacing $f''$ with $f' =  \delta_Y^{-1} \circ f''$, there will exist an isomorphism $\delta_X: X'_{m,n} \xrightarrow{\sim} X'_{m,n}$ such that the top-left square will become Cartesian after replacing $f''_0$ with $f'_0 = \delta^{-1}_X \circ f''_0$. This is because due to the definitions of $X'_{m,n}$ and $f''_0$ we retain all the information of the representations in $M_{m,n}$. Informally we can see this happening in \Cref{eg5.3}. It is not too hard to check that we can fully recover the data of $(A_1, B_1, A_2, B_2) \in \r{GL}_{m,n}^4$ in the image of $f''_0 \circ \varphi'_{m,n}$. So if $(\rho', \rho'', R_1, R_2, R_3, (D,F,G,I,J,L))$ is in the image of $f''_0 \circ \varphi'_{m,n}$ our isomorphism $\delta_X$ will then send this to
$$\big(\rho',\, \rho'',\, R_1,\, R_2 B_2 A_2^{-1} B_2^{-1},\, R_3,\, (D,\,F,\,G,\,I,\,J,\,L) \big)$$
and we can see that from the calculation in \Cref{eg5.3} this will make the square commutative (and in fact it will be Cartesian).

\begin{prop}\label{prop1.09}
The left-hand cube in diagram \eqref{fig:10.6} satisfies all the conditions needed to apply \Cref{cor1.06}.
\end{prop}

\begin{proof}
As in \Cref{prop1.08} all the conditions are obvious bar the fact that the top face of the cube can be fitted into a diagram of the form \eqref{fig:1.001}. For this cube we have $\varphi' = \varphi'_{m,n}$ and $\psi'= \psi'_{m,n} \times \r{id}_{\r{Mat}_{m \times n}^{2g+r}}$. We utilise the same factorisation $\tau' = \tau'' \circ \tau$ to factorise $\psi'$ and $\varphi'$. Therefore $\varphi = \varphi_{m,n}$ and $\psi = \psi_{m,n} \times \r{id}_{M_{m,n}}$ with $\varphi_n$ and $\psi_{m,n}$ described in the proof of \Cref{prop1.08}. Then $f_0$ becomes the map that sends $(A_a)_{a \in Q'_1} \in \r{GL}_{m,n}^{2g}$ to
\begin{align}
    \Big((A_a^{(1)})_{a \in Q'_1}\,,\,\,\, (A_a^{(2)})_{a \in Q'_1}\,,\,\, \big(\tau^{\prime \prime -1}(\lambda)(A_a) - \r{Id}_{m+n} \big)^{(3)}\,,\,\, (A_a^{(3)}, p_i(A_a)^{(3)})_{a \in Q'_1,\, i=1, \ldots, r} \Big) \label{eq:5.00}
\end{align}
for $\lambda = \prod_{i=1}^g x_iy_ix_i^{-1}y_i^{-1} \in \m C \langle x_1^{\pm 1}, y_1^{\pm 1}, \ldots, x_g^{\pm 1}, y_g^{\pm 1} \rangle$.

Let $a \in Q'_1$ and $h \in Q_1 \setminus Q'_1$ index the matrices in a representation of $Q$, then $X'_{m,n} \cong X_{m,n} \times \r{GL}_{m,n}^r$ with the isomorphism given by
$$\Big((A'_a, B'_h),\, (A''_a, B''_h),\, R_0, R_1, \ldots, R_r,\, (A'''_a, B'''_h) \Big) \longmapsto \Big((A'_a),\, (A''_a),\, R_0,\, (A'''_a, B'''_h), (C_{h_i}) \Big)$$
where
$$C_{h_i} =
\begin{pmatrix}
B'_{h_i} & R_i\\
0 & B''_{h_i}
\end{pmatrix}$$
for $i=1, \ldots r$.

So as per \Cref{rmk3.1} it suffices to show that there exists some isomorphism $\gamma: M_{m,n} \xrightarrow{\sim} M_{m,n}$ such that
$$f'_0 = \Big(f_0 \times \r{id}_{\r{Mat}_{m \times n}^{2g+r}} \Big) \circ \gamma$$
where $f''_0 = \delta  \circ f'_0$ for some isomorphism $\delta: X_{m,n} \times \r{GL}_{m,n}^r \xrightarrow{\sim} X_{m,n} \times \r{GL}_{m,n}^r$.

Write
\begin{align}
    \frac{\partial W}{\partial e_0} = \prod_{i=r}^1 s_i h_i^{c_i} s_0 - \prod_{i=r}^1 t_i h_i^{d_i} t_0 \label{eq:5.01}
\end{align}
where $h_i \in Q_1 \setminus Q'_1$, $s_i,\, t_i$ are paths in $\m CQ'$ for $i=0, \ldots, r$, and $c_i,\, d_i \in \{0,1\}$ such that for each $i$ $c_i$ and $d_i$ cannot both be 1 (any $h \in Q_1 \setminus Q'_1$ can appear at most once in $\partial W/\partial e_0$, since for a brane tiling each arrow from the dual quiver must appear exactly twice in the potential and we already know that $h$ appears at least once in the relations $\partial W/\partial e_i$ for $i=1, \ldots r$). Using \Cref{lem3.412}, for each $i = 1, \ldots, r$ we can write
\begin{align}
    h_i - p_i = \sum_j l_{i,j}\, u_{i,j} \left(\frac{\partial W}{\partial e_j} \right) v_{i,j} \label{eq:5.02}
\end{align}
where $l_{i,j} \in \m C$, $p_i$ is the path in $\m C \widetilde{Q}'$ related to $h_i$, and $u_{i,j}$ and $v_{i,j}$ are paths in $\m C \widetilde{Q}$.

Write
$$\tau\left(\frac{\partial W}{\partial e_0} \right) =  \prod_{i=r}^1 s_i p_i^{c_i} s_0 - \prod_{i=r}^1 t_i p_i^{d_i} t_0 = p_0 - q_0 $$
for $p_0$ and $q_0$ paths in $\m C \widetilde{Q}'$. Then we can apply \Cref{lem3.411} to $F(Q')$ and $F(x_1,y_1, \ldots, x_g,y_g)$, and $\lambda = \prod_{i=1}^g x_iy_ix_i^{-1}y_i^{-1}$ and $\lambda' = p_0 q_0^{-1}$, giving us the unit $u \in \m C \langle x_1^{\pm 1}, y_1^{\pm 1}, \ldots, x_g^{\pm 1}, y_g^{\pm 1} \rangle$ such that
$$\tau''(\lambda) = u \, \lambda^c \, u^{-1}$$
for $c \in \{1, -1\}$. Define $u' \in \m C \widetilde{Q}'$ to be
$$u' = \tau^{\prime \prime -1}(u)$$
hence
\begin{align}
    p_0 q_0^{-1} = u'\, \tau^{\prime \prime -1}(\lambda^c)\, u^{\prime -1}. \label{eq:5.03}
\end{align}
For the moment assume that $c=1$. Then for $(A_a)_{a \in Q'_1} \in \r{GL}_{m,n}^{2g}$ we have
\begin{align}
    p_0(A_a) - q_0(A_a) &= \tau \left(\frac{\partial W}{\partial e_0} \right)(A_a) \label{eq:5.04}\\
    &= \frac{\partial W}{\partial e_0}(\varphi_{m,n}(A_a))\nonumber\\
    &= \prod_{i=r}^1 s_i(A_a)\, p_i(A_a)^{c_i}\, s_0(A_a) - \prod_{i=r}^1 t_i(A_a)\, p_i(A_a)^{d_i}\, t_0(A_a). \nonumber
\end{align}
Then define $\delta^{-1}: X_{m,n} \times \r{GL}_{m,n}^r \xrightarrow{\sim} X_{m,n} \times \r{GL}_{m,n}^r$ by sending $\Big((A'_a),\, (A''_a),\, R_0,\, (A'''_a, B'''_h), (C_{h}) \Big)_{a \in Q'_1,\, h \in Q_1 \setminus Q'_1}$ to
$$\bigg((A'_a),\, (A''_a),\, \Big(u^{\prime -1}(X_a)\, \big(R - S - T \big)\, q_0^{-1}(X_a)\, u'(X_a) \Big)^{(3)},\, (A'''_a, L'''_h),\, (D_h) \bigg)_{a \in Q'_1,\, h \in Q_1 \setminus Q'_1}$$
where for $a \in Q'_1$ and $h \in Q_1 \setminus Q'_1$
\begin{align*}
    X_a &= \begin{pmatrix}
    A'_a & A'''_a\\
    0 & A''_a
    \end{pmatrix}\\
    Y_h &= \begin{pmatrix}
     C_{h}^{(1)} & B'''_h\\
    0 &  C_{h}^{(2)}
    \end{pmatrix}\\
    F_{h_i} &= C_{h_i} - \begin{pmatrix}
    C_{h_i}^{(1)} - \partial W/\partial e_i(A'_a, C_h^{(1)}) & 0\\
    0 & C_{h_i}^{(2)} - \partial W/\partial e_i(A''_a, C_h^{(2)})
    \end{pmatrix}\\
    G_{h_i} &= \sum_j l_{i,j}\, u_{i,j}(X_a,Y_h)\, F_{h_j}\, v_{i,j}(X_a,Y_h)\\
    L'''_{h_i} &= B'''_{h_i} - G_{h_i}^{(3)}\\
    D_{h_i} &= \begin{pmatrix}
    B'_{h_i} & G_{h_i}^{(3)}\\
    0 & B''_{h_i}
    \end{pmatrix}
\end{align*}
and
\begin{align*}
    R &= R \big((A'_a), (A''_a), (C_h), R_0 \big)\\
    &= \begin{pmatrix}
    \partial W/\partial e_0(A'_a, C_h^{(1)}) & R_0\\
    0 & \partial W/\partial e_0(A''_a, C_h^{(2)})
    \end{pmatrix}\\
    S &= S \big((X_a), (Y_h), (G_{h_i}) \big)\\
    &= \sum_{i : c_i = 1} \Bigg(\bigg(\prod_{j=r}^{i+1} s_j(X_a) Y_{h_j}^{c_j} \bigg) s_i(X_a)\, G_{h_i} \bigg(\prod_{j=i-1}^1 s_j(X_a)\, p_j(X_a)^{c_j} \bigg) s_0(X_a) \Bigg)\\
    T &= T \big((X_a), (Y_h), (G_{h_i}) \big)\\
    &= \sum_{i : d_i = 1} \Bigg(\bigg(\prod_{j=r}^{i+1} t_j(X_a) Y_{h_j}^{d_j} \bigg) t_i(X_a)\, G_{h_i} \bigg(\prod_{j=i-1}^1 t_j(X_a)\, p_j(X_a)^{d_j} \bigg) t_0(X_a) \Bigg).
\end{align*}
We can therefore calculate where $f'_0 = \delta^{-1} \circ f''_0$ sends $\big((A_a, B_h)_{a \in Q'_1,\, h \in Q_1 \setminus Q'_1} \big)$:

Using the above notation and \eqref{eq:5.02} we get
\begin{align*}
    C_{h_i} &= \begin{pmatrix}
    B_{h_i}^{(1)} & (\partial W/\partial e_i(A_a,B_h))^{(3)}\\
    0 & B_{h_i}^{(2)}
    \end{pmatrix}\\
    X_a &= A_a\\
    Y_h &= B_h\\
    F_{h_i} &= C_{h_i} -
    \begin{pmatrix}
    B_{h_i}^{(1)} - \partial W/\partial e_i(A_a^{(1)},B_h^{(1)})) & 0\\
    0 & B_{h_i}^{(2)} - \partial W/\partial e_i(A_a^{(2)},B_h^{(2)}))
    \end{pmatrix}\\
    &= \frac{\partial W}{\partial e_i}(A_a, B_h)\\
    G_{h_i} &= \sum_j l_{i,j}\, u_{i,j}(A_a,B_h)\, \frac{\partial W}{\partial e_j}(A_a, B_h)\, v_{i,j}(A_a,B_h)\\
    &= B_{h_i} - p_i(A_a)\\
    L'''_{h_i} &= B_{h_i}^{(3)} - \big(B_{h_i} - p_i(A_a) \big)^{(3)}\\
    &= p_i(A_a)^{(3)}\\
    D_{h_i} &= \begin{pmatrix}
    B_{h_i}^{(1)} & \big(B_{h_i} - p_i(A_a) \big)^{(3)}\\
    0 & B_{h_i}^{(2)}
    \end{pmatrix}\\
    &= B_{h_i} - \begin{pmatrix}
    0 & p_i(A_a)^{(3)}\\
    0 & 0
    \end{pmatrix}\\
\end{align*}
and
\begin{align*}
    R &= R \big((A_a^{(1)}), (A_a^{(2)}), (C_h), (\partial W/\partial e_0(A_a,B_h))^{(3)} \big)\\
    &= \begin{pmatrix}
    \partial W/\partial e_0(A_a^{(1)}, C_h^{(1)}) & (\partial W/\partial e_0(A_a,B_h))^{(3)}\\
    0 & \partial W/\partial e_0(A_a^{(2)}, C_h^{(2)})
    \end{pmatrix}\\
    &= \begin{pmatrix}
    \partial W/\partial e_0(A_a^{(1)}, B_h^{(1)}) & (\partial W/\partial e_0(A_a,B_h))^{(3)}\\
    0 & \partial W/\partial e_0(A_a^{(2)}, B_h^{(2)})
    \end{pmatrix}\\
    &= \frac{\partial W}{\partial e_0}(A_a, B_h)\\
    S &= S \big((X_a), (Y_h), (G_h) \big) = S \big((A_a), (B_h), (B_{h_i} - p_i(A_a)) \big)\\
    &= \sum_{i : c_i = 1} \Bigg(\bigg(\prod_{j=r}^{i+1} s_j(A_a) B_{h_j}^{c_j} \bigg) s_i(A_a) \big(B_{h_i} - p_i(A_a) \big) \bigg(\prod_{j=i-1}^1 s_j(A_a)\, p_j(A_a)^{c_j} \bigg) s_0(A_a) \Bigg)\\
    &= \prod_{j=r}^1 s_j(A_a) B_{h_j}^{c_j} s_0(A_a) - \prod_{j=r}^1 s_j(A_a) p_j(A_a)^{c_j} s_0(A_a)\\
    T &= T \big((X_a), (Y_h), (G_h) \big) = T \big((A_a), (B_h), (B_{h_i} - p_i(A_a)) \big)\\
    &= \sum_{i : d_i = 1} \Bigg(\bigg(\prod_{j=r}^{i+1} t_j(A_a) B_{h_j}^{d_j} \bigg) t_i(A_a) \big(B_{h_i} - p_i(A_a) \big) \bigg(\prod_{j=i-1}^1 t_j(A_a)\, p_j(A_a)^{d_j} \bigg) t_0(A_a) \Bigg)\\
    &= \prod_{j=r}^1 t_j(A_a) B_{h_j}^{d_j} t_0(A_a) - \prod_{j=r}^1 t_j(A_a) p_j(A_a)^{d_j} t_0(A_a)
\end{align*}
since the sums in $S$ and $T$ are telescopic. Hence from \eqref{eq:5.01} and \eqref{eq:5.04} we see that
\begin{align*}
    R - S - T &= \frac{\partial W}{\partial e_0}(A_a, B_h) - \bigg[\prod_{j=r}^1 s_j(A_a) B_{h_j}^{c_j} s_0(A_a) - \prod_{j=r}^1 s_j(A_a) p_j(A_a)^{c_j} s_0(A_a) +\\
    & \quad\,\,\, \prod_{j=r}^1 t_j(A_a) B_{h_j}^{d_j} t_0(A_a) - \prod_{j=r}^1 t_j(A_a) p_j(A_a)^{d_j} t_0(A_a) \bigg]\\
    &= \prod_{i=r}^1 s_i(A_a)\, p_i(A_a)^{c_i}\, s_0(A_a) - \prod_{i=r}^1 t_i(A_a)\, p_i(A_a)^{d_i}\, t_0(A_a)\\
    &= \tau \left(\frac{\partial W}{\partial e_0} \right)(A_a)\\
    &= p_0(A_a) - q_0(A_a)
\end{align*}
and so using \eqref{eq:5.03} we get
\begin{align*}
    \Big(u^{\prime -1}(X_a)\, \big(R - S - T \big)\, q_0^{-1}(X_a)\, u'(X_a) \Big)^{(3)} &= \Big(u^{\prime -1}(A_a)\, \big(p_0(A_a) - q_0(A_a) \big)\, q_0^{-1}(A_a)\, u'(A_a) \Big)^{(3)}\\
    &= \Big(u^{\prime -1}(A_a)\, p_0(A_a)\, q_0^{-1}(A_a)\, u'(A_a) - \r{Id}_{m+n} \Big)^{(3)}\\
    &= \Big(\tau^{\prime \prime -1}(\lambda)(A_a) - \r{Id}_{m+n} \Big)^{(3)}.
\end{align*}
Therefore we have that
$f'_0\big((A_a, B_h)_{a \in Q'_1,\, h \in Q_1 \setminus Q'_1} \big)$ equals
$$\left((A_a^{(1)}),\, (A_a^{(2)}),\, \Big(\tau^{\prime \prime -1}(\lambda)(A_a) - \r{Id}_{m+n} \Big)^{(3)},\, (A_a^{(3)}, p_i(A_a)^{(3)}),\, \left(B_{h_i} - \begin{pmatrix}
0 & p_i(A_a)^{(3)}\\
0 & 0
\end{pmatrix} \right) \right)_{a \in Q'_1,\, h \in Q_1 \setminus Q'_1}$$
and so if we take $\gamma: M_{m,n} \xrightarrow{\sim} M_{m,n}$ to be the isomorphism
$$(A_a, B_h)_{a \in Q'_1,\, h \in Q_1 \setminus Q'_1} \longmapsto \left(A_a, B_{h_i} - \begin{pmatrix}
0 & p_i(A_a)^{(3)}\\
0 & 0
\end{pmatrix} \right)_{a \in Q'_1,\, h \in Q_1 \setminus Q'_1}$$
comparing with \eqref{eq:5.00} we see that
$$f'_0 = \Big(f_0 \times \r{id}_{\r{GL}_{m,n}^r} \Big) \circ \gamma$$
as required.

Lastly we must check that $\delta|_Z = \r{id}_Z$. So $\big((A'_a, B'_h),\, (A''_a, B''_h),\, (A'''_a, B'''_h) \big) \in Z$ is sent by $j'$ to $\big((A'_a, B'_h),\, (A''_a, B''_h),\, 0, \ldots, 0,\, (A'''_a, B'''_h) \big) \in X'_{m,n}$ which corresponds to $\big((A'_a),\, (A''_a),\, 0,\, (A'''_a, B'''_h),\, (C_h) \big) \in X_{m,n} \times \r{GL}_{m,n}^r$ where
$$
C_h = \begin{pmatrix}
B'_h & 0\\
0 & B''_h
\end{pmatrix}.
$$
Then, using the fact that since $(A'_a, B'_h) \in Z_m$ and $(A''_a, B''_h) \in Z_m$ the evaluation of $\partial W/\partial e_i$ on these matrices is 0 for $i=0, \ldots r$, we have
\begin{align*}
    X_a &= \begin{pmatrix}
    A'_a & A'''_a\\
    0 & A''_a
    \end{pmatrix}\\
    Y_h &= \begin{pmatrix}
     B'_h & B'''_h\\
    0 & B''_h
    \end{pmatrix}\\
    F_{h_i} &= \begin{pmatrix}
    B'_{h_i} & 0\\
    0 & B''_{h_i}
    \end{pmatrix}
    -
    \begin{pmatrix}
    B'_{h_i} - \partial W/\partial e_i(A'_a, B'_{h}) & 0\\
    0 & B''_{h_i} - \partial W/\partial e_i(A''_a, B''_{h})
    \end{pmatrix}\\
    &= \begin{pmatrix}
    B'_{h_i} & 0\\
    0 & B''_{h_i}
    \end{pmatrix}
    -
    \begin{pmatrix}
    B'_{h_i} & 0\\
    0 & B''_{h_i}
    \end{pmatrix}\\
    &=0\\
    G_{h_i} &= \sum_j l_{i,j}\, u_{i,j}(X_a,Y_h)\, F_{h_j}\, v_{i,j}(X_a,Y_h)\\
    &= 0\\
    L'''_{h_i} &= B'''_{h_i}\\
    D_{h_i} &= \begin{pmatrix}
    B'_{h_i} & 0\\
    0 & B''_{h_i}
    \end{pmatrix}\\
    &= C_{h_i}
\end{align*}
and
\begin{align*}
    R &= R \big((A'_a), (A''_a), (C_h), 0 \big)\\
    &= \begin{pmatrix}
    \partial W/\partial e_0(A'_a, C_h^{(1)}) & 0\\
    0 & \partial W/\partial e_0(A''_a, C_h^{(2)})
    \end{pmatrix}\\
    &= \begin{pmatrix}
    \partial W/\partial e_0(A'_a, B'_h) & 0\\
    0 & \partial W/\partial e_0(A''_a, B''_h)
    \end{pmatrix}\\
    &= 0\\
    S &= S \big((X_a), (Y_h), (0) \big)\\
    &= \sum_{i : c_i = 1} \Bigg(\bigg(\prod_{j=r}^{i+1} s_j(X_a) Y_{h_j}^{c_j} \bigg) s_i(X_a) \cdot 0 \cdot \bigg(\prod_{j=i-1}^1 s_j(X_a)\, p_j(X_a)^{c_j} \bigg) s_0(X_a) \Bigg)\\
    &= 0\\
    T &= T \big((X_a), (Y_h), (0) \big)\\
    &= \sum_{i : d_i = 1} \Bigg(\bigg(\prod_{j=r}^{i+1} t_j(X_a) Y_{h_j}^{d_j} \bigg) t_i(X_a) \cdot 0 \cdot \bigg(\prod_{j=i-1}^1 t_j(X_a)\, p_j(X_a)^{d_j} \bigg) t_0(X_a) \Bigg)\\
    &= 0.
\end{align*}
Hence $\big((A'_a),\, (A''_a),\, 0,\, (A'''_a, B'''_h),\, (C_h) \big)$ is sent by $\delta^{-1}$ to
$$\Big( (A'_a),\, (A''_a),\, 0,\, (A'''_a, B'''_h),\, (C_h) \Big)$$
and so we do indeed get that
$$\delta|_Z = \r{id}_Z.$$
Finally if $c=-1$ i.e. if 
$$\lambda' = u' \tau^{\prime \prime -1}(\lambda^{-1})\, u^{\prime -1}$$
then we modify $\delta$ to send $\Big((A'_a),\, (A''_a),\, R_0,\, (A'''_a, B'''_h), (C_{h}) \Big)_{a \in Q'_1,\, h \in Q_1 \setminus Q'_1}$ to
$$\bigg((A'_a),\, (A''_a),\, -\Big(\tau^{\prime \prime -1}(\lambda)(X_a)\, u^{\prime -1}(X_a)\, \big(R - S - T \big)\, q_0^{-1}(X_a)\, u'(X_a) \Big)^{(3)},\, (A'''_a, L'''_h),\, (D_h) \bigg)_{a \in Q'_1,\, h \in Q_1 \setminus Q'_1}.$$
Following the above calculations we get that for $f'_0 =  \delta^{-1} \circ f''_0$
\begin{align*}
    -\Big(\tau^{\prime \prime -1}(\lambda)(X_a)\, u^{\prime -1}&(X_a)\, \big(R - S - T \big)\, q_0^{-1}(X_a)\, u'(X_a) \Big)^{(3)} =\\
    &= -\Big(\tau^{\prime \prime -1}(\lambda)(A_a)\, u^{\prime -1}(A_a)\, \big(p_0(A_a) - q_0(A_a) \big)\, q_0^{-1}(A_a)\, u'(A_a) \Big)^{(3)}\\
    &= -\Big(\tau^{\prime \prime -1}(\lambda)(A_a)\, u^{\prime -1}(A_a)\, p_0(A_a)\, q_0^{-1}(A_a)\, u'(A_a) - \tau^{\prime \prime -1}(\lambda)(A_a) \Big)^{(3)}\\
    &= -\Big(\tau^{\prime \prime -1}(\lambda)(A_a)\,\tau^{\prime \prime -1}(\lambda^{-1})(A_a) - \tau^{\prime \prime -1}(\lambda)(A_a) \Big)^{(3)}\\
    &= \Big(\tau^{\prime \prime -1}(\lambda)(A_a) - \r{Id}_{m+n} \Big)^{(3)}
\end{align*}
as required.
\end{proof}

We can now put all the pieces together allowing us to prove \Cref{thm1.1}.

\begin{thm}\label{thm3.6}
Let $\Sigma_g$ be a Riemann surface of genus $g$ with brane tiling $\Delta$ giving dual quiver $Q_\Delta$ and potential $W_\Delta$. Fix a cut $E$ for $W_\Delta$ and a maximal tree $T$ in $Q_\Delta \setminus E$, and let $Q$ be the quiver obtained by contracting $T$ in $Q_\Delta \setminus E$ with corresponding potential $W$. Then the 2D CoHA 
$$\bigoplus_{n \in \m N} \r{H}_c(\r{Rep}_n(\pi_1(\Sigma_g)), \m Q)^\vee$$
with multiplication induced by the diagram \eqref{fig:2} is isomorphic as an algebra to the 2D CoHA
$$\bigoplus_{n \in \m N} \r{H}_c(\r{Rep}_n(\r{Jac}(\widetilde{Q},W,E)), \m Q)^\vee$$
with multiplication induced by the diagram \eqref{fig:1}.
\end{thm}

\begin{proof}
We use diagram \eqref{fig:10.5} to construct a bridge between the two multiplications. This gives us the following diagram on cohomology; where $p_m: Z_m/\r{GL}_m \rightarrow \r{pt}$, $p_n: Z_n/\r{GL}_n \rightarrow \r{pt}$, $p_{m+n}: Z_{m+n}/\r{GL}_{m+n} \rightarrow \r{pt}$ and $p: (Z_m \times Z_n)/(\r{GL}_m \times \r{GL}_n) \rightarrow \r{pt}$ are the respective structure morphisms, and we use an overline on a map to denote the respective induced map on the quotient stack.
\begin{equation}\label{fig:10.555}
\begin{tikzcd}[column sep = 10em, row sep = 7em]
\r{H}_{c, \r{GL}_{m+n}}\big(Z_{m+n}, \m Q\big) \arrow[r, "p_{m+n !} \eta^{\overline{\widetilde{\varphi}}'_{m+n}}(\m Q)", "\sim"'] \arrow[d, "q_Z^\star"] & \r{H}_{c, \r{GL}_{m+n}}\big(V_{m+n}, \m Q\big) \arrow[d, "q_V^\star"']\\
\r{H}_{c, \r{GL}_{m,n}}\big(Z_{m+n}, \m Q\big) \arrow[r, "p_{m+n !} \eta^{\overline{\widetilde{\varphi}}'_{m+n}}(\m Q)", "\sim"'] \arrow[d, "\overline{\widetilde{h}}^{\prime \star}"] & \r{H}_{c, \r{GL}_{m,n}}\big(V_{m+n}, \m Q\big) \arrow[d, "\overline{\widetilde{h}}^\star"']\\
\r{H}_{c, \r{GL}_{m,n}}\big(Z_{m,n}, \m Q\big) \arrow[r, "p_! \overline{\widetilde{f}}''_! \eta^{\overline{\widetilde{\varphi}}'_{m,n}}(\m Q)", "\sim"'] \arrow[d, "\overline{f}''_{\star}"] & \r{H}_{c, \r{GL}_{m,n}}\big(V_{m,n}, \m Q\big) \arrow[d, "\overline{f}_\star"']\\
\r{H}_{c, \r{GL}_{m,n}}\big(Z_m \times Z_n, \m Q\big) \arrow[r, "p_! \eta^{\overline{\widetilde{\varphi}}'_m \times \overline{\widetilde{\varphi}}'_n}(\m Q)", "\sim"'] \arrow[d, "\sim" sloped] & \r{H}_{c, \r{GL}_{m,n}}\big(V_m \times V_n, \m Q\big) \arrow[d, "\sim" sloped]\\
\r{H}_{c, \r{GL}_m}\big(Z_m, \m Q\big) \otimes \r{H}_{c, \r{GL}_n}\big(Z_n, \m Q\big) \arrow[r, "p_{m !} \eta^{\overline{\widetilde{\varphi}}'_m}(\m Q)\, \otimes\, p_{n !} \eta^{\overline{\widetilde{\varphi}}'_n}(\m Q)", "\sim"'] & \r{H}_{c, \r{GL}_m}\big(V_m, \m Q\big) \otimes \r{H}_{c, \r{GL}_n}\big(V_n, \m Q\big)
\end{tikzcd}
\end{equation}
Recalling that $\r{Rep}_n(\r{Jac}(\widetilde{Q},W,E)) \cong Z_n/\r{GL}_n$ and $\r{Rep}_n(\m C[\pi_1(\Sigma_g)]) \cong V_n/\r{GL}_n$, the dual of the left-hand side of \eqref{fig:10.555} exactly gives the multiplication on $\r{H}_c(\r{Rep}_n(\r{Jac}(\widetilde{Q},W,E)), \m Q)$ and the dual of the right-hand side gives the multiplication on $\r{H}_c(\r{Rep}_n(\pi_1(\Sigma_g)), \m Q)$, as described in Section 3.1. As $\overline{\widetilde{\varphi}'_n}$ is an isomorphism each of the horizontal arrows in \eqref{fig:10.555} are isomorphisms, hence to prove the two 2D CoHAs are isomorphic as algebras it suffices to show that \eqref{fig:10.555} commutes.

Commutativity of the bottom square follows from the naturality of isomorphisms involved. For the top square note we have the following commutative diagram of maps
\begin{equation}\label{fig:10.556}
\begin{tikzcd}[column sep = 7em, row sep = 7em]
V_{m+n}/\r{GL}_{m,n} \arrow[d, "q_V"] \arrow[r, "\overline{\widetilde{\varphi}}'_{m+n}", "\sim"'] & Z_{m+n}/\r{GL}_{m,n} \arrow[d, "q_Z"]\\
V_{m+n}/\r{GL}_{m+n} \arrow[r, "\overline{\widetilde{\varphi}}'_{m+n}", "\sim"'] & Z_{m+n}/\r{GL}_{m+n}
\end{tikzcd}
\end{equation}
Now for a general commutative square 
\[
\begin{tikzcd}[column sep = 6em, row sep = 5em]
X \arrow[d, "f'"] \arrow[r, "f"] & Y \arrow[d, "g"]\\
W \arrow[r, "g'"] & Z
\end{tikzcd}
\]
we have that the composition
$$\m Q_Z \xrightarrow{\eta^g(\m Q)} g_* g^* \m Q_Z \xrightarrow{g_*(\eta^f(g^* \m Q))} g_* f_* f^* g^* \m Q_Z$$
equals
$$\m Q_Z \xrightarrow{\eta^{g \circ f}(\m Q)} (g \circ f)_* (g \circ f)^* \m Q_Z = g_* f_* f^* g^* \m Q_Z$$
which equals
$$\m Q_Z \xrightarrow{\eta^{g' \circ f'}(\m Q)} (g' \circ f')_* (g' \circ f')^* \m Q_Z = g'_* f'_* f^{\prime *} g^{\prime *} \m Q_Z$$
which equals
$$\m Q_Z \xrightarrow{\eta^{g'}(\m Q)} g'_* g^{\prime *} \m Q_Z \xrightarrow{g'_*(\eta^{f'}(g^{\prime *} \m Q))} g'_* f'_* f^{\prime *} g^{\prime *} \m Q_Z$$
i.e.
\begin{align}
    g_*(\eta^f(g^* \m Q_Z)) \circ \eta^g(\m Q_Z) = g'_*(\eta^{f'}(g^{\prime *} \m Q_Z)) \circ \eta^{g'}(\m Q_Z). \label{eq:5.110}
\end{align}
Applying this to \eqref{fig:10.556} and recalling the definitions of the morphisms from Section 3.1
\begin{align*}
    q_V^\star &:= p_{m+n !} \overline{\widetilde{\varphi}}'_{m+n !}(\eta^{q_V}(\m Q_{V_{m+n}/\r{GL}_{m+n}}))\\
    q_Z^\star &:= p_{m+n !}(\eta^{q_Z}(\m Q_{Z_{m+n}/\r{GL}_{m+n}}))
\end{align*}
it follows that
\begin{align*}
    q_V^\star \circ p_{m+n !}(\eta^{\overline{\widetilde{\varphi}}'_{m+n}}(\m Q)) &= p_{m+n !} \overline{\widetilde{\varphi}}'_{m+n !}(\eta^{q_V}(\m Q) \circ p_{m+n !}(\eta^{\overline{\widetilde{\varphi}}'_{m+n}}(\m Q))\\
    &= p_{m+n !} \overline{\widetilde{\varphi}}'_{m+n !}(\eta^{q_V}(\overline{\widetilde{\varphi}}_{m+n}^{\prime *}\m Q) \circ p_{m+n !}(\eta^{\overline{\widetilde{\varphi}}'_{m+n}}(\m Q))\\
    &= p_{m+n !}(\eta^{\overline{\widetilde{\varphi}}'_{m+n} \circ q_V}(\m Q))\\
    &= p_{m+n !}(\eta^{q_Z \circ \overline{\widetilde{\varphi}}'_{m+n}}(\m Q))\\
    &= p_{m+n !} q_{Z !}(\eta^{\overline{\widetilde{\varphi}}'_{m+n}}(q_Z^* \m Q)) \circ p_{m+n !}(\eta^{q_Z}(\m Q))\\
    &=  p_{m+n !} q_{Z !}(\eta^{\overline{\widetilde{\varphi}}'_{m+n}}(\m Q)) \circ q_Z^\star.
\end{align*}
The middle two squares in diagram \eqref{fig:10.555} are induced by the respective morphisms of sheaves in $\r{D}(Z_{m+n})$, namely
\begin{equation}\label{fig:10.557}
\begin{tikzcd}[column sep = 6.5em, row sep = 8em]
\m Q_{Z_{m+n}} \arrow[r, "\eta^{\widetilde{\varphi}'_{m+n}}(\m Q)", "\sim"'] \arrow[d, "\widetilde{h}^{\prime \star}"] & \widetilde{\varphi}'_{m+n !} \m Q_{V_{m+n}} \arrow[d, "\widetilde{\varphi}'_{m+n !}(\widetilde{h}^\star)"]\\
\widetilde{h}'_! \m Q_{Z_{m,n}} \arrow[r, "\widetilde{h}'_! \eta^{\widetilde{\varphi}'_{m,n}}(\m Q)", "\sim"'] & \widetilde{h}'_! \widetilde{\varphi}'_{m,n !} \m Q_{V_{m,n}} = \widetilde{\varphi}'_{m+n !}  \widetilde{h}_!  \m Q_{V_{m,n}}
\end{tikzcd}
\end{equation}
and
\begin{equation}\label{fig:10.558}
\begin{tikzcd}[column sep = 7em, row sep = 9em]
\widetilde{f}''_! \m Q_{Z_{m,n}}[2d] \arrow[r, "\widetilde{f}''_! \eta^{\widetilde{\varphi}'_{m,n}}(\m Q{[2d]})", "\sim"'] \arrow[d, "f''_{\star}"] & \widetilde{f}''_! \widetilde{\varphi}'_{m,n !} \m Q_{V_{m,n}} = (\widetilde{\varphi}'_m \times \widetilde{\varphi}'_n)_! \widetilde{f}_!  \m Q_{V_{m,n}}[2d] \arrow[d, "(\widetilde{\varphi}'_m \times \widetilde{\varphi}'_n)_!(f_\star)"]\\
\m Q_{Z_m \times Z_n} \arrow[r, "\eta^{\widetilde{\varphi}'_m \times \widetilde{\varphi}'_n}(\m Q)", "\sim"'] & (\widetilde{\varphi}'_m \times \widetilde{\varphi}'_n)_! \m Q_{V_m \times V_n}
\end{tikzcd}
\end{equation}
hence showing \eqref{fig:10.557} and \eqref{fig:10.558} commute is enough. Commutativity of \eqref{fig:10.557} follows from the fact that the bottom-left square in diagram \eqref{fig:10.5} is commutative. Indeed recall the definition of the morphisms $\widetilde{h}^\star$ and $\widetilde{h}^{\prime \star}$ from \eqref{eq:1.0000045}
\begin{align*}
    \widetilde{h}^\star &= \eta^{\widetilde{h}}(\m Q_{V_{m+n}})\\
    \widetilde{h}^{\prime \star} &= \eta^{\widetilde{h}'}(\m Q_{Z_{m+n}})
\end{align*}
then applying \eqref{eq:5.110} we get that
\begin{align*}
    \widetilde{\varphi}'_{m+n !}(\widetilde{h}^\star) \circ \eta^{\widetilde{\varphi}'_{m+n}}(\m Q) &= \widetilde{\varphi}'_{m+n !}(\eta^{\widetilde{h}}(\m Q)) \circ \eta^{\widetilde{\varphi}'_{m+n}}(\m Q)\\
    &= \widetilde{\varphi}'_{m+n !}(\eta^{\widetilde{h}}(\widetilde{\varphi}_{m+n}^{\prime *}\m Q)) \circ \eta^{\widetilde{\varphi}'_{m+n}}(\m Q)\\
    &= \eta^{\widetilde{\varphi}'_{m+n} \circ \widetilde{h}}(\m Q)\\
    &= \eta^{\widetilde{h}' \circ \widetilde{\varphi}'_{m,n}}(\m Q)\\
    &= \widetilde{h}'_!(\eta^{\widetilde{\varphi}'_{m,n}}(\widetilde{h}'^* \m Q)) \circ \eta^{\widetilde{h}'}(\m Q)\\
    &= \widetilde{h}'_!(\eta^{\widetilde{\varphi}'_{m,n}}(\m Q)) \circ \widetilde{h}^{\prime \star}.
\end{align*}
It remains to show the commutativity of \eqref{fig:10.558}. Again recall the definition of the morphisms $f_\star$ and $f''_\star$ from \eqref{eq:1.000005}
\begin{align*}
    f_\star &= i^*(\nu^{f}(\m Q_{Y_{m,n}})) \circ \epsilon^{i,f}(f^!\m Q_{Y_{m,n}})^{-1}\\
    f''_\star &= i^{\prime *}(\nu^{f''}(\m Q_{Y'_{m,n}})) \circ \epsilon^{i',f''}(f^{\prime \prime !}\m Q_{Y'_{m,n}})^{-1}.
\end{align*}
Since $f= \pi_1 \circ f_0$ and $f''= \pi'_1 \circ f''_0$ we have that $f_\star = \pi_{1 \star} \circ \widetilde{\pi}_{1 !}(f_{0 \star})$ and $f''_\star= \pi'_{1 \star} \circ \widetilde{\pi}'_{1 !}(f''_{0 \star})$ from \Cref{lem3.11}. Hence we can use diagram \eqref{fig:10.6} to split \eqref{fig:10.558} into the following diagram
\begin{equation}\label{fig:1.02}
\begin{tikzcd}[column sep = 9em, row sep = 11em]
\widetilde{\pi}'_{1 !} \widetilde{f}''_{0 !} \m Q_{Z_{m,n}}[2d] \arrow[r, "\widetilde{\pi}'_{1 !} \widetilde{f}''_{0 !} \eta^{\widetilde{\varphi}'_{m,n}}(\m Q{[2d]})", "\sim"'] \arrow[d, "\widetilde{\pi}'_{1 !}(f''_{0 \star}{[2e]})"] & \widetilde{\pi}'_{1 !} \widetilde{f}''_{0 !} \widetilde{\varphi}'_{m,n !} \m Q_{V_{m,n}} = \widetilde{\psi}'_{m,n !} \widetilde{\pi}_{1 !} \widetilde{f}_{0 !}  \m Q_{V_{m,n}}[2d] \arrow[d, "\widetilde{\psi}'_{m,n !} \widetilde{\pi}_{1 !}(f_{0 \star}{[2e]})"]\\
\widetilde{\pi}'_{1 !} \m Q_{Z}[2e] \arrow[r, "\widetilde{\pi}'_{1 !} \eta^{\widetilde{\psi}'_{m,n} \times \r{id}}(\m Q{[2e]})", "\sim"'] \arrow[d, "\pi'_{1 \star}"] & \widetilde{\pi}'_{1 !} (\widetilde{\psi}'_{m,n} \times \r{id})_! \m Q_{V}[2e] = \widetilde{\psi}'_{m,n !} \widetilde{\pi}_{1 !} \m Q_{V}[2e] \arrow[d, "\widetilde{\psi}'_{m,n !}(\pi_{1 \star})"]\\
\m Q_{Z_m \times Z_n} \arrow[r, "\eta^{\widetilde{\psi}'_{m,n}}(\m Q)", "\sim"'] & \widetilde{\psi}'_{m,n !} \m Q_{V_m \times V_n}
\end{tikzcd}
\end{equation}
where $\r{dim}(\pi_1) = \r{dim}(\pi'_1) = e$ and $\r{dim}(f_0) = \r{dim}(f''_0) = d-e$, and recall $\widetilde{\psi}'_{m,n} =  \widetilde{\varphi}'_m \times \widetilde{\varphi}'_n$. Then \Cref{prop1.08} and \Cref{prop1.09} tell us that we can apply \Cref{cor1.06} to both cubes in diagram \eqref{fig:10.6}, which in turn tells us that the both squares in \eqref{fig:1.02} commute and hence \eqref{fig:10.558} commutes and therefore so does \eqref{fig:10.555}.
\end{proof}

\newpage
\bibliographystyle{plain}
\bibliography{references}

\end{document}